\documentclass[11pt,reqno]{amsart}
\usepackage{amssymb,amscd,amsfonts,amsbsy,enumerate}
\usepackage{latexsym}
\usepackage{exscale}
\usepackage{amsmath,amsthm,amsfonts}
\usepackage{mathrsfs}
\usepackage{epsf,epsfig}

\usepackage[colorlinks=true,linkcolor=blue,citecolor=red,urlcolor=red]{hyperref} % Add coloured links

\usepackage{esint} % For \fint symbol

\newcommand{\meanint}{{\int{\mkern-19mu}-}}

\usepackage{color}

\numberwithin{equation}{section}

\def\Xint#1{\mathchoice
   {\XXint\displaystyle\textstyle{#1}}%
   {\XXint\textstyle\scriptstyle{#1}}%
   {\XXint\scriptstyle\scriptscriptstyle{#1}}%
   {\XXint\scriptscriptstyle\scriptscriptstyle{#1}}%
   \!\int}
\def\XXint#1#2#3{{\setbox0=\hbox{$#1{#2#3}{\int}$}
     \vcenter{\hbox{$#2#3$}}\kern-.5\wd0}}
\def\aver#1{\Xint-_{#1}}

\allowdisplaybreaks

\newtheorem{theorem}{Theorem}[section]
\newtheorem{lemma}[theorem]{Lemma}
\newtheorem{corollary}[theorem]{Corollary}
\newtheorem{proposition}[theorem]{Proposition}
\newtheorem{definition}[theorem]{Definition}

\theoremstyle{remark}
\newtheorem{remark}[theorem]{Remark}

\begin{document}
\allowdisplaybreaks

\title[Green Function]
{The Green Function for Elliptic Systems in the Upper-Half Space}

\author{Martin Dindo\v{s}}
\address{Martin Dindo\v{s}
\\
School of Mathematics, 
\\
The University of Edinburgh and 
\\
Maxwell Institute of Mathematical Sciences, UK} \email{M.Dindos@ed.ac.uk}

\author{Dorina Mitrea}
\address{Dorina Mitrea
\\
Department of Mathematics
\\
Baylor University
\\
Sid Richardson Bldg., 1410 S.~4th Street
\\
Waco, TX 76706, USA} \email{Dorina\_\,Mitrea@baylor.edu}

\author{Irina Mitrea}
\address{Irina Mitrea
\\
Department of Mathematics
\\
Temple University\!
\\
1805\,N.\,Broad\,Street
\\
Philadelphia, PA 19122, USA} \email{imitrea@temple.edu}

\author{Marius Mitrea\\ }
\address{Marius Mitrea
\\
Department of Mathematics
\\
Baylor University
\\
Sid Richardson Bldg., 1410 S.~4th Street
\\
Waco, TX 76706, USA} \email{Marius\_\,Mitrea@baylor.edu}

\thanks{The first author has been supported in part by EPSRC grant EP/Y033078/1.
The second author has been supported in part by Simons Foundation grant $\#\,$958374. The third author 
has been supported in part by Simons Foundation grant $\#\,$00002820.
The fourth author has been supported in part by the Simons Foundation grant $\#\,$637481. 
This work has been possible thanks to the support and hospitality
of \textit{Temple University} (USA) and \textit{Baylor University} (USA).
The authors express their gratitude to these institutions.}

%\date{October 11, 2024} %\textit{Revised}: \today}



\begin{abstract}
Let $L$ be a second-order, homogeneous, constant (complex) coefficient elliptic system in 
${\mathbb{R}}^n$. The goal of this article is provide a qualitative and quantitative study 
of the nature of the Green function associated with the system $L$ in the upper-half space. 
Starting with a definition of the Green function which brings forth the minimal features
which identify this object uniquely, we establish optimal nontangential maximal function estimates
and regularity results up to the boundary for the said Green function. The main tools employed in the 
proof include the Agmon-Douglis-Nirenberg construction of a Poisson kernel for the system $L$, 
the Agmon-Douglis-Nirenberg a priori regularity estimates near the boundary, and the brand of Divergence Theorem 
from \cite{GHA.I} in which the boundary trace of the corresponding vector field is taken in 
nontangential pointwise sense.
\vskip -0.40in \null
\end{abstract}

\maketitle

\allowdisplaybreaks

\section{Introduction}
\setcounter{equation}{0}
\label{S-1}

Fix a space dimension $n\in{\mathbb{N}}$ and, unless otherwise stated, assume $n\geq 2$.
Pick an integer $M\in{\mathbb{N}}$ and consider the second-order, homogeneous, $M\times M$ system,
with constant complex coefficients, written (with the usual convention of summation over
repeated indices in place) as
\begin{equation}\label{L-def}
L:=\Big(a^{\alpha\beta}_{rs}\partial_r\partial_s\Big)_{1\leq\alpha,\beta\leq M},
\end{equation}
naturally acting on vector-valued (or matrix-valued) distributions defined in (open subsets of) ${\mathbb{R}}^n$. 
Examples to keep in mind are the Laplacian and the Lam\'e system. Typically, two types of ellipticity conditions 
are considered. First, we shall say that $L$ expressed as in \eqref{L-def} is {\tt strongly} {\tt elliptic} provided 
there exists a real number $c>0$ such that the following Legendre-Hadamard condition is satisfied:
\begin{equation}\label{L-ell.X}
\begin{array}{c}
{\rm Re}\,\big[a^{\alpha\beta}_{rs}\xi_r\xi_s\overline{\eta_\alpha}
\eta_\beta\,\big]\geq c|\xi|^2|\eta|^2\,\,\text{ for every}
\\[8pt]
\xi=(\xi_r)_{1\leq r\leq n}\in{\mathbb{R}}^n\,\,\text{ and }\,\,
\eta=(\eta_\alpha)_{1\leq\alpha\leq M}\in{\mathbb{C}}^M.
\end{array}
\end{equation}
Second, we shall say that $L$ is {\tt weakly} {\tt elliptic} if  
\begin{equation}\label{Def-ELLa.INTR}
{\rm det}\Big[\big(a^{\alpha\beta}_{rs}\xi_r\xi_s\big)_{1\leq\alpha,\beta\leq M}\Big]\not=0
\,\,\text{ for every }\,\,\xi=(\xi_r)_{1\leq r\leq n}\in{\mathbb{R}}^n\setminus\{0\}.
\end{equation}
Obviously, any strongly elliptic system is weakly elliptic. For example, considering the one-parameter family of 
scalar differential operators in ${\mathbb{R}}^n$ given by 
\begin{equation}\label{eq:LLLambda}
L_\lambda:=\partial_1^2+\cdots+\partial_{n-1}^2+\lambda\partial_n^2,\qquad\lambda\in{\mathbb{C}},
\end{equation}
then $L_\lambda$ satisfies the Legendre-Hadamard ellipticity condition \eqref{L-ell.X} if and only if ${\rm Re}\,\lambda>0$, 
whereas $L_\lambda$ is weakly elliptic if and only if $\lambda\in{\mathbb{C}}\setminus(-\infty,0]$. For more on this and related 
topics the reader is referred to \cite{GHA.III}.

As is known from the classical work of S.\,Agmon, A.\,Douglis, and L.\,Nirenberg in \cite{ADNI}-\cite{ADNII} 
(cf. also \cite{Lop}, \cite{Shap}, \cite{Sol1}, \cite{Sol2}, \cite[\S{10.3}]{KMR2}, and the discussion in \cite[p.\,24]{KM2012}),
every operator $L$ as in \eqref{L-def}-\eqref{L-ell.X} has a Poisson kernel, denoted by $P^L$ 
(an object whose properties mirror the most basic characteristics of the classical harmonic Poisson kernel). 
For details, see Theorem~\ref{ya-T4-fav} below. 

Our main result is Theorem~\ref{ta.av-GGG.2A} elaborating on the nature of the Green function associated 
with a given elliptic system. Prior to formulating this result, some comments on the notation used
are in order. Throughout, ${\mathbb{R}}^n_{+}:=\{(x',x_n)\in{\mathbb{R}}^{n-1}\times{\mathbb{R}}:\,x_n>0\}$ will 
denote the upper-half space. Given a set $K\subseteq{\mathbb{R}}^n_{+}$ along with a function $u$ defined in 
${\mathbb{R}}^n_{+}\setminus K$, by ${\mathcal{N}}^{\,{\mathbb{R}}^n_{+}\setminus K}_\kappa u$ we shall 
denote the nontangential maximal function of $u$ with aperture $\kappa$; see \eqref{NT-Fct.23} for a precise 
definition. Next, by $u\big|^{{}^{\kappa-{\rm n.t.}}}_{\partial{\mathbb{R}}^n_{+}}$ we denote the 
($\kappa$-)nontangential limit of the given function $u$ on the boundary of the upper half-space 
(canonically identified with ${\mathbb{R}}^{n-1}$), as defined in \eqref{nkc-EE-2}. Given any 
$n\in{\mathbb{N}}$, the Lebesgue measure in ${\mathbb{R}}^n$ will be denoted by ${\mathscr{L}}^n$. 
Also, we shall denote by $\mathcal{D}'({\mathbb{R}}^n_{+})$ the space of distributions in ${\mathbb{R}}^n_{+}$.
Finally, we agree to abbreviate ${\rm diag}:=\{(x,x):\,x\in{\mathbb{R}}^n_{+}\}$ for the diagonal in 
the Cartesian product ${\mathbb{R}}^n_{+}\times{\mathbb{R}}^n_{+}$.

To set the stage, we make the following definition.

\begin{definition}\label{ta.av-GGG}
Fix $n,M\in{\mathbb{N}}$ with $n\geq 2$.
Let $L$ be an $M\times M$ system with constant complex coefficients as in \eqref{L-def} and \eqref{Def-ELLa.INTR}. 
Call $G^L(\cdot,\cdot):{\mathbb{R}}^n_{+}\times{\mathbb{R}}^n_{+}\setminus{\rm diag}\to{\mathbb{C}}^{M\times M}$ 
a {\tt Green} {\tt function} for $L$ in ${\mathbb{R}}^n_{+}$ provided for each $y=(y',y_n)\in\mathbb{R}^n_{+}$ 
the following properties hold {\rm (}for some aperture parameter $\kappa>0${\rm )}:
\begin{align}\label{GHCewd-22.RRe}
& G^L(\cdot\,,y)\in\big[L^1_{\rm loc}({\mathbb{R}}^n_{+})\big]^{M\times M},
\\[4pt]
& G^L(\cdot\,,y)\big|^{{}^{\kappa-{\rm n.t.}}}_{\partial{\mathbb{R}}^n_{+}}=0
\,\,\text{ at ${\mathscr{L}}^{n-1}$-a.e. point in }\,\,{\mathbb{R}}^{n-1}\equiv\partial{\mathbb{R}}^n_{+},
\label{GHCewd-24.RRe}
\\[4pt]
& \int_{{\mathbb{R}}^{n-1}}\Big(
{\mathcal{N}}^{\,{\mathbb{R}}^n_{+}\setminus\overline{B(y,y_n/2)}}_{\kappa}\,G^L(\cdot\,,y)\Big)(x')\frac{dx'}{1+|x'|^{n-1}}<+\infty,
\label{GHCewd-25.RRe}
\\[4pt]
& L\big[G^L(\cdot\,,y)\big]=\delta_{y}\,I_{M\times M}\,\,\text{ in }\,\,\big[{\mathcal{D}}'({\mathbb{R}}^n_{+})\big]^{M\times M},
\label{GHCewd-23.RRe}
\end{align}
where the $M\times M$ system $L$ acts in the ``dot" variable on the columns of $G$, 
and $\delta_y$ denotes Dirac's distribution with mass at $y$.  
\end{definition}

A few comments pertaining to the nature of Definition~\ref{ta.av-GGG} are appropriate. 
We remark that \eqref{GHCewd-22.RRe} and \eqref{GHCewd-23.RRe} simply amount to saying that the 
matrix-valued function $G(\cdot,y)$ is a fundamental solution for the system $L$ the upper-half space 
which is adapted to this domain, in the sense that it has a vanishing boundary trace 
(considered in nontangential sense, in the ``dot" variable; cf. \eqref{GHCewd-24.RRe}).
As a byproduct, elliptic regularity guarantees that 
\begin{equation}\label{GHCvCbN.2T}
G(\cdot\,,y)\in\big[{\mathscr{C}}^\infty({\mathbb{R}}^n_{+}\setminus\{y\})\big]^{M\times M}
\,\,\text{ for each }\,\,y\in\mathbb{R}^n_{+}.
\end{equation}
This being said, $G(\cdot\,,y)$ has a singularity at the pole $y$ which, in particular, 
explains why the nontangential maximal function in \eqref{GHCewd-25.RRe} is restricted over 
${\mathbb{R}}^n_{+}\setminus\overline{B(y,y_n/2)}$, a region which avoids the singular point $y$. 
Of course, instead of this region we could have considered (with the same effect) 
${\mathbb{R}}^n_{+}\setminus K$ for any compact subset $K$ of ${\mathbb{R}}^n_{+}$ whose interior, 
denoted by $\mathring{K}$, contains the pole $y$. 

Since for each $\varepsilon>0$ there exists $C=C(n,\kappa,\varepsilon,y)\in(0,\infty)$ such that (cf. \eqref{UUNNpkjgr}) 
\begin{equation}\label{GHCvCbN.2T.D+M}
\Big({\mathcal{N}}^{\,{\mathbb{R}}^n_{+}\setminus\overline{B(y,y_n/2)}}_{\kappa}(1+|\cdot|)^{-\varepsilon}\Big)(x')
\leq C(1+|x'|)^{-\varepsilon}\,\,\text{ for all }\,\,x'\in{\mathbb{R}}^{n-1},
\end{equation}
it follows that the finiteness condition in \eqref{GHCewd-25.RRe} is automatically satisfied if, 
for each fixed pole $y\in{\mathbb{R}}^n_{+}$, 
\begin{equation}\label{GHCvCbN.2T.D+M.2}
\parbox{8.70cm}{$G^L(\cdot\,,y)$ is bounded away from the pole and there exists $\varepsilon>0$ so that
$G^L(x,y)=O\big(|x|^{-\varepsilon}\big)$ as $|x|\to\infty$.}
\end{equation}

Properties \eqref{GHCewd-22.RRe}-\eqref{GHCewd-24.RRe} and \eqref{GHCewd-23.RRe} are staple features of any notion 
of Green function, but these alone do not identify said Green function uniquely. For example, adding any 
constant multiple of $x_n$ to such a Green function does not affect properties \eqref{GHCewd-22.RRe}-\eqref{GHCewd-24.RRe} 
and \eqref{GHCewd-23.RRe}. We have imposed the finiteness condition in \eqref{GHCewd-25.RRe} in order to (hopefully!) ensure uniqueness.
That the conditions stipulated in Definition~\ref{ta.av-GGG} indeed identify the Green function uniquely
makes the object of our main result, stated in Theorem~\ref{ta.av-GGG.2A} below. In addition to existence 
and uniqueness, this elaborates on a wealth of other basic properties of our brand of Green function. 
Among other things, it is shown that our integral finiteness condition formulated in \eqref{GHCewd-25.RRe} 
self-improves to the pointwise properties stated in \eqref{GHCvCbN.2T.D+M.2}, at least if $n\geq 3$; see \eqref{mainest3G.LLL}. 

Before stating the aforementioned theorem, we make two conventions regarding notation. First, given a generic function $G(\cdot,\cdot)$ 
of two vector variables, $(x,y)\in{\mathbb{R}}^n_{+}\times{\mathbb{R}}^n_{+}\setminus{\rm diag}$, for each $k\in\{1,\dots,n\}$ 
we agree to write $\partial_{X_k} G$ and $\partial_{Y_k} G$, respectively, for the partial derivative 
of $G$ with respect to $x_k$, and $y_k$. This convention may be iterated, lending a natural meaning to 
$\partial^\alpha_X\partial^\beta_Y G$, for each pair of multi-indices $\alpha,\beta\in{\mathbb{N}}_0^n$. 
Second, we shall interpret $\nabla_X G$ and $\nabla_Y G$ as the gradients of $G(x,y)$ with respect to 
$x$ and $y$, respectively. Lastly, as usual, we set $\log_{+}t:=\max\big\{0\,,\,\ln t\big\}$ for each $t\in(0,\infty)$.

\begin{theorem}\label{ta.av-GGG.2A}
Fix $n,M\in{\mathbb{N}}$ with $n\geq 2$. Assume $L$ is an $M\times M$ system with constant complex coefficients 
as in \eqref{L-def}-\eqref{L-ell.X}. Then there exists a unique Green function $G^L(\cdot,\cdot)$ for $L$ in 
$\mathbb{R}^n_{+}$, in the sense of Definition~\ref{ta.av-GGG}. Moreover, this Green function also satisfies 
the following additional properties:

\begin{enumerate}
\item[(1)] Given $\kappa>0$, for each $y\in{\mathbb{R}}^n_{+}$ and each compact subset 
$K$ of ${\mathbb{R}}^n_{+}$ with $y\in\mathring{K}$ there exists a finite constant $C_{y,K}>0$ 
such that for every $x'\in\mathbb{R}^{n-1}$ there holds
\begin{equation}\label{bound-NK-G}
\Big(\mathcal{N}^{^{\,{\mathbb{R}}^n_{+}\setminus K}}_\kappa G^L(\cdot,y)\Big)(x')
\leq C_{y,K}\Big(\frac{1+\log_{+}|x'|}{1+|x'|^{n-1}}\Big).
\end{equation}
Moreover, for any multi-indices $\alpha,\beta\in\mathbb{N}_0^n$ such that
$|\alpha|+|\beta|>0$, there exists some constant $C_{y,K}\in(0,\infty)$ such that
\begin{equation}\label{bouMNN}
\Big(\mathcal{N}^{\,{\mathbb{R}}^n_{+}\setminus K}_\kappa
(\partial^\alpha_X \partial^\beta_Y G^L)(\cdot,y)\Big)(x')\leq\frac{C_{y,K}}{1+|x'|^{n-2+|\alpha|+|\beta|}}.
\end{equation}
In particular,
\begin{equation}\label{GHCewd-22.RRe.4}
{\mathcal{N}}_\kappa^{\,{\mathbb{R}}^n_{+}\setminus K}(\partial_X^\alpha\partial_Y^\beta G^L)(\cdot\,,y)
\in\bigcap_{1<p\leq\infty}L^p({\mathbb{R}}^{n-1}),\qquad\forall\,\alpha,\beta\in{\mathbb{N}}_0^n.
\end{equation}
\item[(2)] For each fixed $y\in{\mathbb{R}}^n_{+}$ there holds
\begin{equation}\label{Fvabbb-7tF}
G^L(\cdot\,,y)\in\big[{\mathscr{C}}^\infty\big(\overline{{\mathbb{R}}^n_{+}}\setminus B(y,\varepsilon)\big)
\big]^{M\times M}\,\,\text{ for every }\,\,\varepsilon>0.
\end{equation}
\item[(3)] For each $\alpha,\beta\in\mathbb{N}_0^n$ the function 
$\partial^\alpha_X\partial^\beta_Y G^L$ is translation invariant 
in the tangential variables, in the sense that
\begin{equation}\label{JKBvc-ut4}
\begin{array}{c}
\big(\partial^\alpha_X\partial^\beta_Y G^L\big)\big(x-(z',0),y-(z',0)\big)
=\big(\partial^\alpha_X\partial^\beta_Y G^L\big)(x,y)
\\[8pt]
\text{for each }\,\,(x,y)\in{\mathbb{R}}^n_{+}\times{\mathbb{R}}^n_{+}
\setminus{\rm diag}\,\,\text{ and }\,\,z'\in{\mathbb{R}}^{n-1},
\end{array}
\end{equation}
and is positive homogeneous, in the sense that
\begin{equation}\label{mainest3G.LLL}
\begin{array}{c}
\big(\partial^\alpha_X\partial^\beta_Y G^L\big)(\lambda x,\lambda y)
=\lambda^{2-n-|\alpha|-|\beta|}\big(\partial^\alpha_X\partial^\beta_Y G^L\big)(x,y)
\\[6pt]
\text{for each }\,\,x,y\in{\mathbb{R}}^n_{+}\,\,\text{ with }\,\,x\not=y
\,\,\text{ and }\,\,\lambda\in(0,\infty),
\\[6pt]
\text{provided either $n\geq 3$, or $|\alpha|+|\beta|>0$}.
\end{array}
\end{equation}
\item[(4)] With $W^{k,p}$ denoting the $L^p$-based Sobolev space or order $k$ 
and with ${\rm Tr}$ denoting the Sobolev trace on $\partial{\mathbb{R}}^n_{+}$
{\rm (}cf. \eqref{Veri-S2TG.3}-\eqref{Veri-S2TG.5}{\rm )}, one has
\begin{equation}\label{Aivb-gVV}
\begin{array}{c}
G^L(\cdot\,,y)\in\bigcap\limits_{k\in{\mathbb{N}}}
\bigcap\limits_{\frac{n}{n-1}<p<\infty}\big[W^{k,p}({\mathbb{R}}^n_{+}\setminus K)\big]^{M\times M}
\,\,\text{ and }\,\,{\rm Tr}\,\big[G^L(\cdot\,,y)\big]=0,
\\[10pt]
\text{for every $y\in{\mathbb{R}}^n_{+}$ and every compact 
$K\subset{\mathbb{R}}^n_{+}$ with $y\in\mathring{K}$}.
\end{array}
\end{equation}
\item[(5)] If $G^{L^\top}\!\!(\cdot,\cdot)$ denotes the {\rm (}unique, by the first claim in the statement{\rm )} 
Green function for $L^\top$ in ${\mathbb{R}}^n_{+}$ then 
\begin{equation}\label{GHCewd-22.RRe.5}
G^L(x,y)=\Big[G^{L^\top}\!\!(y,x)\Big]^\top,\qquad\forall\,(x,y)\in
{\mathbb{R}}^n_{+}\times{\mathbb{R}}^n_{+}\setminus{\rm diag}.
\end{equation}
Hence, as a consequence of \eqref{GHCewd-22.RRe.5}, \eqref{GHCewd-24.RRe}, and
\eqref{Fvabbb-7tF}, for each fixed $x\in{\mathbb{R}}^n_{+}$ and $\varepsilon>0$, 
\begin{equation}\label{ghagUGDS}
G^L(x,\cdot)\in\big[{\mathscr{C}}^\infty\big(\overline{{\mathbb{R}}^n_{+}}\setminus B(x,\varepsilon)\big)\big]^{M\times M}
\,\,\text{ and }\,\,G^L(x,\cdot)\Big|_{\partial{\mathbb{R}}^n_{+}}=0
\,\,\text{ on }\,\,{\mathbb{R}}^{n-1}.
\end{equation}
\item[(6)] If $E^L$ denotes the fundamental solution of $L$ 
from Theorem~\ref{FS-prop}, then the matrix-valued function
\begin{equation}\label{defRRR}
R_L(x,y):=E^L(x-y)-G^L(x,y),\qquad\forall\,(x,y)\in{\mathbb{R}}^n_{+}\times{\mathbb{R}}^n_{+}\setminus{\rm diag},
\end{equation}
extends to a function 
\begin{equation}\label{defRRR-smooth}
R_L(\cdot,\cdot)\in\big[{\mathscr{C}}^\infty\big({\mathbb{R}}^n_{+}\times{\mathbb{R}}^n_{+}\big)\big]^{M\times M}
\end{equation}
which satisfies the following estimate: for any 
multi-indices $\alpha,\beta\in\mathbb{N}_0^n$ there exists a finite 
constant $C_{\alpha\beta}>0$ with the property that
for every $(x,y)\in{\mathbb{R}}^n_{+}\times{\mathbb{R}}^n_{+}$,
\begin{equation}\label{mainest}
\big|\big(\partial^\alpha_X\partial^\beta_Y R_L\big)(x,y)\big|\leq\left\{
\begin{array}{l}
C_{\alpha\beta}\,|x-\overline{y}|^{2-n-|\alpha|-|\beta|}\,\,
\text{ if $|\alpha|+|\beta|>0$, or $n\geq 3$},
\\[10pt]
C+C\big|\!\ln|x-\overline{y}|\big|
\,\,\text{ if $|\alpha|=|\beta|=0$ and $n=2$}, 
\end{array}
\right.
\end{equation}
where $C\in(0,\infty)$, and $\overline{y}:=(y',-y_n)$ if  
$y=(y',y_n)\in\mathbb{R}^n_+$.

\vskip 0.08in
\item[(7)] For any multi-indices $\alpha,\beta\in\mathbb{N}_0^n$ 
there exists a finite constant $C_{\alpha\beta}>0$ such that
\begin{equation}\label{mainest2G}
\begin{array}{c}
\big|\big(\partial^\alpha_X\partial^\beta_Y G^L\big)(x,y)\big|
\leq C_{\alpha\beta}|x-y|^{2-n-|\alpha|-|\beta|}
\\[10pt]
\text{$\forall\,(x,y)\in{\mathbb{R}}^n_{+}\times{\mathbb{R}}^n_{+}\setminus{\rm diag}$,
if either $n\geq 3$, or $|\alpha|+|\beta|>0$},
\end{array}
\end{equation}
and, corresponding to $|\alpha|=|\beta|=0$ and $n=2$, there exists
$C\in(0,\infty)$ such that 
\begin{equation}\label{maKnaTTGB}
\big|G^L(x,y)\big|\leq C\big(1+\big|\!\ln|x-\overline{y}|\big|\big),
\quad\forall\,(x,y)\in{\mathbb{R}}^2_{+}\times{\mathbb{R}}^2_{+}\setminus{\rm diag}.
\end{equation}
\item[(8)] For each $\alpha,\beta\in\mathbb{N}_0^n$ one has
\begin{equation}\label{mainest3G}
\begin{array}{c}
\sup\limits_{y\in{\mathbb{R}}^n_{+}}\big\|\big(\partial^\alpha_X
\partial^\beta_Y G^L\big)(\cdot,y)\big\|_{\big[L^{\frac{n}{n-2+|\alpha|+|\beta|},\,\infty}({\mathbb{R}}^n_{+})
\big]^{M\times M}}<+\infty,
\\[10pt]
\text{if either $n\geq 3$, or $|\alpha|+|\beta|>0$}.
\end{array}
\end{equation}
In particular, in a uniform fashion with respect to $y\in{\mathbb{R}}^n_{+}$, 
\begin{align}\label{mainest2G.1}
& \text{the entries of $G^L(\cdot,y)$ are in $L^{\frac{n}{n-2},\,\infty}({\mathbb{R}}^n_{+})$ if }\,\,n\geq 3,
\\[4pt]
& \text{the entries of $\nabla_X G^L(\cdot,y)$ and $\nabla_Y G^L(\cdot,y)$ are in }\,\,
L^{\frac{n}{n-1},\,\infty}({\mathbb{R}}^n_{+}),
\label{mainest2G.2}
\\[4pt]
& \text{the entries of $\nabla^2_X G^L(\cdot,y),\nabla_X\nabla_Y G^L(\cdot,y),\nabla^2_Y G^L(\cdot,y)$
are in }\,\,L^{1,\infty}({\mathbb{R}}^n_{+}).
\label{mainest2G.3}
\end{align}
\item[(9)] If $p\in\big[1,\frac{n}{n-1}\big)$, then for each scalar function
$\zeta\in{\mathscr{C}}^\infty_c({\mathbb{R}}^n)$ one has
\begin{equation}\label{mainest4G}
\begin{array}{c}
\zeta G^L(\cdot,y)\in\big[\mathring{W}^{1,p}({\mathbb{R}}^n_{+})\big]^{M\times M}
\,\,\text{ for each }\,\,y\in{\mathbb{R}}^n_{+}
\\[10pt]
\text{and }\,\,\sup_{y\in{\mathbb{R}}^n_{+}}
\big\|\zeta G^L(\cdot,y)\big\|_{[W^{1,p}({\mathbb{R}}^n_{+})]^{M\times M}}<+\infty,
\end{array}
\end{equation}
where $\mathring{W}^{1,p}({\mathbb{R}}^n_{+})$ stands for the closure of ${\mathscr{C}}^\infty_c({\mathbb{R}}^n)$ in 
$W^{1,p}({\mathbb{R}}^n_{+})$.

\vskip 0.08in
\item[(10)] For every $x=(x',x_n)\in{\mathbb{R}}^n_{+}$ and every $y\in{\mathbb{R}}^n_{+}\setminus\{x\}$ 
one has {\rm (}with $P^L$ denoting the Poisson kernel for $L$ in ${\mathbb{R}}^n_{+}$ from Theorem~\ref{ya-T4-fav}{\rm )}
\begin{equation}\label{GHCewd-2PiK}
G^L(x,y)=E^L(x-y)-P^L_t\ast\Big(\big[E^L(\cdot-y)\big]\big|_{\partial{\mathbb{R}}^n_{+}}\Big)(x'),
\end{equation}
with the convolution applied to each column of the matrix inside the round parentheses.
\end{enumerate}
\end{theorem}

Our strategy for proving Theorem~\ref{ta.av-GGG.2A} involves the following steps.

In Step~1 we use the fundamental solution $E^L$ for $L$ in ${\mathbb{R}}^n$ from Theorem~\ref{FS-prop} and 
the Agmon-Douglis-Nirenberg Poisson kernel $P^L$ for the system $L$ in ${\mathbb{R}}^n_{+}$ from Theorem~\ref{ya-T4-fav}
in order to define $G^L(\cdot,\cdot)$ in the spirit of \eqref{GHCewd-2PiK}, with a caveat. Specifically, at this early stage in the proof 
we find it convenient to enhance the decay of the fundamental solution at infinity by working with $E^L(\cdot-y)-E^L(\cdot-\overline{y}\,)$ 
in place of $E^L(\cdot-y)$ (see \eqref{Ihab-Ygab673}). Based on the properties of the fundamental solution and the Poisson kernel we then 
see that this function satisfies \eqref{bound-NK-G}. Using elliptic regularity, in Step~2 we then establish \eqref{Fvabbb-7tF}.

In Step~3 we obtain a pointwise estimate for the error $R^L(\cdot,\cdot)$ defined as in \eqref{defRRR}, which amounts to the case 
$|\alpha|=|\beta|=0$ in \eqref{mainest}. This is already delicate as it makes use of more specialized properties of the fundamental solution 
$E^L$ and the Poisson kernel $P^L$, as well as the uniqueness for the $L^\infty$-Dirichlet boundary value problem for the Laplacian in 
${\mathbb{R}}^n_{+}$. Estimating higher order derivatives of $R^L(\cdot,\cdot)$ is significantly more challenging, 
and requires a circuitous approach. First, in Step~4, we make use of the full force of the Agmon-Douglis-Nirenberg estimates near the boundary 
(cf. Proposition~\ref{c1.2}) to deal with $\big(\partial^\alpha_X R_L\big)(\cdot,\cdot)$. In turn, this allows us to estimate in Step~5 the 
nontangential maximal function $\mathcal{N}^{\,{\mathbb{R}}^n_{+}\setminus K}_\kappa\big((\partial^\alpha_X G^L)(\cdot,y)\big)$ as in \eqref{bouMNN} 
with $|\beta|=0$. With this in hand, in Step~6 we are able to invoke the brand of Divergence Theorem developed in \cite{GHA.I} (presently 
recalled in Theorem~\ref{theor:div-thm} and subsequently used to establish Lemma~\ref{Lgav-TeD}) to conclude that $G^L(\cdot,\cdot)$ is the 
unique Green function for $L$ in the sense of Definition~\ref{ta.av-GGG}. As detailed in Step~7, Lemma~\ref{Lgav-TeD} and the quantitative aspects 
of $G^L(\cdot,\cdot)$ from Step~1 and Step~5 play a crucial role in justifying the transposition formula \eqref{GHCewd-22.RRe.5}. 
The idea now is to estimate $\big(\partial^\alpha_X\partial^\alpha_Y R_L\big)(\cdot,\cdot)$ based on what we have shown in 
Step~4 for $\big(\partial^\alpha_X R_{L^\top}\big)(\cdot,\cdot)$ and the transposition formula \eqref{GHCewd-22.RRe.5}.
In Step~8 we implement this strategy by once more relying heavily on the Agmon-Douglis-Nirenberg estimates near the boundary, in the format recalled 
in Proposition~\ref{c1.2}. At this stage, we have enough to justify the claims made in items {\it (1)}-{\it (9)} in the statement 
of Theorem~\ref{ta.av-GGG.2A}, and we do so in Steps~9-10. Finally, making use of the well-posedness result established from \cite[Theorem~1.21]{MMMM16},
in Step~11 we show that the initial formula for $G^L(\cdot,\cdot)$ in \eqref{Ihab-Ygab673} actually simplifies to \eqref{GHCewd-2PiK}. 

\medskip

Here is a very useful uniqueness result for the Dirichlet boundary value problem 
for Legendre-Hadamard elliptic systems in the upper-half space:

\begin{corollary}\label{thm:FP.CCC}
Let $L$ be an $M\times M$ system with constant complex coefficients as in \eqref{L-def}-\eqref{L-ell.X}, 
and fix some aperture parameter $\kappa>0$. Then 
\begin{equation}\label{jk-lm-jhR-LLL-HM-RN.w.}
\left.
\begin{array}{r}
u\in\big[{\mathscr{C}}^{\infty}({\mathbb{R}}^n_{+})\big]^M,\quad Lu=0\,\,\text{ in }\,\,{\mathbb{R}}^n_{+}
\\[8pt]
\displaystyle
\int_{\mathbb{R}^{n-1}}\big({\mathcal{N}}_{\kappa}u\big)(x')\,\frac{dx'}{1+|x'|^{n-1}}<\infty
\\[8pt]
u\big|^{{}^{\kappa-{\rm n.t.}}}_{\partial{\mathbb{R}}^n_{+}}=0\,\,\text{ at ${\mathscr{L}}^{n-1}$-a.e. point in }\,\,{\mathbb{R}}^{n-1}
\end{array}
\right\}
\Longrightarrow\,u\equiv 0\,\,\text{ in }\,\,{\mathbb{R}}^n_{+}.
\end{equation}
\end{corollary}

As noted in \eqref{GHUNNDs}, this is a consequence of the existence of a Green function as in Theorem~\ref{ta.av-GGG.2A}. Bearing this in mind, 
from Corollary~\ref{thm:FP.CCC} and \cite[(2.6.13), p.\,104]{GHA.V} we conclude that the weakly elliptic $n\times n$ system 
\begin{equation}\label{eq:LD}
L_D:=\Delta-2\nabla{\rm div}\,\,\text{ in }\,\,{\mathbb{R}}^n
\end{equation}
does not possess a Green function of the sort described in Theorem~\ref{ta.av-GGG.2A} (since there are actually infinitely many linearly 
independent null-solutions for the $L^p$-Dirichlet boundary value problem for $L_D$ in ${\mathbb{R}}^n_{+}$ whenever $1<p<\infty$). 
In particular, this shows that Legendre-Hadamard ellipticity (cf. \eqref{L-ell.X}) cannot be relaxed to mere weak ellipticity 
(cf. \eqref{Def-ELLa.INTR}) in the context of Theorem~\ref{ta.av-GGG.2A}.

There is a version of Theorem~\ref{ta.av-GGG.2A}, stated in Theorem~\ref{FS-prop.INTR} below, which assumes a different 
set of hypotheses for the system $L$, and in which the conclusions regarding the Green function $G^L$ are fine-tuned differently. 
Before presenting a formal statement, we need some definitions. 
Given a second-order, homogeneous, complex constant coefficient, $M\times M$ system $L$ in ${\mathbb{R}}^n$, written as in \eqref{L-def},
along with a matrix $W=(w_{jk})_{1\leq j,k\leq n}\in{\mathbb{C}}^{n\times n}$, define the second-order, homogeneous, complex constant coefficient, $M\times M$ system in ${\mathbb{R}}^n$, formally written as 
\begin{equation}\label{eq:kfgdc.WACO.D+M.new}
L\circ W:=\Big(a_{jk}^{\alpha\beta}(W\nabla)_j(W\nabla)_k\Big)_{1\leq\alpha,\beta\leq M}.
\end{equation}
We shall say that $L$ is invariant under $W$ if $L\circ W=L$. A direct computation reveals 
that if $L$ is written as in \eqref{L-def} then $L$ is invariant under $W$ if and only if 
\begin{equation}\label{JKBvc-ut4.YFav.a.222.D+M}
\begin{array}{c}
\text{for each $\alpha,\beta\in\{1,\dots,M\}$ and $j,k\in\{1,\dots,n\}$} 
\\[0pt]
\text{one has }\,\,\big(a_{rs}^{\alpha\beta}+a_{sr}^{\alpha\beta}\big)w_{rj}w_{sk}=a_{jk}^{\alpha\beta}+a_{kj}^{\alpha\beta}. 
\end{array}
\end{equation}
Of a particular importance to us is the case when $W$ is the reflection in ${\mathbb{R}}^n$ across the horizontal plane $x_n=0$, 
i.e., when $Wx=\overline{x}$ for each $x\in{\mathbb{R}}^n$. In such a scenario, in lieu of saying that $L$ is invariant under $W$ 
we shall simply call $L$ {\tt reflection} {\tt invariant}. As seen from \eqref{JKBvc-ut4.YFav.a.222.D+M}, 
if $L$ is written as in \eqref{L-def} then $L$ is reflection invariant if and only if 
\begin{equation}\label{JKBvc-ut4.YFav.a.222}
\begin{array}{c}
\text{for each $\alpha,\beta\in\{1,\dots,M\}$ one has} 
\\[0pt]
a_{jn}^{\alpha\beta}+a_{nj}^{\alpha\beta}=0\,\,\text{ whenever }\,\,1\leq j<n. 
\end{array}
\end{equation}
In particular,
\begin{equation}\label{JKBvc-ut4.YFav.a.333}
\parbox{11.10cm}{the system $L$ is reflection invariant if it may be written as in \eqref{L-def} using a coefficient tensor 
$A=\big(a_{jk}^{\alpha\beta}\big)_{\substack{1\leq j,k\leq n\\ 1\leq\alpha,\beta\leq M}}$
which has the following block-structure: $a_{jn}^{\alpha\beta}=a_{nj}^{\alpha\beta}=0$ for each $j\in\{1,\dots,n-1\}$.}
\end{equation}
We wish to remark that if the system $L$ is invariant under rotations, then $L$ is reflection invariant.
Finally, as noted in \eqref{eq:edxuye.FDa.4}, a weakly elliptic system $L$ is rotation invariant if and only if 
the fundamental solution $E^L$ canonically associated with $L$ as in Theorem~\ref{FS-prop} 
is rotation invariant, hence a radial function in ${\mathbb{R}}^n\setminus\{0\}$.

The point of the theorem below is that in the aforementioned scenario 
the logarithm in the right side of \eqref{bound-NK-G} may actually be omitted, and one has better decay in \eqref{bouMNN}
(at least when $\beta=0$). Another notable feature is the fact that, under the assumption that $L$ is reflection invariant
we may relax the strong ellipticity assumption from Theorem~\ref{ta.av-GGG.2A} and only demand that $L$ is weakly elliptic.
These assumptions no longer guarantee the existence of a Poisson kernel for $L$, so a new formula for the Green function 
is required (compare \eqref{GHCewd-2PiK} with \eqref{JKBvc-ut4.YFav}). 

\begin{theorem}\label{FS-prop.INTR}
Fix $n,M\in{\mathbb{N}}$, with $n\geq 2$, and consider a second-order, homogeneous, complex constant coefficient, 
$M\times M$ system $L$ in ${\mathbb{R}}^n$. If $L$ is weakly elliptic and reflection invariant, then there exists a unique 
Green function $G^L(\cdot,\cdot)$ for $L$ in $\mathbb{R}^n_{+}$, in the sense of Definition~\ref{ta.av-GGG}.
Specifically, with $E^L$ denoting the fundamental solution canonically associated with $L$ as in Theorem~\ref{FS-prop}, 
the aforementioned unique Green function is given by 
\begin{equation}\label{JKBvc-ut4.YFav}
G^L(x,y):=E^L(x-y)-E^L(x-\overline{y})\,\,\text{ for all }\,\,(x,y)\in{\mathbb{R}}^n_{+}\times{\mathbb{R}}^n_{+}\setminus{\rm diag}.
\end{equation}
Moreover, one has the following partial improvement of \eqref{bound-NK-G}-\eqref{bouMNN}: given any aperture 
parameter $\kappa>0$, any pole $y\in{\mathbb{R}}^n_{+}$, any compact subset $K$ of ${\mathbb{R}}^n_{+}$ with $y\in\mathring{K}$, 
and any multi-index $\alpha\in\mathbb{N}_0^n$, there exists a finite constant $C^{\alpha}_{y,K,\kappa}>0$ 
such that
\begin{equation}\label{bouMNN.XXX}
\Big(\mathcal{N}^{\,{\mathbb{R}}^n_{+}\setminus K}_\kappa
(\partial^\alpha_X G^L)(\cdot,y)\Big)(x')\leq\frac{C^{\alpha}_{y,K,\kappa}}{1+|x'|^{n-1+|\alpha|}}
\end{equation}
for every $x'\in\mathbb{R}^{n-1}$. In addition, one has the following refinement of \eqref{GHCewd-22.RRe.5}:
if $G^{L^\top}\!\!(\cdot,\cdot)$ denotes the {\rm (}unique, by the first claim in the present statement{\rm )} 
Green function for $L^\top$ in ${\mathbb{R}}^n_{+}$ then 
\begin{equation}\label{GHCewd-22.RRe.5.D+M}
\begin{array}{c}
G^L(x,y)=G^L(y,x)\,\,\text{ and }\,\,G^L(x,y)=\Big[G^{L^\top}\!\!(x,y)\Big]^\top
\\[4pt]
\text{for all pairs }\,\,(x,y)\in{\mathbb{R}}^n_{+}\times{\mathbb{R}}^n_{+}\setminus{\rm diag}.
\end{array}
\end{equation}
Finally, if $L$ is as in \eqref{L-def}-\eqref{L-ell.X} and also reflection invariant, then the Green function 
constructed in \eqref{JKBvc-ut4.YFav} coincides with the Green function from Theorem~\ref{ta.av-GGG.2A}.
\end{theorem}

For example, the operator $L_\lambda:=\partial_1^2+\cdots+\partial_{n-1}^2+\lambda\partial_n^2$ from \eqref{eq:LLLambda} 
has a Green function as in Theorem~\ref{FS-prop.INTR} if $\lambda\in{\mathbb{C}}\setminus(-\infty,0]$, and this Green function also 
satisfies the conclusions in Theorem~\ref{ta.av-GGG.2A} whenever ${\rm Re}\,\lambda>0$.

Our next result in this section describes the rather precise relationship between the 
Agmon-Douglis-Nirenberg Poisson kernel and the Green function considered in Theorem~\ref{ta.av-GGG.2A}. 

\begin{corollary}\label{y6tFFF.cc}
Assume that $L$ is an $M\times M$ system with constant complex coefficients as in 
\eqref{L-def}-\eqref{L-ell.X}. Then there exists a unique Poisson kernel $P^L$ for 
$L$ in ${\mathbb{R}}^n_{+}$ in the sense of Definition~\ref{defi:Poisson}. Moreover, 
one may recover $P^L=\big(P^L_{\gamma\alpha}\big)_{1\leq\gamma,\alpha\leq M}$ from 
the Green function $G^L$ according to the formula 
\begin{equation}\label{Ua-eD.LBBw.3B}
\begin{array}{c}
P^L_{\gamma\alpha}(x')=-a^{\beta\alpha}_{nn}
\big(\partial_{Y_n} G^{L}_{\gamma\beta}\big)\big((x',1),0\big),
\quad\forall\,x'\in{\mathbb{R}}^{n-1},
\\[6pt]
\text{for each pair of indexes }\,\,\alpha,\gamma\in\{1,\dots,M\}.
\end{array}
\end{equation}
In addition, if the fundamental solution $E^L=\big(E^L_{\gamma\beta}\big)_{1\leq\gamma,\beta\leq M}$ 
of $L$ from Theorem~\ref{FS-prop} is a radial function, then for 
each $\alpha,\gamma\in\{1,\dots,M\}$ one has
\begin{equation}\label{Ua-eD.kab}
P^L_{\gamma\alpha}(x')=2a^{\beta\alpha}_{nn}(\partial_n E^L_{\gamma\beta})(x',1),
\qquad\forall\,x'\in{\mathbb{R}}^{n-1}.
\end{equation}
\end{corollary}

Historically, Green function estimates and related regularity results have played a fundamental 
role in partial differential equations. Theorem~\ref{ta.av-GGG.2A} is particularly useful in the 
treatment of those boundary value problems for systems in which the size of the solution is measured 
by means of the nontangential maximal operator. For a multitude of examples of Dirichlet boundary 
problems for elliptic systems in the upper-half space formulated in this spirit the reader is 
referred to \cite{MMMM16}, \cite{GHA.V}, \cite{MMMMM}, \cite{LaMi24}, \cite{MiTak24}.

To illustrate this link, here we present a Fatou-type theorem and a Poisson's integral formula for elliptic systems. 

\begin{theorem}\label{thm:FP}
Let $L$ be an $M\times M$ system with constant complex coefficients as in \eqref{L-def}-\eqref{L-ell.X}, 
and fix some aperture parameter $\kappa>0$. Then 
\begin{equation}\label{jk-lm-jhR-LLL-HM-RN.w}
\left\{
\begin{array}{l}
u\in\big[{\mathscr{C}}^{\infty}({\mathbb{R}}^n_{+})\big]^M,\quad Lu=0\,\,\text{ in }\,\,{\mathbb{R}}^n_{+},
\\[8pt]
\displaystyle
\int_{\mathbb{R}^{n-1}}\big({\mathcal{N}}_{\kappa}u\big)(x')\,\frac{dx'}{1+|x'|^{n-1}}<\infty,
\end{array}
\right.
\end{equation}
implies that 
\begin{eqnarray}\label{Tafva.2222}
\left\{
\begin{array}{l}
u\big|^{{}^{\kappa-{\rm n.t.}}}_{\partial{\mathbb{R}}^n_{+}}
\,\,\text{ exists at ${\mathscr{L}}^{n-1}$-a.e. point in }\,\,{\mathbb{R}}^{n-1},
\\[10pt]
\displaystyle
u\big|^{{}^{\kappa-{\rm n.t.}}}_{\partial{\mathbb{R}}^n_{+}}\,\,\text{ belongs to }\,\,
\Big[L^1\Big({\mathbb{R}}^{n-1}\,,\,\frac{dx'}{1+|x'|^{n-1}}\Big)\Big]^M,
\\[12pt]
u(x',t)=\Big(P^L_t\ast\big(u\big|^{{}^{\kappa-{\rm n.t.}}}_{\partial{\mathbb{R}}^n_{+}}\big)\Big)(x')
\,\,\text{ for each }\,\,(x',t)\in{\mathbb{R}}^n_{+},
\end{array}
\right.
\end{eqnarray}
where $P^L=\big(P^L_{\beta\alpha}\big)_{1\leq\beta,\alpha\leq M}$
is the Agmon-Douglis-Nirenberg Poisson kernel for the system $L$ in ${\mathbb{R}}^n_{+}$ 
and $P^L_t(x'):=t^{1-n}P^L(x'/t)$ for each $x'\in{\mathbb{R}}^{n-1}$ and $t>0$.

Moreover, under the additional assumption that $L$ is also reflection invariant, whenever  
\begin{equation}\label{jk-lm-jhR-LLL-HM-RN.w.BIS}
\left\{
\begin{array}{l}
u\in\big[{\mathscr{C}}^{\infty}({\mathbb{R}}^n_{+})\big]^M,\quad Lu=0\,\,\text{ in }\,\,{\mathbb{R}}^n_{+},
\\[8pt]
\displaystyle
\int_{\mathbb{R}^{n-1}}\big({\mathcal{N}}_{\kappa}u\big)(x')\,\frac{dx'}{1+|x'|^{n}}<\infty,
\end{array}
\right.
\end{equation}
it follows that 
\begin{eqnarray}\label{Tafva.2222.BIS}
\left\{
\begin{array}{l}
u\big|^{{}^{\kappa-{\rm n.t.}}}_{\partial{\mathbb{R}}^n_{+}}
\,\,\text{ exists at ${\mathscr{L}}^{n-1}$-a.e. point in }\,\,{\mathbb{R}}^{n-1},
\\[10pt]
\displaystyle
u\big|^{{}^{\kappa-{\rm n.t.}}}_{\partial{\mathbb{R}}^n_{+}}\,\,\text{ belongs to }\,\,
\Big[L^1\Big({\mathbb{R}}^{n-1}\,,\,\frac{dx'}{1+|x'|^{n}}\Big)\Big]^M,
\\[12pt]
u(x',t)=\Big(P^L_t\ast\big(u\big|^{{}^{\kappa-{\rm n.t.}}}_{\partial{\mathbb{R}}^n_{+}}\big)\Big)(x')
\,\,\text{ for each }\,\,(x',t)\in{\mathbb{R}}^n_{+}.
\end{array}
\right.
\end{eqnarray}
\end{theorem}

This refines \cite[Theorem~6.1, p.\,956]{MMMM16}. An alternative proof, which makes essential use of 
Theorem~\ref{ta.av-GGG.2A}, is given in \cite{MMMMM-Fatou}. We also wish to remark that even in the classical 
case when $L:=\Delta$, the Laplacian in ${\mathbb{R}}^n$, Theorem~\ref{thm:FP} is more general 
(in the sense that it allows for a larger class of functions) than the existing results in the literature. 
Indeed, the latter typically assume an $L^p$ integrability condition for the harmonic function which, in the 
range $1<p<\infty$, implies our weighted $L^1$ integrability condition for the nontangential 
maximal function demanded in \eqref{jk-lm-jhR-LLL-HM-RN.w}. In this vein see, e.g.,  
\cite[Theorems~4.8-4.9, pp.\,174-175]{GCRF85}, \cite[Corollary, p.\,200]{St70},
\cite[Proposition~1, p.\,119]{Stein93}. 

A remarkable feature of Theorem~\ref{thm:FP} is that while its statement is phrased exclusively 
in terms of the Agmon-Douglis-Nirenberg Poisson kernel $P^L$, its proof is actually carried out 
largely in terms of the Green function associated with the system $L$ in the upper-half space. 

It is of interest to identifying a rather inclusive setting in which the current machinery yields well-posedness results. 
With ${\mathcal{M}}$ denoting the Hardy-Littlewood maximal operator in ${\mathbb{R}}^{n-1}$, for each $m\in{\mathbb{N}}$ 
consider the linear space 
\begin{align}\label{76tFfaf-7GF}
{\mathscr{Z}}_m:=\Big\{f:{\mathbb{R}}^{n-1}\to{\mathbb{C}}:\,\text{ ${\mathscr{L}}^{n-1}$-measurable and }\,\,
{\mathcal{M}}f\in L^1\big({\mathbb{R}}^{n-1}\,,\,\tfrac{dx'}{1+|x'|^{m}}\big)\Big\}
\end{align}
equipped with the norm 
\begin{align}\label{76tFfaf-7GF.2}
\|f\|_{{\mathscr{Z}_m}}:=\|{\mathcal{M}}f\|_{L^1\big({\mathbb{R}}^{n-1},\frac{dx'}{1+|x'|^{m}}\big)},
\qquad\forall\,f\in{\mathscr{Z}}_m.
\end{align}

\begin{corollary}\label{Them-Gen}
Let $L$ be an $M\times M$ system with constant complex coefficients as in \eqref{L-def}-\eqref{L-ell.X}, 
and fix an aperture parameter $\kappa>0$. Then the following boundary-value problem is well-posed:
\begin{equation}\label{jk-lm-jhR-LLL-HM-RN.w.BVP}
\left\{
\begin{array}{l}
u\in\big[{\mathscr{C}}^{\infty}({\mathbb{R}}^n_{+})\big]^M,
\quad Lu=0\,\,\text{ in }\,\,{\mathbb{R}}^n_{+},
\\[8pt]
\displaystyle
\int_{\mathbb{R}^{n-1}}\big({\mathcal{N}}_{\kappa}u\big)(x')\,\frac{dx'}{1+|x'|^{n-1}}<\infty,
\\[12pt]
u\big|^{{}^{\kappa-{\rm n.t.}}}_{\partial{\mathbb{R}}^n_{+}}=f\in[{\mathscr{Z}}_{n-1}]^M.
\end{array}
\right.
\end{equation}

Furthermore, under the additional assumption that $L$ is also reflection invariant,
the following boundary-value problem is well-posed as well:
\begin{equation}\label{jk-lm-jhR-LLL-HM-RN.w.BVP.D+M}
\left\{
\begin{array}{l}
u\in\big[{\mathscr{C}}^{\infty}({\mathbb{R}}^n_{+})\big]^M,
\quad Lu=0\,\,\text{ in }\,\,{\mathbb{R}}^n_{+},
\\[8pt]
\displaystyle
\int_{\mathbb{R}^{n-1}}\big({\mathcal{N}}_{\kappa}u\big)(x')\,\frac{dx'}{1+|x'|^{n}}<\infty,
\\[12pt]
u\big|^{{}^{\kappa-{\rm n.t.}}}_{\partial{\mathbb{R}}^n_{+}}=f\in[{\mathscr{Z}}_{n}]^M.
\end{array}
\right.
\end{equation}
\end{corollary}

As a very special case, \eqref{jk-lm-jhR-LLL-HM-RN.w.BVP.D+M} contains the well-posedness of the $L^p$-Dirichlet problem, with $1<p<\infty$, 
for the Laplacian in the upper-half space, via a proof which does not make use of the Maximum Principle (as in Stein's 1970 book \cite{St70}), 
or Schwarz's Reflection Principle (as in the 1991 book of Garc{\'\i}a-Cuerva and Rubio de Francia \cite{GCRF85}).

The relevance of the fact that \eqref{jk-lm-jhR-LLL-HM-RN.w} 
implies \eqref{Tafva.2222} in the context of the boundary value problem formulated in \eqref{jk-lm-jhR-LLL-HM-RN.w.BVP} 
is that the nontangential boundary trace $u\big|^{{}^{\kappa-{\rm n.t.}}}_{\partial{\mathbb{R}}^n_{+}}$ is guaranteed to exist 
by the other conditions imposed on the function $u$ in the formulation of said problems, 
and that the solution may be recovered from the boundary datum via convolution 
with the Poisson kernel canonically associated with the system $L$.
Similar remarks apply in relation to the fact that \eqref{jk-lm-jhR-LLL-HM-RN.w.BIS} implies \eqref{Tafva.2222.BIS}, for the 
boundary value problem formulated in \eqref{jk-lm-jhR-LLL-HM-RN.w.BVP.D+M}.

The type of boundary value problems treated here, in which the size of the solution is measured in 
terms of its nontangential maximal function and its trace is taken in a nontangential 
pointwise sense, has been dealt with in the particular case when $L=\Delta$, 
the Laplacian in ${\mathbb{R}}^n$, in a number of monographs, including \cite{ABR}, \cite{GCRF85}, 
\cite{St70}, \cite{Stein93}, and \cite{StWe71}. In all these works, the existence part 
makes use of the explicit form of the harmonic Poisson kernel, 
while the uniqueness relies on either the Maximum Principle, or the Schwarz reflection 
principle for harmonic functions. Neither of the latter techniques may be adapted
successfully to prove uniqueness in the case of generic strongly elliptic systems treated here, and 
our approach is more in line with the work in \cite{MMMM16} (which involves 
Green function estimates and a sharp version of the Divergence Theorem), with some 
significant refinements. A remarkable aspect is that our approach works for the entire 
class of elliptic systems $L$ as in \eqref{L-def}-\eqref{L-ell.X}. 

In closing, we note that constructing Green functions in relation to various classes of 
partial differential operators (or systems) in certain geometric settings remains a relevant topic of active research, 
and that additional information and references may be found in \cite{DK21}, \cite{HK07}, \cite{HK20}, \cite{OKB15}, \cite{TKB12}.

\section{Preliminaries}
\setcounter{equation}{0}
\label{S-2}

Throughout, ${\mathbb{N}}$ stands for the collection of all strictly positive
integers, and ${\mathbb{N}}_0:={\mathbb{N}}\cup\{0\}$. As such, for each $k\in\mathbb{N}$,
we denote by $\mathbb{N}_0^k$ the collection of all multi-indices $\alpha=(\alpha_1,\dots,\alpha_k)$
with $\alpha_j\in\mathbb{N}_0$ for $1\leq j\leq k$. Also, fix $n\in{\mathbb{N}}$ with $n\geq 2$.
We shall work in the upper-half space ${\mathbb{R}}^{n}_{+}$, whose topological boundary 
$\partial{\mathbb{R}}^{n}_{+}={\mathbb{R}}^{n-1}\times\{0\}$ will be frequently identified 
with the horizontal hyperplane ${\mathbb{R}}^{n-1}$ via $(x',0)\equiv x'$. 
The origin in ${\mathbb{R}}^{n-1}$ is denoted by $0'$
and we let $B_{n-1}(x',r)$ stand for the $(n-1)$-dimensional Euclidean ball of radius $r$
centered at $x'\in{\mathbb{R}}^{n-1}$. Having fixed $\kappa>0$, for each boundary point
$x'\in\partial{\mathbb{R}}^{n}_{+}$ introduce the conical nontangential approach region
with vertex at $x'$ as
\begin{equation}\label{NT-1}
\Gamma_\kappa(x'):=\big\{y=(y',t)\in{\mathbb{R}}^{n}_{+}:\,|x'-y'|<\kappa\,t\big\}.
\end{equation}
Given a vector-valued function $u:{\mathbb{R}}^{n}_{+}\to{\mathbb{C}}^M$,
the nontangential maximal function of $u$ is defined by
\begin{equation}\label{NT-Fct}
\big({\mathcal{N}}_\kappa u\big)(x'):=\|u\|_{L^\infty(\Gamma_\kappa(x'))},\qquad
\forall\,x'\in\partial{\mathbb{R}}^{n}_{+}\equiv{\mathbb{R}}^{n-1}, 
\end{equation}
where the essential supremum norm is taken with respect to the Lebesgue measure ${\mathscr{L}}^n$. 
Whenever meaningful, we also define the nontangential trace of a function $u$ which is continuous near 
the boundary of ${\mathbb{R}}^n_{+}$ as 
\begin{equation}\label{nkc-EE-2}
\big(u\big|^{{}^{\kappa-{\rm n.t.}}}_{\partial{\mathbb{R}}^{n}_{+}}\big)(x')
:=\lim_{\Gamma_{\kappa}(x')\ni y\to (x',0)}u(y)
\quad\text{for }\,x'\in\partial{\mathbb{R}}^{n}_{+}\equiv{\mathbb{R}}^{n-1}.
\end{equation}
In the sequel, we shall need to consider a localized version of the
nontangential maximal operator. Specifically, given any
$E\subset{\mathbb{R}}^n_{+}$, for each $u:E\to{\mathbb{C}}^M$
we set
\begin{equation}\label{NT-Fct.23}
\big({\mathcal{N}}^E_\kappa u\big)(x'):=\|u\|_{L^\infty(\Gamma_\kappa(x')\cap E)},\qquad
\forall\,x'\in\partial{\mathbb{R}}^{n}_{+}\equiv{\mathbb{R}}^{n-1}.
\end{equation}
Hence, ${\mathcal{N}}^E_\kappa u={\mathcal{N}}_\kappa\widetilde{u}$ 
where $\widetilde{u}$ is the extension of $u$ to ${\mathbb{R}}^n_{+}$
by zero outside $E$. In the scenario when $u$ is defined in 
the entire upper-half space ${\mathbb{R}}^n_{+}$ to begin with, we may 
therefore write 
\begin{equation}\label{NT-Fct.23PPPP}
{\mathcal{N}}^E_\kappa u={\mathcal{N}}_\kappa({\bf 1}_E u).
\end{equation}

Later on, we shall also need the following result.

\begin{proposition}\label{p3.2.6}
For each number $\kappa>0$ and exponent $p\in(0,\infty)$ there exists 
some finite constant $C=C(n,p,\kappa)>0$ with the property that 
for each measurable set $E\subseteq{\mathbb{R}}^n_{+}$ and 
each measurable function $u:{\mathbb{R}}^n_{+}\to{\mathbb{C}}$ one has
\begin{equation}\label{3.2.38-BIS}
\|u\|_{L^{\frac{np}{n-1}}(E)}
\leq C\big\|\mathcal{N}^E_\kappa u\big\|_{L^p({\mathbb{R}}^{n-1})}.
\end{equation}
\end{proposition}

\begin{proof}
When $E={\mathbb{R}}^n_{+}$, this is a particular case of a more general 
estimate proved in \cite[Proposition~3.24, p.\,2647]{HoMiTa10}. As stated, 
\eqref{3.2.38-BIS} then follows from this and \eqref{NT-Fct.23PPPP}.
\end{proof}

The action of the Hardy-Littlewood maximal operator in $\mathbb{R}^{n-1}$ 
on any Lebesgue measurable function $f$ defined in ${\mathbb{R}}^{n-1}$ is given by 
\begin{equation}\label{MMax}
\big(\mathcal{M}f\big)(x'):=\sup_{r>0}\aver{B_{n-1}(x',r)}|f|\,d{\mathscr{L}}^{n-1},\qquad 
\forall\,x'\in\mathbb{R}^{n-1},
\end{equation}
where the barred integral denotes mean average (for functions which are $\mathbb{C}^M$-valued
the average is taken componentwise). Also, we shall follow the customary notation
$A\approx B$ in order to indicate that each quantity $A,B$ is dominated by a fixed multiple
of the other (via constants independent of the essential parameters intervening in $A,B$).
Recall that $\log_{+}t:=\max\big\{0\,,\,\ln t\big\}$ for each $t\in(0,\infty)$.
For a proof of the following lemma the reader is referred to \cite[Lemma~2.1]{MMMM16}.

\begin{lemma}\label{lemma:M-ball}
Consider $f:\mathbb{R}^{n-1}\to\mathbb{R}$ given by $f(x'):=(1+|x'|)^{1-n}$ for each $x'\in\mathbb{R}^{n-1}$. 
Then 
\begin{equation}\label{eq:M2-ball}
\big(\mathcal{M}f\big)(x')\approx\frac{1+\log_{+}|x'|}{1+|x'|^{n-1}},\qquad x'\in{\mathbb{R}}^{n-1},
\end{equation}
where the implicit constants depend only on $n$.
\end{lemma}

We also recall a useful weak compactness result \cite[Lemma~4.6.1, p.\,822]{GHA.III}. 

\begin{lemma}\label{ydadHBB}
Let $\omega:{\mathbb{R}}^{n-1}\to(0,\infty)$ be a Lebesgue measurable function and
consider a sequence $\{f_j\}_{j\in{\mathbb{N}}}\subset L^1({\mathbb{R}}^{n-1}\,,\omega{\mathscr{L}}^{n-1})$ such that 
\begin{equation}\label{JGTrad}
F:=\sup\limits_{j\in{\mathbb{N}}}|f_j|\in L^1({\mathbb{R}}^{n-1}\,,\omega{\mathscr{L}}^{n-1}).
\end{equation}

Then there exists a subsequence $\big\{f_{j_k}\big\}_{k\in{\mathbb{N}}}$ of $\{f_j\}_{j\in{\mathbb{N}}}$
and a function $f\in L^1({\mathbb{R}}^{n-1}\,,\omega{\mathscr{L}}^{n-1})$ with the property that
\begin{equation}\label{eq:16t44}
\begin{array}{c}
\displaystyle
\int_{{\mathbb{R}}^{n-1}}f_{j_k}\varphi\,\omega\,d{\mathscr{L}}^{n-1}\longrightarrow
\int_{{\mathbb{R}}^{n-1}}f\varphi\,\omega\,d{\mathscr{L}}^{n-1}\,\,\text{ as }\,\,k\to\infty,
\\[10pt]
\text{for every function }\,\,\varphi\in{\mathscr{C}}^0({\mathbb{R}}^{n-1})\cap L^\infty({\mathbb{R}}^{n-1},{\mathscr{L}}^{n-1}).
\end{array}
\end{equation}
\end{lemma}

We next discuss the notion of Poisson kernel in ${\mathbb{R}}^n_{+}$
for an operator $L$ as in \eqref{L-def}-\eqref{L-ell.X}.

\begin{definition}\label{defi:Poisson}
Let $L$ be an $M\times M$ system with constant complex coefficients as in \eqref{L-def}-\eqref{L-ell.X}.
A {\tt Poisson kernel} for $L$ in $\mathbb{R}^{n}_{+}$ is a matrix-valued function
\begin{equation}\label{uahgab-Uan}
P^L=\big(P^L_{\alpha\beta}\big)_{1\leq\alpha,\beta\leq M}:\mathbb{R}^{n-1}\longrightarrow\mathbb{C}^{M\times M}
\end{equation}
such that the following conditions hold:
\begin{list}{(\theenumi)}{\usecounter{enumi}\leftmargin=.8cm
\labelwidth=.8cm\itemsep=0.2cm\topsep=.1cm
\renewcommand{\theenumi}{\alph{enumi}}}
\item there exists $C\in(0,\infty)$ such that
$\displaystyle|P^L(x')|\leq\frac{C}{(1+|x'|^2)^{\frac{n}2}}$ for each
$x'\in\mathbb{R}^{n-1}$;
\item the function $P^L$ is Lebesgue measurable and
$\displaystyle\int_{\mathbb{R}^{n-1}}P^L(x')\,dx'=I_{M\times M}$,
the $M\times M$ identity matrix;
\item if $K^L(x',t):=P^L_t(x'):=t^{1-n}P^L(x'/t)$, for each
$x'\in\mathbb{R}^{n-1}$ and $t\in(0,\infty)$, then the function
$K^L=\big(K^L_{\alpha\beta}\big)_{1\leq\alpha,\beta\leq M}$
satisfies {\rm (}in the sense of distributions{\rm )}
\begin{equation}\label{uahgab-UBVCX}
LK^L_{\cdot\beta}=0\,\,\text{ in }\,\,\mathbb{R}^{n}_{+}
\,\,\text{ for each }\,\,\beta\in\{1,\dots,M\},
\end{equation}
where $K^L_{\cdot\beta}:=\big(K^L_{\alpha\beta}\big)_{1\leq\alpha\leq M}$.
\end{list}
\end{definition}

\vskip 0.06in
\begin{remark}\label{Ryf-uyf}
The following comments pertain to Definition~\ref{defi:Poisson}.
\begin{list}{\textit{(\theenumi)}}{\usecounter{enumi}\leftmargin=.8cm
\labelwidth=.8cm\itemsep=0.2cm\topsep=.1cm
\renewcommand{\theenumi}{\roman{enumi}}}
\item Condition {\it (a)} ensures that the integral in part $(b)$ is
absolutely convergent.
\item Condition {\it (c)} and the ellipticity of the operator $L$ ensure
(cf. \cite[Theorem~10.9, pp.\,363-364]{DMit18}) that
$K^L\in{\mathscr{C}}^\infty(\mathbb{R}^{n}_{+})$. In particular, \eqref{uahgab-UBVCX}
holds in a pointwise sense. Also, given that $P^L(x')=K^L(x',1)$ for each
$x'\in{\mathbb{R}}^{n-1}$, we deduce that $P^L\in{\mathscr{C}}^\infty(\mathbb{R}^{n-1})$.
\item Condition {\it (b)} is equivalent to $\lim\limits_{t\to 0^{+}}P^L_t(x')
=\delta_{0'}(x')\,I_{M\times M}$ in $\big[{\mathcal{D}}'({\mathbb{R}}^{n-1})\big]^{M\times M}$,
where $\delta_{0'}$ is Dirac's distribution with mass at the origin $0'$
of ${\mathbb{R}}^{n-1}$.
\item For all $x\in{\mathbb{R}}^n_{+}$ and $\lambda>0$ we have $K^L(\lambda x)=\lambda^{1-n}K^L(x)$.
\end{list}
\end{remark}

Poisson kernels for elliptic boundary value problems in a half-space have
been studied extensively in \cite{ADNI}, \cite{ADNII}, \cite[\S{10.3}]{KMR2},
\cite{Sol}, \cite{Sol1}, \cite{Sol2}. Here we record a corollary of more general
work done by S.\,Agmon, A.\,Douglis, and L.\,Nirenberg in \cite{ADNII}.

\begin{theorem}\label{ya-T4-fav}
Any $M\times M$ system $L$ with constant complex coefficients as in \eqref{L-def}-\eqref{L-ell.X} has a Poisson
kernel $P^L$ in the sense of Definition~\ref{defi:Poisson}, which has the additional
property that the function
\begin{equation}\label{eq:KDEF}
K^L(x',t):=P^L_t(x')\quad\text{for all }\,\,(x',t)\in{\mathbb{R}}^n_{+},
\end{equation}
satisfies $K^L\in\big[{\mathscr{C}}^\infty\big(\overline{{\mathbb{R}}^n_{+}}
\setminus B(0,\varepsilon)\big)\big]^{M\times M}$ for every $\varepsilon>0$, 
and has the property that for each multi-index $\alpha\in{\mathbb{N}}_0^n$
there exists $C_\alpha\in(0,\infty)$ such that
\begin{equation}\label{eq:KjG}
\big|(\partial^\alpha K^L)(x)\big|\leq C_\alpha\,|x|^{1-n-|\alpha|},\,\,\,\text{ for every }\,\,
x\in{\overline{{\mathbb{R}}^n_{+}}}\setminus\{0\}.
\end{equation}
\end{theorem}

We shall henceforth refer to the function $K^L$ defined in \eqref{eq:KDEF} as
the Agmon-Douglis-Nirenberg kernel associated with $L$.

To continue, we make the convention that the convolution between two functions 
which are matrix-valued and vector-valued, respectively, takes into account the 
algebraic multiplication between a matrix and a vector in a natural fashion.
The next result we recall has been proved in \cite[Theorem~3.1, p.\,934]{MMMM16}.

\begin{proposition}\label{thm:existence}
Let $L$ be an $M\times M$ system with constant complex coefficients as in \eqref{L-def}-\eqref{L-ell.X}, 
and recall the Poisson kernel $P^L$ for $L$ in $\mathbb{R}^{n}_{+}$ from Theorem~\ref{ya-T4-fav}.
Also, fix some arbitrary aperture parameter $\kappa>0$. 
Given a function 
\begin{equation}\label{exist:f}
f\in\Big[L^1\Big(\mathbb{R}^{n-1}\,,\,\frac{dx'}{1+|x'|^n}\Big)\Big]^M,
\end{equation}
set
\begin{equation}\label{exist:u}
u(x',t):=(P^L_t\ast f)(x'),\qquad\forall\,(x',t)\in{\mathbb{R}}^n_{+}.
\end{equation}

Then $u$ is meaningfully defined via an absolutely convergent integral,
\begin{equation}\label{exist:u2}
u\in\big[\mathscr{C}^\infty(\mathbb{R}^n_{+})\big]^{M},\quad
Lu=0\,\,\text{ in }\,\,\mathbb{R}^{n}_{+},\quad
u\big|_{\partial\mathbb{R}^{n}_{+}}^{{}^{\kappa-{\rm n.t.}}}=f
\,\,\text{ at ${\mathscr{L}}^{n-1}$-a.e. point in }\,\,\mathbb{R}^{n-1}
\end{equation}
{\rm (}with the last identity valid in the set of Lebesgue points of $f${\rm )}, 
and there exists a constant $C=C(n,L,\kappa)\in(0,\infty)$ with the property that
\begin{equation}\label{exist:Nu-Mf}
\big(\mathcal{N}_{\kappa}u\big)(x')\leq C\big(\mathcal{M}f\big)(x'),\qquad\forall\,x'\in\mathbb{R}^{n-1}.
\end{equation}
\end{proposition}
 
We next record the following result, detailing a construction of a fundamental solution 
for second-order, homogeneous, constant (complex) coefficient, weakly elliptic systems, which is a special 
case of \cite[Theorem~11.1, pp.\,393--396]{DMit18}. Throughout, the summation convention over repeated indices is in effect.
Also, ${\mathcal{H}}^{n-1}$ denotes the $(n-1)$-dimensional Hausdorff measure in ${\mathbb{R}}^n$. 

\begin{theorem}\label{FS-prop}
Fix $n,M\in{\mathbb{N}}$, with $n\geq 2$, and let 
\begin{equation}\label{L-deEfab.utf}
L=\big(a_{rs}^{\alpha\beta}\partial_r\partial_s\big)_{1\leq\alpha,\beta\leq M}
\end{equation}
be a homogeneous $M\times M$ second-order system in ${\mathbb{R}}^n$,
with complex constant coefficients, which is weakly elliptic in the sense that 
its $M\times M$ characteristic matrix
\begin{equation}\label{Def-ELLa}
L(\xi):=\big(a^{\alpha\beta}_{rs}\xi_r\xi_s\big)_{1\leq\alpha,\beta\leq M},
\qquad\forall\,\xi=(\xi_r)_{1\leq r\leq n}\in{\mathbb{R}}^n,
\end{equation}
satisfies
\begin{equation}\label{Def-ELLa-WE}
{\rm det}\big[L(\xi)\big]\not=0,\qquad\forall\,\xi\in{\mathbb{R}}^n\setminus\{0\}.
\end{equation}

Consider the $M\times M$ matrix-valued function $E^L=\big(E^L_{\alpha\beta}\big)_{1\leq\alpha,\beta\leq M}$ defined 
at each $x\in{\mathbb{R}}^{n}\setminus\{0\}$ by\footnote{for each $m\in{\mathbb{N}}$, we let $\Delta_x^m$ denote 
the $m$-fold application of the Laplacian in the variable $x$}
\begin{equation}\label{Def-ES1-GLOB}
E^L(x):=\left\{
\begin{array}{ll}
\displaystyle
\frac{\Delta_x^{(n-1)/2}}{4(2\pi\,i)^{n-1}}\Bigg\{\,
\int\limits_{S^{n-1}}\big|\langle x,\xi\rangle\big|
\big[L(\xi)\big]^{-1}d{\mathcal{H}}^{\,n-1}(\xi)\Bigg\}
&\text{ if $n$ is odd},
\\[30pt]
\displaystyle
-\frac{\Delta_x^{(n-2)/2}}{(2\pi\,i)^{n}}\Bigg\{\,\int\limits_{S^{n-1}}
\ln\big|\langle x,\xi\rangle\big|\,\big[L(\xi)\big]^{-1}d{\mathcal{H}}^{\,n-1}(\xi)\Bigg\}
&\text{ if $n$ is even}.
\end{array}
\right.
\end{equation}

In relation to the ${\mathbb{C}}^{M\times M}$-valued function \eqref{Def-ES1-GLOB}, 
the following properties hold:

\begin{enumerate}
\item[(1)] For each $\alpha,\beta\in\{1,\dots,M\}$ one has
\begin{equation}\label{smmth-odd}
\begin{array}{c}
E^L_{\alpha\beta}\in{\mathscr{C}}^{\,\infty}(\mathbb{R}^n\setminus\{0\})\cap L^1_{\rm loc}({\mathbb{R}}^n),
\\[8pt]
E^L_{\alpha\beta}(-x)=E_{\alpha\beta}(x)\,\,\text{ for all }\,\,x\in{\mathbb{R}}^n\setminus\{0\}.
\end{array}
\end{equation}
In fact, each entry $E^L_{\alpha\beta}$ is a real-analytic function in ${\mathbb{R}}^n\setminus\{0\}$.
Moreover, each $E^L_{\alpha\beta}$ is an even tempered distribution in $\mathbb{R}^n$
{\rm (}induced via integration against Schwartz functions{\rm )}. In addition, 
each $E^L_{\alpha\beta}$ is positive homogeneous of degree $2-n$ if $n\geq 3$.

\vskip 0.08in
\item[(2)] If for each $y\in{\mathbb{R}}^n$ one denotes by $\delta_y$ Dirac's delta distribution
with mass at $y$ in ${\mathbb{R}}^n$, then in the sense of tempered distributions in ${\mathbb{R}}^n$ one has
\begin{equation}\label{fs-GLOB}
L_x\big[E^L(x-y)\big]=\delta_y(x)\,I_{M\times M},\qquad\forall\,y\in{\mathbb{R}}^n,
\end{equation}
where $I_{M\times M}$ is the $M\times M$ identity matrix, and the subscript $x$ indicates 
that the operator $L$ in \eqref{fs-GLOB} is applied to each column of the matrix $E^L(x-y)$ in the variable $x$.

\vskip 0.08in
\item[(3)] For each multi-index $\gamma\in\mathbb{N}_0^n$ with $|\gamma|>0$, the tempered distribution 
$\partial^\gamma E^L$ is positive homogeneous of degree $2-n-|\gamma|$ in ${\mathbb{R}}^n$.
This is also true for $|\gamma|=0$ provided $n\geq 3$, i.e., the tempered distribution
$E^L$ is positive homogeneous of degree $2-n$ in ${\mathbb{R}}^n$ if $n\geq 3$.
Finally, corresponding to $n=2$ and $|\gamma|=0$, one may express
\begin{equation}\label{fs-str}
E^L(x)=\Phi(x)+\frac{\ln|x|}{4\pi^2}\int_{S^{1}}\big[L(\xi)\big]^{-1}\,d{\mathcal{H}}^1(\xi),
\qquad\forall\,x\in\mathbb{R}^2\setminus\{0\},
\end{equation}
where $\Phi:\mathbb{R}^2\setminus\{0\}\to{\mathbb{C}}^{M\times M}$, given by 
\begin{equation}\label{Def-ES2-Eeda}
\Phi(x):=\frac{1}{4\pi^2}\int\limits_{S^{1}}\ln\Big|\Big\langle\frac{x}{|x|},\xi\Big\rangle\Big|
\,\big[L(\xi)\big]^{-1}\,d{\mathcal{H}}^{1}(\xi),\qquad
\forall\,x\in{\mathbb{R}}^2\setminus\{0\},
\end{equation}
is a function of class ${\mathscr{C}}^{\,\infty}$ and positive homogeneous of degree $0$
in ${\mathbb{R}}^2\setminus\{0\}$.

\vskip 0.03in
\item[(4)] For each $\gamma\in\mathbb{N}_0^n$ there exists a finite constant
$C_\gamma>0$ such that for each $x\in{\mathbb{R}}^n\setminus\{0\}$
\begin{equation}\label{fs-est}
\big|(\partial^\gamma E^L)(x)\big|\leq
\left\{
\begin{array}{l}
\displaystyle\frac{C_\gamma}{|x|^{n+|\gamma|-2}}
\,\,\text{ if either $n\geq 3$, or $n=2$ and $|\gamma|>0$},
\\[16pt]
C_0\big(1+\big|\ln|x|\big|\big)\quad\text{ whenever }\,\,n=2
\,\,\text{ and }\,\,|\gamma|=0.
\end{array}
\right.
\end{equation}

\item[(5)] The fundamental solution $E^L$ defined in \eqref{Def-ES1-GLOB} satisfies 
\begin{equation}\label{E-Trans}
\begin{array}{c}
\big(E^L\big)^\top=E^{L^\top},\quad
\overline{\big(E^L\big)}=E^{\overline{L}},\quad
\big(E^L\big)^\ast=E^{L^\ast},
\\[8pt]
\text{as well as }\,E^{\lambda L}=\lambda^{-1}E^L
\,\,\text{ for each }\,\lambda\in{\mathbb{C}}\setminus\{0\}.
\end{array}
\end{equation}
where $\top$, $\overline{\cdot}$, $\ast$ denote, respectively, transposition, complex conjugation, 
and complex {\rm (}or Hermitian{\rm )} adjunction. 

\vskip 0.03in
\item[(6)] Let `hat' denote the Fourier transform in ${\mathbb{R}}^n$ {\rm (}originally defined 
on Schwartz functions, then extended to tempered distributions via duality{\rm )}.
Then $\widehat{E}$ is a tempered distribution in ${\mathbb{R}}^n$ {\rm (}which is positive homogeneous 
of degree $-2$ if $n\geq 3${\rm )}, whose restriction to ${\mathbb{R}}^n\setminus\{0\}$ 
is a {\rm (}matrix-valued{\rm )} function of class ${\mathscr{C}}^{\infty}$. In fact, 
\begin{equation}\label{E-ftXC}
\widehat{E}(\xi)=\big[L(\xi)\big]^{-1}\,\,\text{ for every }\,\,\xi\in{\mathbb{R}}^n\setminus\{0\}.
\end{equation}
\end{enumerate}
\end{theorem}

We augment the above theorem with a couple of remarks. 

\begin{remark}\label{tafh.R1}
As it may be easily seen from definitions, two systems are equal {\rm (}as linear mappings acting on the 
space of distributions in ${\mathbb{R}}^n${\rm )} if and only if their characteristic matrices coincide 
{\rm (}as matrix-valued functions defined in ${\mathbb{R}}^n${\rm )}, i.e., with the piece of notation 
introduced in \eqref{Def-ELLa}, 
\begin{equation}\label{eq:edxuye.FDa}
L_1=L_2\,\Longleftrightarrow\,L_1(\xi)=L_2(\xi)\,\,\text{ for each }\,\,\xi\in{\mathbb{R}}^n.
\end{equation}
Alternatively, 
\begin{equation}\label{L-def.vqusvq.I}
\Big(a^{\alpha\beta}_{rs}\partial_r\partial_s\Big)_{1\leq\alpha,\beta\leq M}
=\Big(\widetilde{a}^{\alpha\beta}_{rs}\partial_r\partial_s\Big)_{1\leq\alpha,\beta\leq M}
\end{equation}
(as mappings on vector distributions) if and only if 
\begin{equation}\label{L-def.vqusvq.II}
a^{\alpha\beta}_{rs}+a^{\alpha\beta}_{sr}
=\widetilde{a}^{\alpha\beta}_{rs}+\widetilde{a}^{\alpha\beta}_{sr}
\,\,\text{ for all }\,\,\alpha,\beta\in\{1,\dots,M\},\,\,r,s\in\{1,\dots,n\}.
\end{equation}
\end{remark}

\begin{remark}\label{tafh.R2}
With notation introduced in \eqref{eq:kfgdc.WACO.D+M.new} and \eqref{Def-ES1-GLOB}, one has
\begin{equation}\label{eq:edxuye.FDa.2}
E^{L\circ R}=E^L\circ R\,\,\text{ for any unitary transformation $R$ in ${\mathbb{R}}^n$}.
\end{equation}
Indeed, this is seen from definitions, the invariance of integral over the unit sphere under unitary transformations 
(cf., e.g., \cite{HoMiTa07}, \cite{GHA.II}), and the invariance of the Laplacian under unitary transformations 
(cf., e.g., \cite[Exercise~7.77, (7.14.3), p.\,319]{DMit18}). Consequently, from \eqref{eq:edxuye.FDa.2}, 
\eqref{eq:edxuye.FDa}, and \eqref{E-ftXC} we see that 
\begin{equation}\label{eq:edxuye.FDa.3}
\parbox{10.00cm}{given a unitary transformation $R$ in ${\mathbb{R}}^n$, it follows that $E^L$ is invariant under $R$ 
if and only if $L$ is invariant under $R$.}
\end{equation}
In particular, 
\begin{equation}\label{eq:edxuye.FDa.4}
\parbox{9.20cm}{$E^L$ is rotation (respectively, reflection) invariant  
if and only if $L$ is rotation (respectively, reflection) invariant.}
\end{equation}
\end{remark}

Moving on, we shall now record the following versatile version of interior estimates for
second-order elliptic systems. A proof may be found in \cite[Theorem~11.12, p.\,415]{DMit18}.

\begin{theorem}\label{ker-sbav}
Consider a homogeneous, constant coefficient, second-order, system $L$ satisfying the weak 
ellipticity condition ${\rm det}\,[L(\xi)]\not=0$ for each $\xi\in{\mathbb{R}}^n\setminus\{0\}$.
Then for each null-solution $u$ of $L$ in a ball $B(x,R)$ {\rm (}where $x\in{\mathbb{R}}^n$ and $R>0${\rm )}, 
$p\in(0,\infty)$, $\lambda\in(0,1)$, $\ell\in{\mathbb{N}}_0$, and $r\in(0,R)$, one has
\begin{equation}\label{detraz}
\sup_{z\in B(x,\lambda r)}|(\nabla^\ell u)(z)|
\leq\frac{C}{r^\ell}\left(\aver{B(x,r)}|u|^p\,d{\mathscr{L}}^n\right)^{1/p},
\end{equation}
where $C=C(L,p,\ell,\lambda,n)>0$ is a finite constant.
\end{theorem}

Given $u:{\mathbb{R}}^n_{+}\to{\mathbb{C}}$ which is absolutely integrable over 
bounded Lebesgue measurable subsets of ${\mathbb{R}}^n_{+}$, define (whenever meaningful) the Sobolev 
trace
\begin{equation}\label{Veri-S2TG.3}
\big({\rm Tr}\,u\big)(x'):=\lim\limits_{r\to 0^{+}}\meanint_{B((x',0),r)\cap{\mathbb{R}}^n_{+}}
u\,d{\mathscr{L}}^n,\qquad x'\in{\mathbb{R}}^{n-1}. 
\end{equation}
Then for each $u\in W^{1,p}({\mathbb{R}}^n_{+})$, $1<p<\infty$, 
the trace ${\rm Tr}\,u$ exists a.e. on $\partial{\mathbb{R}}^n_{+}$ and 
belongs to $B^{p,p}_{1-1/p}({\mathbb{R}}^{n-1})$, 
where for each $p\in(1,\infty)$ and $s\in(0,1)$ the Besov space 
$B^{p,p}_s({\mathbb{R}}^{n-1})$ is defined as the collection of all measurable 
functions $f:{\mathbb{R}}^{n-1}\to{\mathbb{C}}$ with the property that
\begin{equation}\label{Veri-S2TG.4}
\|f\|_{B^{p,p}_s({\mathbb{R}}^{n-1})}:=\|f\|_{L^p({\mathbb{R}}^{n-1})}+\Big(\int_{{\mathbb{R}}^{n-1}}
\int_{{\mathbb{R}}^{n-1}}\frac{|f(x')-f(y')|^p}{|x'-y'|^{n-1+sp}}\,dx'dy'\Big)^{1/p}<+\infty.
\end{equation}
In fact, for each $p\in(1,\infty)$ the operator 
\begin{equation}\label{Veri-S2TG.5}
{\rm Tr}:W^{1,p}({\mathbb{R}}^n_{+})
\longrightarrow B^{p,p}_{1-1/p}({\mathbb{R}}^{n-1})
\end{equation}
is well defined, linear and bounded, and has a linear and bounded right-inverse. 
The proposition recorded below is a particular case of a more general result proved in \cite[Theorem~5.6, p.\,4372]{BMMM}.

\begin{proposition}\label{HgUV954.Ki-5}
For every $u\in W^{1,p}({\mathbb{R}}^n_{+})$ with $p\in(1,\infty)$ and every $M>1$ one has
\begin{equation}\label{Veri-S2TG.6}
\big({\rm Tr}\,u\big)(x')=
\lim\limits_{r\to 0^{+}}\meanint_{B((x',\,r),\,r/M)}u\,d{\mathscr{L}}^n
\,\,\text{ at a.e. }\,\,x'\in{\mathbb{R}}^{n-1}. 
\end{equation}
\end{proposition}

The following result, which may be found in \cite[Corollary~2.4]{MaMiSh}, is a 
consequence of the {\it a priori} regularity estimates obtained in \cite{ADNII}
and Sobolev embeddings.

\begin{proposition}\label{c1.2}
Let $L$ be an $M\times M$ system with constant complex coefficients as in \eqref{L-def}-\eqref{L-ell.X}.
Consider a function $u:{\mathbb{R}}^n_{+}\to{\mathbb{C}}^M$ with the property that 
$u\in\big[W^{1,2}\big({\mathbb{R}}^n_{+}\cap B(0,R)\big)\big]^M$ for each $R>0$ and such that
\begin{equation}\label{eJB-iY}
\left\{
\begin{array}{ll}
Lu=0 &\text{ in }\,\,\mathbb{R}^n_{+},
\\[6pt] 
{\rm Tr}\,u=f &\text{ on }\,\,\mathbb{R}^{n-1}, 
\end{array}
\right.
\end{equation}
for some $f\in\big[{\mathscr{C}}_{\rm loc}^{k+1}(\mathbb{R}^{n-1})\big]^M$,  
where $k\in{\mathbb{N}}$. Then for any $z\in\overline{\mathbb{R}^n_{+}}$ and $\rho>0$,
\begin{equation}\label{eq1.18}
\sup_{\mathbb{R}^n_{+}\cap B(z,\rho)}|\nabla^k u|\leq C\,\Big(\,\rho^{-k}\sup_{\mathbb{R}^n_{+}\cap B(z,2\rho)}|u|
+\sum_{\ell=0}^{k+1}\rho^{\ell-k}\sup_{\partial\mathbb{R}^n_{+}\cap B(z,2\rho)}|\nabla^\ell f|\Big),
\end{equation}
where $C\in(0,\infty)$ is a  constant independent of $\rho$, $z$, $u$ and $f$.
\end{proposition}

Below we record a version of the Divergence Theorem obtained in \cite{GHA.I}, 
which is particularly suitable for the purposes we have in mind. Stating it requires a few 
preliminaries which we dispense with first. We shall write $\mathcal{E}({\mathbb{R}}^n_{+})$
for the space of smooth functions in ${\mathbb{R}}^n_{+}$ equipped with the topology of uniform 
convergence on compact sets for derivatives of any order. Its dual space, $\mathcal{E}'({\mathbb{R}}^n_{+})$,  
may then be identified with the subspace of ${\mathcal{D}}'({\mathbb{R}}^n_{+})$ consisting of those 
distributions which are compactly supported. Hence,
\begin{equation}\label{TDanab.4r}
\mathcal{E}'({\mathbb{R}}^n_{+})\hookrightarrow\mathcal{D}'({\mathbb{R}}^n_{+})
\,\,\text{ and }\,\,L^1_{\rm loc}({\mathbb{R}}^n_{+})\hookrightarrow\mathcal{D}'({\mathbb{R}}^n_{+}).
\end{equation}
For each compact set $K\subset{\mathbb{R}}^n_{+}$, define
$\mathcal{E}'_K({\mathbb{R}}^n_{+}):=\big\{u\in\mathcal{E}'({\mathbb{R}}^n_{+})
:\,{\rm supp}\,u\subset K\big\}$ and consider
\begin{multline}\label{TDY-87764}
\mathcal{E}'_K({\mathbb{R}}^n_{+})+L^1({\mathbb{R}}^n_{+})
:=\big\{u\in\mathcal{D}'({\mathbb{R}}^n_{+}):\,\exists\,v_1\in
\mathcal{E}'_K({\mathbb{R}}^n_{+})\,\text{ and }\,
\exists\,v_2\in L^1({\mathbb{R}}^n_{+})
\\[4pt]
\text{such that }\,
u=v_1+v_2\,\text{ in }\,\mathcal{D}'({\mathbb{R}}^n_{+})\big\}.
\end{multline}
Also, introduce $\mathscr{C}_b^\infty({\mathbb{R}}^n_{+}):=
\mathscr{C}^\infty({\mathbb{R}}^n_{+})\cap L^\infty({\mathbb{R}}^n_{+})$
and let $\big(\mathscr{C}_b^\infty({\mathbb{R}}^n_{+})\big)^\ast$ denote its
algebraic dual. Moreover, we let
${}_{(\mathscr{C}_b^\infty({\mathbb{R}}^n_{+}))^\ast}\big\langle\cdot\,,\,\cdot
\big\rangle_{\mathscr{C}_b^\infty({\mathbb{R}}^n_{+})}$ denote the natural duality pairing
between these spaces. It is useful to observe that
for every compact set $K\subset{\mathbb{R}}^n_{+}$ one has
\begin{equation}\label{TDY-877ii}
\mathcal{E}'_K({\mathbb{R}}^n_{+})+L^1({\mathbb{R}}^n_{+})\subset
\big(\mathscr{C}_b^\infty({\mathbb{R}}^n_{+})\big)^\ast.
\end{equation}
As a consequence of the last part of \cite[Corollary~1.4.2, pp.\,40--41]{GHA.I} we have the following:

\begin{theorem}\label{theor:div-thm}
Assume that $K\subset{\mathbb{R}}^n_{+}$ is a compact set and that
$\vec{F}\in\big[L^1_{\rm loc}({\mathbb{R}}^n_{+})\big]^n$ is a vector
field satisfying the following conditions {\rm (}for some aperture parameter $\kappa>0${\rm )}:
\begin{enumerate}\itemsep=0.2cm
\item [(a)] ${\rm div}\,\vec{F}\in\mathcal{E}'_K({\mathbb{R}}^n_{+})+L^1({\mathbb{R}}^n_{+})$,
where the divergence is taken in the sense of distributions;
\item[(b)] the nontangential maximal function ${\mathcal{N}}_\kappa^{\,{\mathbb{R}}^n_{+}\setminus K}\vec{F}$ 
belongs to $L^1({\mathbb{R}}^{n-1})$;
\item [(c)] the nontangential boundary trace 
$\vec{F}\big|_{\partial{\mathbb{R}}^n_{+}}^{{}^{\kappa-{\rm n.t.}}}$ exists 
{\rm (}in ${\mathbb{C}}^n${\rm )} at ${\mathscr{L}}^{n-1}$-a.e. point in ${\mathbb{R}}^{n-1}$.
\end{enumerate}

\noindent Then, with $e_n:=(0,\dots,0,1)\in{\mathbb{R}}^n$ and ``dot" denoting the
standard inner product in ${\mathbb{R}}^n$,
\begin{equation}\label{eqn:div-form}
{}_{(\mathscr{C}_b^\infty({\mathbb{R}}^n_{+}))^\ast}\big\langle{\rm div}\,\vec{F},1
\big\rangle_{\mathscr{C}_b^\infty({\mathbb{R}}^n_{+})}
=-\int_{{\mathbb{R}}^{n-1}}e_n\cdot
\big(\vec F\,\big|^{{}^{\kappa-{\rm n.t.}}}_{\partial{\mathbb{R}}^n_{+}}\bigr)\,d{\mathscr{L}}^{n-1}.
\end{equation}
\end{theorem}

In turn, Theorem~\ref{theor:div-thm} is used to justify an integration by parts formula which 
is a key technical step in the proof of Theorem~\ref{ta.av-GGG.2A}.

\begin{lemma}\label{Lgav-TeD}
Assume that $L$ is an $M\times M$ constant complex coefficient system as in \eqref{L-def} and \eqref{Def-ELLa.INTR}, 
and fix an aperture parameter $\kappa>0$. Pick two points $x^\star_j\in{\mathbb{R}}^{n}_{+}$ 
and two radii $r_j\in(0,\infty)$, $j\in\{1,2\}$, satisfying 
\begin{equation}\label{CPTs}
r_j<{\rm dist}\,(x^\star_j,\partial{\mathbb{R}}^n_{+})\frac{\kappa}{\sqrt{1+\kappa^2}}
\,\,\text{ for }\,\,j\in\{1,2\},\,\,\text{ and }\,\,|x^\star_1-x^\star_2|>r_1+r_2, 
\end{equation}
then abbreviate $K_j:=\overline{B(x^\star_j,r_j)}$ for $j\in\{1,2\}$.

Next, consider two vector-valued functions
\begin{equation}\label{Tvav-TREINRv}
u=\big(u_{\alpha}\big)_{1\leq\alpha\leq M}\in\big[W^{1,1}_{\rm loc}({\mathbb{R}}^n_{+})\big]^M,\quad
v=\big(v_{\alpha}\big)_{1\leq\alpha\leq M}\in\big[W^{1,1}_{\rm loc}({\mathbb{R}}^n_{+})\big]^M
\end{equation}
satisfying 
\begin{align}\label{hBbac.1}
& {\mathcal{N}}_\kappa^{\,\mathbb{R}^{n}_{+}\setminus K_2}v,\,
{\mathcal{N}}_\kappa^{\,\mathbb{R}^{n}_{+}\setminus K_2}(\nabla v)\in L^1_{\rm loc}({\mathbb{R}}^{n-1}),\quad
L^{\top}v\in\big[{\mathcal{E}}'_{K_2}(\mathbb{R}^{n}_{+})\big]^M,
\\[4pt]
& v\big|^{{}^{\kappa-{\rm n.t.}}}_{\partial{\mathbb{R}}^n_{+}}=0
\,\,\text{ a.e. in }\,\,{\mathbb{R}}^{n-1},\quad
\big(\nabla v\big)\big|_{\partial\mathbb{R}^{n}_{+}}^{{}^{\kappa-{\rm n.t.}}}
\,\,\text{ exists a.e. in }\,\,{\mathbb{R}}^{n-1},
\label{hBbac.2}
\\[4pt]
\label{DiBBa.LLap}
& Lu\in\big[{\mathcal{E}}'_{K_1}(\mathbb{R}^{n}_{+})\big]^M,\quad
u\big|_{\partial\mathbb{R}^{n}_{+}}^{{}^{\kappa-{\rm n.t.}}}\,\,\text{ exists a.e. in }\,\,{\mathbb{R}}^{n-1},
\\[4pt]
& \text{and }\,\,
\big({\mathcal{N}}^{\,\mathbb{R}^{n}_{+}\setminus K_1}_\kappa u\big)\cdot
\big({\mathcal{N}}_\kappa^{\,\mathbb{R}^{n}_{+}\setminus K_2}(\nabla v)\big)\in L^1({\mathbb{R}}^{n-1}).
\label{DiBBa.LLap.BBBB}
\end{align}
Finally, select two scalar functions, 
\begin{equation}\label{Ua-eNBVVatGG}
\begin{array}{c}
\phi\in{\mathscr{C}}^\infty_c({\mathbb{R}}^n_{+})
\,\,\text{ such that $\phi\equiv 1$ near $K_1$ and $\phi\equiv 0$ near $K_2$},
\\[6pt]
\psi\in{\mathscr{C}}^\infty_c({\mathbb{R}}^n_{+})
\,\,\text{ such that $\psi\equiv 1$ near $K_2$ and $\psi\equiv 0$ near $K_1$}.
\end{array}
\end{equation}

Then
\begin{align}\label{Ua-eDPa4.L}
{}_{[{\mathcal{E}}'(\mathbb{R}^{n}_{+})]^M}\big\langle Lu,\phi v
\big\rangle_{[{\mathcal{E}}(\mathbb{R}^{n}_{+})]^M}
=&\,
{}_{[{\mathcal{E}}'(\mathbb{R}^{n}_{+})]^M}\big\langle L^{\top}v,
\psi u\big\rangle_{[{\mathcal{E}}(\mathbb{R}^{n}_{+})]^M}
\nonumber\\[4pt]
&\,+\int_{\mathbb{R}^{n-1}}
\Big(u_\alpha\big|_{\partial\mathbb{R}^{n}_{+}}^{{}^{\kappa-{\rm n.t.}}}\Big) 
a^{\beta\alpha}_{nn}\,\Big[\big(\partial_n v_\beta\big)\Big]
\Big|_{\partial\mathbb{R}^{n}_{+}}^{{}^{\kappa-{\rm n.t.}}}\,d{\mathscr{L}}^{n-1}.
\end{align}
\end{lemma}

\begin{proof}
Note that the ellipticity of $L^{\top}$ (entailed by that of $L$) and the 
fact that $v\in\big[L^1_{\rm loc}({\mathbb{R}}^n_{+})\big]^M$ satisfies
$L^{\top}v\in\big[{\mathcal{E}}'_{K_2}(\mathbb{R}^{n}_{+})\big]^M$ imply that 
\begin{equation}\label{Di-UPs-4}
v\in\big[{\mathscr{C}}^\infty(\mathbb{R}^{n}_{+}\setminus K_2)\big]^M\,\,\text{ and }\,\,
L^{\top}v=0\,\,\text{ pointwise in }\,\,\mathbb{R}^{n}_{+}\setminus K_2.
\end{equation}
Likewise, the conditions on $u$ give  
\begin{equation}\label{Di-UabTRF.a}
u\in\big[{\mathscr{C}}^\infty(\mathbb{R}^{n}_{+}\setminus K_1)\big]^M\,\,\text{ and }\,\,
Lu=0\,\,\text{ pointwise in }\,\,\mathbb{R}^{n}_{+}\setminus K_1.
\end{equation}
We shall first establish formula \eqref{Ua-eDPa4.L} under the additional assumption that 
\begin{equation}\label{Di-UabTRF.b}
\begin{array}{c}
\big({\mathcal{N}}_\kappa^{\,\mathbb{R}^{n}_{+}\setminus K_2}v\big)\cdot
\big({\mathcal{N}}^{\,\mathbb{R}^{n}_{+}\setminus K_1}_\kappa(\nabla u)\big)\in L^1({\mathbb{R}}^{n-1})
\,\,\text{ and }\,\,(\nabla u)\Big|_{\partial\mathbb{R}^{n}_{+}}^{{}^{\kappa-{\rm n.t.}}}\,\,
\text{ exists a.e. in }\,\,{\mathbb{R}}^{n-1}.
\end{array}
\end{equation}
Assuming that this is the case define 
\begin{equation}\label{Yab-HBBaR46}
\vec{F}:=\Big(v_{\alpha}\,a^{\alpha\beta}_{jk}\,\partial_k u_\beta
-u_\alpha a^{\beta\alpha}_{kj}\partial_k v_{\beta}\Big)_{1\le j\le n}
\,\,\,\text{ in }\,\,{\mathbb{R}}^n_{+}.
\end{equation}
Then \eqref{Yab-HBBaR46} entails 
\begin{equation}\label{Yab-HBAAA.h5}
\big|\vec{F}\big|\leq C\big(|v||\nabla u|+|u|\big|\nabla v\big|\big)
\,\,\text{ a.e. in }\,\,{\mathbb{R}}^n_{+}.
\end{equation}
Note that $|v|$ is locally absolutely integrable (by \eqref{Tvav-TREINRv}), 
while $|\nabla u|$ is bounded in a neighborhood of $K_2$ (by \eqref{Di-UabTRF.a} 
and the fact that $K_1\cap K_2=\varnothing$), hence the product $|v||\nabla u|$ is 
absolutely integrable in a neighborhood of $K_2$. Away from $K_2$, we have that 
$|v|$ is locally bounded (by \eqref{Di-UPs-4}), while $|\nabla u|$ is locally 
integrable (by \eqref{Tvav-TREINRv}). Thus, the product $|v||\nabla u|$ is also 
locally absolutely integrable away from $K_2$. In summary, 
$|v||\nabla u|\in L^1_{\rm loc}({\mathbb{R}}^n_{+})$. In a similar manner, we also 
obtain that $|u||\nabla v|\in L^1_{\rm loc}({\mathbb{R}}^n_{+})$, 
hence ultimately, 
\begin{equation}\label{Yab-HBAAA.aF}
\vec{F}\in\big[L^1_{\rm loc}({\mathbb{R}}^n_{+})\big]^n.
\end{equation}
In particular, it makes sense to consider ${\rm div}\vec{F}$ in the sense of 
distributions in ${\mathbb{R}}^n_{+}$. In fact, a direct calculation gives
\begin{equation}\label{div-F.Req}
{\rm div}\vec{F}=-u_{\alpha}\,(L^{\top}v)_\alpha+v_{\alpha}\,(Lu)_\alpha
\,\,\,\text{ in }\,\,{\mathcal{D}}'({\mathbb{R}}^n_{+}).
\end{equation}
Consequently, if $K_\star:=K_1\cup K_2$ then $K_\star$ is a compact 
subset of ${\mathbb{R}}^n_{+}$ and, as seen from \eqref{div-F.Req}, the last membership in \eqref{hBbac.1},
the membership in \eqref{DiBBa.LLap}, \eqref{Di-UPs-4}-\eqref{Di-UabTRF.a}, we have
\begin{equation}\label{div-F.Req.3}
{\rm div}\vec{F}\in{\mathcal{E}}'_{K_\star}({\mathbb{R}}^n_{+}). 
\end{equation}
Furthermore, \eqref{Yab-HBAAA.h5} gives
\begin{equation}\label{Yab-HBna-TAn}
{\mathcal{N}}^{\,\mathbb{R}^{n}_{+}\setminus K_\star}_{\kappa}\vec{F}
\leq C\Big({\mathcal{N}}^{\,\mathbb{R}^{n}_{+}\setminus K_2}_{\kappa}v\cdot
{\mathcal{N}}^{\,\mathbb{R}^{n}_{+}\setminus K_1}_{\kappa}(\nabla u) 
+{\mathcal{N}}^{\,\mathbb{R}^{n}_{+}\setminus K_1}_{\kappa}u
\cdot{\mathcal{N}}^{\,\mathbb{R}^{n}_{+}\setminus K_2}_{\kappa}(\nabla v)\Big)
\,\,\text{ in }\,\,{\mathbb{R}}^{n-1}.
\end{equation}
From \eqref{Yab-HBna-TAn}, the first working assumption made in \eqref{Di-UabTRF.b}, and \eqref{DiBBa.LLap.BBBB}
we therefore obtain (also bearing in mind that the nontangential maximal function is lower-semicontinuous, hence 
measurable)
\begin{equation}\label{Yab-HBna-TAn.2}
{\mathcal{N}}^{\,\mathbb{R}^{n}_{+}\setminus K_\star}_{\kappa}\vec{F}\in L^1({\mathbb{R}}^{n-1}).
\end{equation}
Finally, from \eqref{Yab-HBBaR46}, \eqref{hBbac.2}, the second working assumption made in \eqref{Di-UabTRF.b}, 
and the middle condition in \eqref{DiBBa.LLap}, we see that
$\vec{F}\,\big|^{{}^{\kappa-{\rm n.t.}}}_{\partial\mathbb{R}^n_+}$
exists a.e. in ${\mathbb{R}}^{n-1}$ and, in fact, 
\begin{equation}\label{KnaJHcwkfc}
\vec{F}\,\Big|^{{}^{\kappa-{\rm n.t.}}}_{\partial\mathbb{R}^n_+}
=-\Big(\Big(u_\alpha\Big|^{{}^{\kappa-{\rm n.t.}}}_{\partial\mathbb{R}^n_+}\Big)
a^{\beta\alpha}_{kj}\Big(\big(\partial_k v_{\beta}\big)
\Big|^{{}^{\kappa-{\rm n.t.}}}_{\partial\mathbb{R}^n_+}\Big)_{1\leq j\le n}
\,\,\text{ a.e. on }\,\,{\mathbb{R}}^{n-1}.
\end{equation}
Then Theorem~\ref{theor:div-thm} applies and on account of \eqref{div-F.Req}
and \eqref{KnaJHcwkfc} yields, for any two given functions $\psi,\phi$ 
as in \eqref{Ua-eNBVVatGG},
\begin{align}\label{Uahg-ihGPPa4}
&\hskip -0.40in
-{}_{[{\mathcal{E}}'(\mathbb{R}^{n}_{+})]^M}\big\langle L^{\top}v,
\psi u\big\rangle_{[{\mathcal{E}}(\mathbb{R}^{n}_{+})]^M}
+{}_{[{\mathcal{E}}'(\mathbb{R}^{n}_{+})]^M}\big\langle Lu\,,\,
\phi\,v\big\rangle_{[{\mathcal{E}}(\mathbb{R}^{n}_{+})]^M}
\nonumber\\[4pt]
&\hskip 0.80in
={}_{(\mathscr{C}_b^\infty(\mathbb{R}^n_+))^*}\big\langle{\rm div}\vec{F},1
\big\rangle_{\mathscr{C}_b^\infty(\mathbb{R}^n_+)} 
=-\int_{\mathbb{R}^{n-1}}e_n\cdot
\big(\vec F\,\big|^{{}^{\kappa-{\rm n.t.}}}_{\partial\mathbb{R}^n_+}\bigr)\,d{\mathscr{L}}^{n-1}
\nonumber\\[4pt]
&\hskip 0.80in
=\int_{\mathbb{R}^{n-1}}
\Big(u_\alpha\big|_{\partial\mathbb{R}^{n}_{+}}^{{}^{\kappa-{\rm n.t.}}}\Big)a^{\beta\alpha}_{kn}
\Big[\big(\partial_k v_{\beta}\big)
\Big]\Big|_{\partial\mathbb{R}^{n}_{+}}^{{}^{\kappa-{\rm n.t.}}}\,d{\mathscr{L}}^{n-1}.
\end{align}

At this stage, we claim that 
\begin{equation}\label{Ua-nabUGC4D}
\begin{array}{c}
\Big[\big(\partial_k v_\beta\big)\Big]
\Big|_{\partial\mathbb{R}^{n}_{+}}^{{}^{\kappa-{\rm n.t.}}}=0
\quad\text{ a.e. on }\,\,\partial{\mathbb{R}}^n_{+}\equiv\mathbb{R}^{n-1},
\\[10pt]
\forall\,k\in\{1,\dots,n-1\}\,\,\text{ and }\,\,\forall\,\beta\in\{1,\dots,M\}.
\end{array}
\end{equation}
To see that this is the case, fix $k\in\{1,\dots,n-1\}$ along with 
$\beta\in\{1,\dots,M\}$, and pick an arbitrary function 
\begin{equation}\label{Ua-nabUGC4D.2}
\varphi\in{\mathscr{C}}^\infty_c({\mathbb{R}}^n)\,\,\text{ such that }\,\,
K_2\cap{\rm supp}\,\varphi=\emptyset. 
\end{equation}
Next, consider the vector field
\begin{equation}\label{Ua-nabUGC4D.3}
\vec{H}:=\Big(\big(\partial_k v_\beta\big)\varphi+v_\beta\partial_k\varphi\Big)e_n
-\Big(\big(\partial_n v_{\beta}\big)\varphi+v_{\beta}\partial_n\varphi\Big)e_k
\,\,\,\text{ in }\,\,{\mathbb{R}}^n_{+}.
\end{equation}
Note that 
\begin{equation}\label{Ua-nabUGC4D.4}
\vec{H}\in\big[{\mathscr{C}}^\infty({\mathbb{R}}^n_{+})\big]^n\,\,\text{ and }\,\,
{\rm div}\vec{H}=0\,\,\text{ in }\,\,{\mathcal{D}}'({\mathbb{R}}^n_{+}),
\end{equation}
by \eqref{Ua-nabUGC4D.3}, \eqref{Ua-nabUGC4D.2}, and \eqref{Di-UPs-4}. 
Also, since ${\mathcal{N}}_{\kappa}\varphi$, ${\mathcal{N}}_{\kappa}(\nabla\varphi)$ 
are bounded functions with compact support in ${\mathbb{R}}^{n-1}$, we have
\begin{equation}\label{Ua-nabUGC4D.5}
{\mathcal{N}}_{\kappa}\vec{H}\in L^1({\mathbb{R}}^{n-1})
\end{equation}
by \eqref{Ua-nabUGC4D.3}, \eqref{Ua-nabUGC4D.2}, and \eqref{hBbac.1}.
In addition, thanks to \eqref{hBbac.2}, 
\begin{equation}\label{Ua-nabUGC4D.6}
\vec{H}\Big|_{\partial\mathbb{R}^{n}_{+}}^{{}^{\kappa-{\rm n.t.}}}
=\big(\varphi\big|_{\partial\mathbb{R}^{n}_{+}}\big)
\Big[\big(\partial_k v_{\beta}\big)\Big]
\Big|_{\partial\mathbb{R}^{n}_{+}}^{{}^{\kappa-{\rm n.t.}}}\,e_n
-\big(\varphi\big|_{\partial\mathbb{R}^{n}_{+}}\big)
\Big[\big(\partial_n v_{\beta}\big)\Big]
\Big|_{\partial\mathbb{R}^{n}_{+}}^{{}^{\kappa-{\rm n.t.}}}\,e_k,
\end{equation}
at a.e. point on ${\mathbb{R}}^{n-1}$.
Granted this and keeping in mind that $k\in\{1,\dots,n-1\}$, 
we therefore obtain 
\begin{equation}\label{Ua-nabUGC4D.7}
e_n\cdot\Big(\vec{H}\Big|_{\partial\mathbb{R}^{n}_{+}}^{{}^{\kappa-{\rm n.t.}}}\Big)
=\big(\varphi\big|_{\partial\mathbb{R}^{n}_{+}}\big)
\Big[\big(\partial_k v_{\beta}\big)\Big]
\Big|_{\partial\mathbb{R}^{n}_{+}}^{{}^{\kappa-{\rm n.t.}}}
\,\,\text{ a.e. in }\,\,{\mathbb{R}}^{n-1}.
\end{equation}
Collectively, \eqref{Ua-nabUGC4D.4}-\eqref{Ua-nabUGC4D.6} ensure that 
Theorem~\ref{theor:div-thm} is applicable to the vector field $\vec{H}$. 
Based on the second condition in \eqref{Ua-nabUGC4D.4} and \eqref{Ua-nabUGC4D.7} 
we may therefore write 
\begin{align}\label{Phgav-Y54V}
0 & ={}_{(\mathscr{C}_b^\infty(\mathbb{R}^n_+))^*}\big\langle{\rm div}\vec{H},1
\big\rangle_{\mathscr{C}_b^\infty(\mathbb{R}^n_+)} 
=-\int_{\mathbb{R}^{n-1}}e_n\cdot
\big(\vec{H}\,\big|^{{}^{\kappa-{\rm n.t.}}}_{\partial\mathbb{R}^n_+}\big)\,d{\mathscr{L}}^{n-1}
\nonumber\\[4pt]
&=-\int_{\mathbb{R}^{n-1}}\big(\varphi\big|_{\partial\mathbb{R}^{n}_{+}}\big)
\Big[\big(\partial_k v_{\beta}\big)
\Big]\Big|_{\partial\mathbb{R}^{n}_{+}}^{{}^{\kappa-{\rm n.t.}}}\,d{\mathscr{L}}^{n-1}.
\end{align}
On the other hand, 
\begin{equation}\label{uag-tREW.1}
\Big\{\varphi\big|_{\partial\mathbb{R}^{n}_{+}}:\,
\varphi\in{\mathscr{C}}^\infty_c({\mathbb{R}}^n)\,\,
\text{ such that }\,\,K_2\cap{\rm supp}\,\varphi=\varnothing\Big\}
={\mathscr{C}}^\infty_c({\mathbb{R}}^{n-1})
\end{equation}
while from \eqref{hBbac.1}-\eqref{hBbac.2} we have 
\begin{equation}\label{ubav-KBav4fE}
\Big[\big(\partial_k v_{\beta}\big)
\Big]\Big|_{\partial\mathbb{R}^{n}_{+}}^{{}^{\kappa-{\rm n.t.}}}\in L^1_{\rm loc}({\mathbb{R}}^{n-1}).
\end{equation}
Together, \eqref{ubav-KBav4fE}, \eqref{Phgav-Y54V} and \eqref{uag-tREW.1} 
ultimately prove that $\big[\big(\partial_k v_{\beta}\big)\big]
\big|_{\partial\mathbb{R}^{n}_{+}}^{{}^{\kappa-{\rm n.t.}}}=0$ a.e. in ${\mathbb{R}}^{n-1}$, 
finishing the justification of the claim made in \eqref{Ua-nabUGC4D}. 
In turn, \eqref{Uahg-ihGPPa4} and \eqref{Ua-nabUGC4D} prove \eqref{Ua-eDPa4.L} 
under the additional assumption \eqref{Di-UabTRF.b}. 

To dispense with \eqref{Di-UabTRF.b}, fix $\phi,\psi$ as in \eqref{Ua-eNBVVatGG}.
For each $\varepsilon>0$ define $u^\varepsilon=(u^\varepsilon_\alpha)_{1\leq\alpha\leq M}$ by 
\begin{equation}\label{Nvav-T54.a}
u^\varepsilon(x):=u(x+\varepsilon e_n),\qquad\forall\,x\in{\mathbb{R}}^n_{+}.
\end{equation}
Then $u^\varepsilon\in\big[W^{1,1}_{\rm loc}({\mathbb{R}}^n_{+})\big]^{M}$ and $u^\varepsilon$ extends to a  
${\mathscr{C}}^\infty$ function in a two-sided neighborhood of $\partial{\mathbb{R}}^n_{+}$ provided 
$\varepsilon>0$ is small enough. In particular, both 
$u^\varepsilon\big|_{\partial\mathbb{R}^{n}_{+}}^{{}^{\kappa-{\rm n.t.}}}$ and 
$(\nabla u^\varepsilon)\big|_{\partial\mathbb{R}^{n}_{+}}^{{}^{\kappa-{\rm n.t.}}}$ exist a.e. in ${\mathbb{R}}^{n-1}$.
Also, if $0<\varepsilon<{\rm dist}\,(K_1,\partial{\mathbb{R}}^n_{+})$ it follows that 
$K_1-\varepsilon e_n$ is a compact subset of ${\mathbb{R}}^n_{+}$ and 
$Lu^\varepsilon\in\big[{\mathcal{E}}'_{K_1-\varepsilon e_n}(\mathbb{R}^{n}_{+})\big]^M$.

Moving on, fix $\kappa_o\in(0,\kappa)$ such that 
\begin{equation}\label{CPTs.i}
r_j<{\rm dist}\,(x^\star_j,\partial{\mathbb{R}}^n_{+})\frac{\kappa_o}{\sqrt{1+\kappa^2_o}}
\,\,\text{ for }\,\,j\in\{1,2\}.
\end{equation}
Let us now choose $x'\in{\mathbb{R}}^{n-1}\equiv\partial{\mathbb{R}}^n_{+}$ 
with the property that $\big(v\big|^{{}^{\kappa-{\rm n.t.}}}_{\partial{\mathbb{R}}^n_{+}}\big)(x')=0$.
We claim that 
\begin{equation}\label{Cr654.A-HM.1}
|v(y)|\leq Cy_n\cdot\big({\mathcal{N}}^{\,\mathbb{R}^{n}_{+}\setminus K_2}_{\kappa}(\nabla v)\big)(x')
\,\,\text{ for each }\,\,y=(y',y_n)\in\Gamma_{\kappa_o}(x')\setminus K_2.
\end{equation}
To see this, let $y$ be as in \eqref{Cr654.A-HM.1} and choose a rectifiable path 
$\gamma:[0,1]\to\overline{{\mathbb{R}}^n_{+}}$ joining $(x',0)$ with $y$, whose length is $\leq Cy_n$, 
and such that $\gamma((0,1))\subseteq\Gamma_{\kappa_o}(x')\setminus K_2$.
Then, for some constant $C\in(0,\infty)$ independent of $x'$ and $y$, we may estimate 
\begin{align}\label{Cr654.A-HM.1.xxx}
|v(y)| &=\Big|v(y)-\big(v\big|^{{}^{\kappa-{\rm n.t.}}}_{\partial{\mathbb{R}}^n_{+}}\big)(x')\Big|
=\Big|\int_0^1\frac{d}{dt}[v(\gamma(t))]\,dt\Big|
\nonumber\\[4pt]
&=\Big|\int_0^1(\nabla v)(\gamma(t))\cdot\gamma'(t)\,dt\Big|
\leq\Big(\sup_{\xi\in\gamma((0,1))}|(\nabla v)(\xi)|\Big)\int_0^1|\gamma'(t)|\,dt
\nonumber\\[4pt]
&\leq Cy_n\cdot\big({\mathcal{N}}^{\,\mathbb{R}^{n}_{+}\setminus K_2}_{\kappa_o}(\nabla v)\big)(x')
\leq Cy_n\cdot\big({\mathcal{N}}^{\,\mathbb{R}^{n}_{+}\setminus K_2}_{\kappa}(\nabla v)\big)(x'),
\end{align}
using the Fundamental Theorem of Calculus, Chain Rule, and \eqref{NT-Fct.23}. This establishes \eqref{Cr654.A-HM.1}.

For the remainder of the proof assume that the parameter $\varepsilon>0$ is sufficiently small, say 
$\varepsilon<\frac{1}{2}\,{\rm dist}\,(K_1,\partial{\mathbb{R}}^n_{+})$, and define 
\begin{align}\label{Cr654.A-HM.2-FDS.aTgab}
K_1^\varepsilon:=\{z\in{\mathbb{R}}^n:{\rm dist}\,(z,K_1-\varepsilon e_n)\leq\varepsilon\}.
\end{align}
This is a compact subset of ${\mathbb{R}}^n_{+}$. Next, choose some sufficiently small 
parameter $a\in(0,1)$. Specifically, start with some parameter $a$ satisfying 
\begin{align}\label{Cr654.A-HM.2-FDS.aaa}
0<a<\frac{\kappa-\kappa_o}{1+\kappa}
\end{align}
and which is small enough so that 
\begin{align}\label{Cr654.A-HM.2-FDS.a344}
B(y,a\cdot y_n)\cap(K_1-\varepsilon e_n)=\varnothing\,\,\text{ for each }\,\,
y\in{\mathbb{R}}^n_{+}\setminus K_1^\varepsilon.
\end{align}
To see why this may be arranged, pick $R:=2\,\sup\{z_n:\,z=(z',z_n)\in K_1-\varepsilon e_n\}$
then take $a<\min\big\{1/2\,,\,\varepsilon/R\big\}$. Pick an arbitrary point 
$y=(y',y_n)\in{\mathbb{R}}^n_{+}\setminus K_1^\varepsilon$. On the one hand, if $y_n\geq R$ then 
$B(y,a\cdot y_n)$ lies above the strip $\{z=(z',z_n)\in{\mathbb{R}}^n_{+}:z_n<(1-a)R\}$ 
and the latter strip contains the set $K_1-\varepsilon e_n$. Thus, the disjointness condition in 
\eqref{Cr654.A-HM.2-FDS.a344} holds in this case. On the other hand, if $y_n<R$ then 
$a\cdot y_n<a\cdot R<\varepsilon$, and since $B(y,\varepsilon)$ does not intersect 
$K_1-\varepsilon e_n$ (given that ${\rm dist}\,(y,K_1-\varepsilon e_n)>\varepsilon$
since $y\in{\mathbb{R}}^n_{+}\setminus K_1^\varepsilon$), the disjointness condition
in \eqref{Cr654.A-HM.2-FDS.a344} holds in this case as well. 

Having chosen $a$ as in \eqref{Cr654.A-HM.2-FDS.aaa} so that \eqref{Cr654.A-HM.2-FDS.a344} holds, we now fix an arbitrary point 
$y\in\Gamma_{\kappa_o}(x')\setminus K_1^\varepsilon$. Observe that having $z=(z',z_n)\in B(y,a\cdot y_n)$ entails 
\begin{align}\label{Cr654.A-HM.2-FDS.1}
y_n\leq z_n+|z-y|<z_n+a\cdot y_n\Longrightarrow y_n<(1-a)^{-1}z_n,
\end{align}
which, bearing in mind the current location of $y$, permits us to conclude that  
\begin{align}\label{Cr654.A-HM.2-FDS.2}
|z'-x'| &\leq|z'-y'|+|y'-x'|\leq|z-y|+\kappa_o\cdot y_n<a\cdot y_n+\kappa_o\cdot y_n
\nonumber\\[6pt]
&=(\kappa_o+a)y_n<\frac{\kappa_o+a}{1-a}z_n<\kappa z_n,
\end{align}
where the last inequality is guaranteed by \eqref{Cr654.A-HM.2-FDS.aaa}.
From \eqref{Cr654.A-HM.2-FDS.2} we see that $z\in\Gamma_{\kappa}(x')$.
Together with \eqref{Cr654.A-HM.2-FDS.a344}, this proves that 
\begin{align}\label{Cr654.A-HM.2-FDS.a344.ag}
B(y,a\cdot y_n)\subseteq\Gamma_{\kappa}(x')\setminus(K_1-\varepsilon e_n)\,\,\text{ for each }\,\,
y\in\Gamma_{\kappa_o}(x')\setminus K_1^\varepsilon.
\end{align}
Using interior estimates (cf. Theorem~\ref{ker-sbav}) for the function $u^\varepsilon$ we may then write 
\begin{align}\label{Cr654.A-HM.2-bis}
|(\nabla u^\varepsilon)(y)| &\leq\frac{C}{y_n}\meanint_{B(y,a\cdot y_n)}|u^\varepsilon(z)|\,dz
\nonumber\\[6pt]
&\leq Cy_n^{-1}\cdot
\sup_{z\in\Gamma_{\kappa}(x')\setminus(K_1-\varepsilon e_n)}|u^\varepsilon(z)|
\nonumber\\[6pt]
&\leq Cy_n^{-1}\cdot\big({\mathcal{N}}_{\kappa}^{\,\mathbb{R}^{n}_{+}\setminus (K_1-\varepsilon e_n)}u^\varepsilon\big)(x')
\nonumber\\[6pt]
&\leq Cy_n^{-1}\cdot\big({\mathcal{N}}_{\kappa}^{\,\mathbb{R}^{n}_{+}\setminus K_1}u\big)(x'),
\end{align}
where $C=C(L,n,a)\in(0,\infty)$. Combining \eqref{Cr654.A-HM.1} with \eqref{Cr654.A-HM.2-bis} then gives
\begin{align}\label{Cr654.A-HM.2-b22}
\big({\mathcal{N}}^{\,\mathbb{R}^{n}_{+}\setminus (K_1^\varepsilon\cup K_2)}_{\kappa_o}
(|v||\nabla u^\varepsilon|)\big)(x')
\leq C\big({\mathcal{N}}^{\,\mathbb{R}^{n}_{+}\setminus K_2}_{\kappa}(\nabla v)\big)(x')
\big({\mathcal{N}}^{\,\mathbb{R}^{n}_{+}\setminus K_1}_{\kappa}u\big)(x').
\end{align}
Also, it is clear that 
\begin{align}\label{Cr654.A-HM.2-b22.bb}
\big({\mathcal{N}}^{\,\mathbb{R}^{n}_{+}\setminus (K_1^\varepsilon\cup K_2)}_{\kappa}(|\nabla v||u^\varepsilon|)\big)(x')
\leq C\big({\mathcal{N}}^{\,\mathbb{R}^{n}_{+}\setminus K_2}_{\kappa}(\nabla v)\big)(x')
\big({\mathcal{N}}^{\,\mathbb{R}^{n}_{+}\setminus K_1}_{\kappa}u\big)(x').
\end{align}
In concert with \eqref{DiBBa.LLap.BBBB}, these ultimately show that if we now define 
\begin{equation}\label{Yab-HBBaR46-bis}
\vec{F}^\varepsilon:=\Big(v_{\alpha}\,a^{\alpha\beta}_{jk}\,\partial_k u^\varepsilon_\beta
-u^\varepsilon_\alpha a^{\beta\alpha}_{kj}\partial_k v_{\beta}\Big)_{1\le j\le n}
\,\,\,\text{ in }\,\,{\mathbb{R}}^n_{+},
\end{equation}
then 
\begin{equation}\label{Yab-HBna-TAn.2-bis}
{\mathcal{N}}^{\,\mathbb{R}^{n}_{+}\setminus(K_1^\varepsilon\cup K_2)}_{\kappa_o}\vec{F}^\varepsilon\in L^1({\mathbb{R}}^{n-1}).
\end{equation}
Granted this, the same type of argument as in the first part of the proof 
relying on Theorem~\ref{theor:div-thm} (now applied to $\vec{F}^\varepsilon$ in place of $\vec{F}$, with 
$K_1^\varepsilon\cup K_2$ in place of $K_\star$, and with for the aperture parameter $\kappa_o$) gives that
\begin{align}\label{Ua-eDPa4.LAC}
{}_{[{\mathcal{E}}'(\mathbb{R}^{n}_{+})]^M}\big\langle Lu^\varepsilon,\phi v
\big\rangle_{[{\mathcal{E}}(\mathbb{R}^{n}_{+})]^M}
=&\,{}_{[{\mathcal{E}}'(\mathbb{R}^{n}_{+})]^M}\big\langle L^{\top}v,
\psi u^\varepsilon\big\rangle_{[{\mathcal{E}}(\mathbb{R}^{n}_{+})]^M}
\nonumber\\[4pt]
&\,+\int_{\mathbb{R}^{n-1}}
\Big(u^\varepsilon_\alpha\big|_{\partial\mathbb{R}^{n}_{+}}^{{}^{\kappa-{\rm n.t.}}}\Big) 
a^{\beta\alpha}_{nn}\,\Big[\big(\partial_n v_\beta\big)\Big]
\Big|_{\partial\mathbb{R}^{n}_{+}}^{{}^{\kappa-{\rm n.t.}}}\,d{\mathscr{L}}^{n-1}.
\end{align}

There remains to eliminate $\varepsilon$. To this end, observe that 
\begin{equation}\label{Ua-eDPa4.LAC.2}
{}_{[{\mathcal{E}}'(\mathbb{R}^{n}_{+})]^M}\big\langle Lu^\varepsilon,\phi v
\big\rangle_{[{\mathcal{E}}(\mathbb{R}^{n}_{+})]^M}
={}_{[{\mathcal{E}}'(\mathbb{R}^{n}_{+})]^M}\big\langle 
Lu\,,\,\phi(\cdot-\varepsilon e_n)v(\cdot-\varepsilon e_n)
\big\rangle_{[{\mathcal{E}}(\mathbb{R}^{n}_{+})]^M},
\end{equation}
and that (as seen from the smoothness properties of the functions involved)
\begin{equation}\label{naIaPa.2}
\lim\limits_{\varepsilon\to 0^{+}}
\big[\phi(\cdot-\varepsilon e_n)v(\cdot-\varepsilon e_n)\big]=\phi\,v
\,\,\text{ and }\,\,\lim\limits_{\varepsilon\to 0^{+}}\big[\psi u^\varepsilon\big]
=\psi\,u\,\,\text{ in }\,\,{\mathcal{E}}(\mathbb{R}^{n}_{+}).
\end{equation}
Hence, 
\begin{equation}\label{naIaPa.2.iiia}
\lim\limits_{\varepsilon\to 0^{+}}
{}_{[{\mathcal{E}}'(\mathbb{R}^{n}_{+})]^M}\big\langle Lu^\varepsilon,\phi v
\big\rangle_{[{\mathcal{E}}(\mathbb{R}^{n}_{+})]^M}
={}_{[{\mathcal{E}}'(\mathbb{R}^{n}_{+})]^M}\big\langle Lu,\phi v
\big\rangle_{[{\mathcal{E}}(\mathbb{R}^{n}_{+})]^M}
\end{equation}
and
\begin{equation}\label{naIaPa.2.iiib}
\lim\limits_{\varepsilon\to 0^{+}}
{}_{[{\mathcal{E}}'(\mathbb{R}^{n}_{+})]^M}\big\langle L^{\top}v,
\psi u^\varepsilon\big\rangle_{[{\mathcal{E}}(\mathbb{R}^{n}_{+})]^M}
={}_{[{\mathcal{E}}'(\mathbb{R}^{n}_{+})]^M}\big\langle L^{\top}v,
\psi u\big\rangle_{[{\mathcal{E}}(\mathbb{R}^{n}_{+})]^M}.
\end{equation}
Moreover, we have $|u(\cdot+\varepsilon e_n)|\leq{\mathcal{N}}^{\,\mathbb{R}^{n}_{+}\setminus K_1}_\kappa u$ 
on $\mathbb{R}^{n-1}$, and for each $\alpha\in\{1,\dots,M\}$ the last condition in \eqref{DiBBa.LLap} implies that 
\begin{equation}\label{naIaPa.3}
\lim\limits_{\varepsilon\to 0^{+}}u_\alpha(\cdot+\varepsilon e_n)
=u_\alpha\big|_{\partial\mathbb{R}^{n}_{+}}^{{}^{\kappa-{\rm n.t.}}}
\,\,\text{ a.e. in }\,\,\mathbb{R}^{n-1}.
\end{equation}
Consequently, formula \eqref{Ua-eDPa4.L} follows by letting $\varepsilon\to 0^{+}$
in \eqref{Ua-eDPa4.LAC}, on account of \eqref{Ua-eDPa4.LAC.2}-\eqref{naIaPa.3}, 
\eqref{DiBBa.LLap.BBBB}, and Lebesgue's Dominated Convergence Theorem.
\end{proof}

We conclude this section by establishing a regularity result which is going to play an important 
role in the proof of Theorem~\ref{ta.av-GGG.2A}.

\begin{lemma}\label{UYfa-uVdssk}
Assume that $L$ is an $M\times M$ system with constant complex coefficients as in \eqref{L-def}-\eqref{L-ell.X},
and recall the Agmon-Douglis-Nirenberg kernel $K^L:{\mathbb{R}}^n_{+}\to{\mathbb{C}}^{M\times M}$ associated with $L$
as in Theorem~\ref{ya-T4-fav}. Also, suppose $\psi\in\big[{\mathscr{C}}^\infty({\mathbb{R}}^{n-1})\big]^M$ is a ${\mathbb{C}}^{M}$-valued 
function with the property that there exists $N>-1$ such that for each multi-index
$\gamma'\in{\mathbb{N}}_0^{n-1}$ one can find $C_{\gamma'}\in(0,\infty)$ for which 
\begin{equation}\label{TFFV-b54EE.1}
\big|\big(\partial^{\gamma'}\psi\big)(z')\big|\leq C_{\gamma'}\big[|z'|+1]^{-N-|\gamma'|},
\qquad\forall\,z'\in{\mathbb{R}}^{n-1}.
\end{equation}
Finally, define 
\begin{align}\label{eq1.6Adaf.Rd}
u(x',t):=(P^L_t\ast\psi)(x')=\int_{{\mathbb{R}}^{n-1}}K^L(x'-y',t)\psi(y')\,dy',\qquad
\forall\,(x',t)\in{\mathbb{R}}^n_{+}.
\end{align}

Then 
\begin{equation}\label{Lnabv-ufwhJ}
u\in\big[{\mathscr{C}}^\infty\big(\overline{{\mathbb{R}}^n_{+}}\,\big)\big]^{M}.
\end{equation}
\end{lemma}

\begin{proof}
Fix $\psi$ as in the statement, and select some scalar-valued function 
$\theta\in{\mathscr{C}}^\infty({\mathbb{R}}^{n-1})$ such that 
\begin{equation}\label{Lnabv-ufwhJ.2}
\theta\equiv 1\,\,\text{ on }\,\,B_{n-1}(0',1)\,\,\text{ and }\,\,
{\rm supp}\,\theta\subseteq B_{n-1}(0',2).
\end{equation}
For each $R\in[1,\infty)$, define $\theta_R:{\mathbb{R}}^{n-1}\to{\mathbb{C}}$ 
by setting $\theta_R(x'):=\theta(x'/R)$ for each $x'\in{\mathbb{R}}^{n-1}$. 
As a consequence,
\begin{equation}\label{Lnabv-ufwhJ.3}
\lim\limits_{R\to\infty}\theta_R(x')=1\,\,\text{ for each }\,\,x'\in{\mathbb{R}}^{n-1}.
\end{equation}
To proceed, for each $R\in[1,\infty)$ introduce $\psi_R:=\theta_R\psi$ and note that
\begin{equation}\label{Lnabv-ufwhJ.3BIs}
\psi_R\in{\mathscr{C}}^\infty_c({\mathbb{R}}^{n-1}).
\end{equation}
In addition, for each $\alpha'\in{\mathbb{N}}_0^{n-1}$ with $|\alpha'|>0$ we have 
\begin{align}\label{Lnabv-ufwhJ.4}
\big(\partial^{\alpha'}\psi_R\big)(x')=&\,\big(\partial^{\alpha'}\psi\big)(x')\theta_R(x')
\\[4pt]
&\,+\sum_{\stackrel{\beta'+\gamma'=\alpha'}{|\beta'|>0}}
\frac{\alpha'!}{\beta'!\gamma'!}R^{-|\beta'|}(\partial^{\beta'}\theta)(x'/R)
(\partial^{\gamma'}\psi)(x'),\qquad\forall\,x'\in{\mathbb{R}}^{n-1}.
\nonumber
\end{align}
In concert with \eqref{Lnabv-ufwhJ.3} this readily implies that for each $\alpha\,'\in{\mathbb{N}}_0^{n-1}$,
\begin{equation}\label{Lnabv-ufwhJ.5}
\lim_{R\to\infty}\big(\partial^{\alpha\,'}\psi_R\big)(x')
=\big(\partial^{\alpha\,'}\psi\big)(x'),\qquad\forall\,x'\in{\mathbb{R}}^{n-1}.
\end{equation}
Next, observe that $R\leq |x'|\leq 2R$ whenever $x'\in{\mathbb{R}}^{n-1}$ and 
$\beta'\in{\mathbb{N}}_0^{n-1}$ are such that $|\beta'|>0$ and 
$(\partial^{\beta'}\theta)(x'/R)\not=0$. Based on this, \eqref{Lnabv-ufwhJ.4}, 
\eqref{TFFV-b54EE.1}, and the properties of $\theta$, we deduce that 
the following decay condition holds:
\begin{equation}\label{Lnabv-ufwhJ.6}
\begin{array}{c}
\text{$\forall\,\alpha\,'\in{\mathbb{N}}_0^{n-1}$ there 
exists $C_{\theta,\alpha\,'}\in(0,\infty)$ so that for each $R\in[1,\infty)$}
\\[6pt]
\big|\big(\partial^{\alpha\,'}\psi_R\big)(x')\big|
\leq C_{\theta,\alpha\,'}\big[|x'|+1\big]^{-N-|\alpha\,'|},
\qquad\forall\,x'\in{\mathbb{R}}^{n-1}.
\end{array}
\end{equation}

Switching gears, we shall need a structure result, proved in 
\cite[pp.\,56--57, and Lemma~4.1 on p.\,58]{ADNII} 
(cf. also \cite[pp.\,636--637]{ADNI}), to the effect that for each 
\begin{equation}\label{ha-GVVC-v6bu}
\text{$q\in{\mathbb{N}}$ with the same parity as $n$} 
\end{equation}
there exists a ${\mathbb{C}}^{M\times M}$-valued function
\begin{equation}\label{Lnabv-ufwhJ.7}
K_q\in\big[{\mathscr{C}}^{q-1}\big(\overline{{\mathbb{R}}^n_{+}}\,\big)
\cap{\mathscr{C}}^{\infty}\big(\overline{{\mathbb{R}}^n_{+}}
\setminus B(0,\varepsilon)\big)\big]^{M\times M},\qquad\forall\,\varepsilon>0
\end{equation}
(in fact, $K_q\in\big[{\mathscr{C}}^{q}\big(\overline{{\mathbb{R}}^n_{+}}\,\big)\big]^{M\times M}$ 
if $n\geq 3$), with the property that for each $\alpha\in{\mathbb{N}}_0^n$ 
there exists $C_\alpha\in(0,\infty)$ for which 
\begin{equation}\label{Lnabv-ufwhJ.8}
\big|\big(\partial^\alpha K_q\big)(x)\big|\leq C_\alpha |x|^{q-|\alpha|}
\big(1+\big|{\rm ln}\,|x|\big|\big),\qquad\forall\,x\in{\mathbb{R}}^n_{+},
\end{equation}
(moreover, if $|\alpha|\geq q+1$ then $\partial^\alpha K_q$ is positive homogeneous of degree $q-|\alpha|$ in ${\mathbb{R}}^n_{+}$ 
and the logarithmic term in \eqref{Lnabv-ufwhJ.8} may be omitted), and such that 
\begin{equation}\label{Lnabv-ufwhJ.9}
K^L(x',t)=\Delta_{x'}^{(n+q)/2}\big[K_q(x',t)\big],\qquad\forall\,(x',t)\in{\mathbb{R}}^n_{+},
\end{equation}
where $\Delta_{x'}$ denotes the $(n-1)$-dimensional Laplacian in the variable $x'\in{\mathbb{R}}^{n-1}$.

Granted this, if $u$ is defined as in \eqref{eq1.6Adaf.Rd} for some 
${\mathbb{C}}^{M}$-valued function $\psi\in{\mathscr{C}}^\infty({\mathbb{R}}^{n-1})$ 
satisfying \eqref{TFFV-b54EE.1}, then for every $q$ as in \eqref{ha-GVVC-v6bu}
and for each fixed $(x',t)\in{\mathbb{R}}^n_{+}$ we may write 
\begin{align}\label{Lnabv-ufwhJ.10}
u(x',t) &=\int_{{\mathbb{R}}^{n-1}}K^L(x'-y',t)\psi(y')\,dy'
=\int_{{\mathbb{R}}^{n-1}}\Delta_{x'}^{(n+q)/2}\big[K_q(x'-y',t)\big]\psi(y')\,dy'
\nonumber\\[4pt]
&=\int_{{\mathbb{R}}^{n-1}}\Delta_{y'}^{(n+q)/2}\big[K_q(x'-y',t)\big]\psi(y')\,dy'
\nonumber\\[4pt]
&=\lim_{R\to\infty}\int_{{\mathbb{R}}^{n-1}}
\Delta_{y'}^{(n+q)/2}\big[K_q(x'-y',t)\big]\psi_R(y')\,dy'
\nonumber\\[4pt]
&=\lim_{R\to\infty}\int_{{\mathbb{R}}^{n-1}}
K_q(x'-y',t)\Delta_{y'}^{(n+q)/2}[\psi_R(y')]\,dy'
\nonumber\\[4pt]
&=\int_{{\mathbb{R}}^{n-1}}K_q(x'-y',t)\Delta_{y'}^{(n+q)/2}[\psi(y')]\,dy',
\end{align}
by \eqref{Lnabv-ufwhJ.9}, Lebesgue's Dominated Convergence Theorem (which,
in turn, is based on \eqref{Lnabv-ufwhJ.5}-\eqref{Lnabv-ufwhJ.6} 
and \eqref{Lnabv-ufwhJ.8}) used twice, and integration by parts 
(which relies on \eqref{Lnabv-ufwhJ.3BIs} and \eqref{Lnabv-ufwhJ.7}).
Thus, for any $q$ as in \eqref{ha-GVVC-v6bu}, we have
\begin{align}\label{Lnabv-ufwhJ.11}
u(x',t)=\int_{{\mathbb{R}}^{n-1}}K_q(x'-y',t)\Delta_{y'}^{(n+q)/2}[\psi(y')]\,dy',
\qquad\forall\,(x',t)\in{\mathbb{R}}^n_{+}.
\end{align}
Differentiating under the integral sign in \eqref{Lnabv-ufwhJ.11} 
(using \eqref{Lnabv-ufwhJ.8}, \eqref{TFFV-b54EE.1}, and Lebesgue's Dominated Convergence Theorem)
then shows that at each point $(x',t)\in{\mathbb{R}}^n_{+}$ we have, for every $q$ as in \eqref{ha-GVVC-v6bu},  
\begin{equation}\label{Lnabv-ufwhJ.12}
(\partial^\alpha u)(x',t)=\int_{{\mathbb{R}}^{n-1}}(\partial^\alpha K_q)(x'-y',t)\Delta_{y'}^{(n+q)/2}[\psi(y')]\,dy',
\qquad\forall\,\alpha\in{\mathbb{N}}_0^n. 
\end{equation}
This proves that $u\in\big[{\mathscr{C}}^\infty({\mathbb{R}}^n_{+})\big]^M$, with a useful 
accompanying integral representation formula for derivatives.

As far as the full force of the claim in \eqref{Lnabv-ufwhJ} is concerned, 
in a first stage we shall show that
\begin{equation}\label{NnvD4rfDS}
\partial^\alpha u\text{ extends continuously to }\,\,\overline{{\mathbb{R}}^n_{+}}
\,\,\text{ for each }\,\,\alpha\in{\mathbb{N}}_0^n.
\end{equation}
With this goal in mind, fix an arbitrary $\alpha\in{\mathbb{N}}_0^n$ and consider 
the representation of $\partial^\alpha u$ in ${\mathbb{R}}^n_{+}$ given 
by \eqref{Lnabv-ufwhJ.12} in which we now make the additional assumption that 
the parameter $q$ satisfies $q\geq|\alpha|+1$. In particular, this ensures that
\begin{equation}\label{Jfg-vSkdKV}
\partial^\alpha K_q\in\big[{\mathscr{C}}^{0}\big(\overline{{\mathbb{R}}^n_{+}}\,\big)\big]^{M\times M}
\end{equation}
by \eqref{Lnabv-ufwhJ.7}. Consider the issue of the existence of the limit 
\begin{equation}\label{LnaGV-rE}
\lim\limits_{{\mathbb{R}}^n_{+}\ni (x',\,t)\to (z',0)}\int_{{\mathbb{R}}^{n-1}}
(\partial^\alpha K_q)(x'-y',t)\Delta_{y'}^{(n+q)/2}[\psi(y')]\,dy',
\end{equation}
where $z'\in{\mathbb{R}}^{n-1}$ is an arbitrary fixed point. In this regard, observe that 
\eqref{Lnabv-ufwhJ.8} implies that for each $\varepsilon\in(0,N+1)$ there exists a finite constant 
$C=C(z',\varepsilon)>0$ with the property that for each $y'\in{\mathbb{R}}^{n-1}$ we have
\begin{equation}\label{IFsVNi}
\sup\limits_{(x',\,t)\in B((z',0),1)\cap{\mathbb{R}}^n_{+}}\big|(\partial^\alpha K_q)(x'-y',t)\big|
\leq C\big[|y'|+1\big]^{q-|\alpha|+\varepsilon}.
\end{equation}
Collectively \eqref{Jfg-vSkdKV}, \eqref{IFsVNi}, \eqref{TFFV-b54EE.1}, and 
Lebesgue's Dominated Convergence Theorem give that, on the one hand,  
\begin{align}\label{LnaGV-rE-B}
&\lim\limits_{{\mathbb{R}}^n_{+}\ni (x',\,t)\to (z',0)}\int_{{\mathbb{R}}^{n-1}}
(\partial^\alpha K_q)(x'-y',t)\Delta_{y'}^{(n+q)/2}[\psi(y')]\,dy'
\nonumber\\[6pt]
&\hskip 1.00in
=\int_{{\mathbb{R}}^{n-1}}(\partial^\alpha K_q)(z'-y',0)\Delta_{y'}^{(n+q)/2}[\psi(y')]\,dy',
\end{align}
where the last integral is absolutely convergent.
This analysis proves that given any $\alpha\in{\mathbb{N}}_0^n$ the limit
$\lim\limits_{{\mathbb{R}}^n_{+}\ni (x',\,t)\to (z',0)}(\partial^\alpha u)(x',t)$ 
exists and is finite for each $z'\in{\mathbb{R}}^{n-1}$. Moreover, the same 
type of analysis may also be employed to show that the function 
\begin{equation}\label{Lnahgfs-ii}
\Psi(z'):=\int_{{\mathbb{R}}^{n-1}}(\partial^\alpha K_q)(z'-y',0)
\Delta_{y'}^{(n+q)/2}[\psi(y')]\,dy',\qquad z'\in{\mathbb{R}}^{n-1},
\end{equation}
is continuous on ${\mathbb{R}}^{n-1}$. Thus, \eqref{NnvD4rfDS} is established.

The end-game in the proof of the lemma is the justification of the fact 
that \eqref{NnvD4rfDS} implies \eqref{Lnabv-ufwhJ}. Specifically, from 
\eqref{NnvD4rfDS} we deduce that for each $z'\in{\mathbb{R}}^{n-1}$ and $r>0$, 
the restriction of $u$ to $B^{+}_r(z'):=B\big((z',0),r\big)\cap{\mathbb{R}}^n_{+}$
belongs to the Sobolev space $\big[W^{k,p}\big(B^{+}_r(z')\big)\big]^M$ for every exponent
$p\in(1,\infty)$ and $k\in{\mathbb{N}}$. Using, e.g., Stein's universal extension 
operator as well as standard Sobolev embedding results, it is therefore possible 
to find $U_{z',r}\in\big[{\mathscr{C}}^\infty({\mathbb{R}}^n)\big]^M$ with the property that
$U_{z',r}\big|_{B^{+}_r(z')}=u\big|_{B^{+}_r(z')}$. Finally, gluing together such 
local extensions via a smooth partition of unity yields a ${\mathscr{C}}^\infty$ 
extension of $u$ to a neighborhood of $\overline{{\mathbb{R}}^n_{+}}$. 
\end{proof}

\section{Proofs of the Main Results}
\setcounter{equation}{0}
\label{S-3}

First, we take on the task of presenting the proof of Theorem~\ref{ta.av-GGG.2A}.

\vskip 0.08in
\begin{proof}[Proof of Theorem~\ref{ta.av-GGG.2A}]
Recall from Theorem~\ref{ya-T4-fav} that $L$ has a Poisson kernel $P^L$ 
(in the sense of Definition~\ref{defi:Poisson}) which satisfies by $P^L(x')=K^L(x',1)$ 
for each $x'\in{\mathbb{R}}^{n-1}$ (cf. item {\it (ii)} in Remark~\ref{Ryf-uyf}).
In particular (cf. item {\it (c)} in Definition~\ref{defi:Poisson}), 
\begin{equation}\label{tarad-tre445}
P^L_t(x'-z')=K^L(x'-z',t),\qquad\forall\,x',z'\in{\mathbb{R}}^{n-1},\quad\forall\,t>0.
\end{equation}
Throughout the proof, given an arbitrary point $y=(y_1,\dots,y_n)\in\mathbb{R}^n_+$
we shall denote by $\overline{y}$ its reflection across $\partial\mathbb{R}^n_+$, i.e., 
$\overline{y}=(y_1,\dots,y_{n-1},-y_n)\in\mathbb{R}^n_{-}:=\mathbb{R}^n\setminus\overline{\mathbb{R}^n_{+}}$. 
Also, $E_{L}$ will denote the fundamental solution for $L$ from Theorem~\ref{FS-prop}.

The strategy is to use the Poisson kernel $P^L$ in order to construct a 
particular Green function, which is seen to enjoy considerably stronger 
properties than those stipulated in Definition~\ref{ta.av-GGG}. In turn, 
this special Green function will be employed to show that there is  
precisely one Green function in the sense of Definition~\ref{ta.av-GGG}.
This is accomplished in a series of steps, starting with:

\vskip 0.08in
\noindent{\tt Step~1:} \textit{For each $y\in{\mathbb{R}}^n_{+}$ and each 
$x=(x',t)\in{\mathbb{R}}^n_{+}\setminus\{y\}$ define 
\begin{equation}\label{Ihab-Ygab673}
G^L(x,y):=E^L(x-y)-E^L(x-\overline{y}\,)
-P^L_t\ast\Big(\big[E^L(\cdot-y)-E^L(\cdot-\overline{y}\,)
\big]\big|_{\partial{\mathbb{R}}^n_{+}}\Big)(x'),
\end{equation}
where the convolution with $P^L_t$ is applied to each column of the matrix inside the round parentheses. 
Then $G^L(\cdot,\cdot):{\mathbb{R}}^n_{+}\times{\mathbb{R}}^n_{+}\setminus{\rm diag}\to{\mathbb{C}}^{M\times M}$ 
is a Green function for $L$ in $\mathbb{R}^n_+$ {\rm (}in the sense of Definition~\ref{ta.av-GGG}{\rm )}
which also satisfies 
\begin{equation}\label{GHCewd-22.RRe.2}
G^L(\cdot\,,y)\in\big[W^{1,1}_{\rm loc}({\mathbb{R}}^n_{+})\big]^{M\times M}
\,\,\text{ for each }\,\,y\in{\mathbb{R}}^n_{+},
\end{equation}
and
\begin{equation}\label{hBb-TR.D2}
\begin{array}{c}
\text{for each point $y\in{\mathbb{R}}^n_{+}$, each compact $K\subset{\mathbb{R}}^n_{+}$ 
with $y\in\mathring{K}$}, 
\\[6pt]
\text{and each $\kappa>0$, there exists a constant $C_{y,K}\in(0,\infty)$ such that}
\\[10pt]
\displaystyle
\big(\mathcal{N}^{^{\,{\mathbb{R}}^n_{+}\setminus K}}_\kappa G^L(\cdot,y)\big)(x')\leq
C_{y,K}\Big(\frac{1+\log_{+}|x'|}{1+|x'|^{n-1}}\Big)\,\,\text{ at each }\,\,x'\in{\mathbb{R}}^{n-1}.
\end{array}
\end{equation}
}
\vskip 0.08in

To get started, fix some $\kappa>0$ along with a point $x^\star\in\mathbb{R}^n_{+}$, and consider the compact set 
\begin{equation}\label{GHCewd-abc}
K_\star:=\overline{B(x^\star,r)}\subset{\mathbb{R}}^n_{+}\quad\text{ where }\,\,
r:=\tfrac14\,{\rm dist}\,\big(x^\star,\partial{\mathbb{R}}^n_{+}\big).
\end{equation}
We then claim that for any $N\geq 0$ there holds 
\begin{equation}\label{GHCewd-12}
\sup_{y\in\Gamma_\kappa(x')\setminus K_\star}\Big[|y-x^\star|^{-N}\Big]\approx
\big|(x',0)-x^\star\big|^{-N},\,\,\text{ uniformly for }\,\,x'\in{\mathbb{R}}^{n-1}.
\end{equation}
To justify \eqref{GHCewd-12} fix $x'\in{\mathbb{R}}^{n-1}$ and note that, 
in one direction, 
\begin{equation}\label{GHCewd-13}
\sup_{y\in\Gamma_\kappa(x')\setminus K_\star}\Big[|y-x^\star|^{-N}\Big]
\geq\lim_{\Gamma_\kappa(x')\setminus K_\star\ni y\to(x',0)}\Big[|y-x^\star|^{-N}\Big]
=\big|(x',0)-x^\star\big|^{-N}.
\end{equation}
In the opposite direction, consider first the case when 
\begin{equation}\label{GHCewd-UUU}
\big|x'-(x^\star)'\big|>8\big(1+(1+\kappa^2)^{1/2}\big)r.
\end{equation}
Given $y=(y',t)\in\Gamma_\kappa(x')$ it follows that $|y'-x'|<\kappa\,t$,
hence 
\begin{equation}\label{GHCewd-14}
\big|(y',t)-(x',0)\big|<(1+\kappa^2)^{1/2}t\leq(1+\kappa^2)^{1/2}
\big|(y',t)-((x^\star)',0)\big|.
\end{equation}
Based on this and \eqref{GHCewd-UUU} we may then write 
\begin{align}\label{GHCewd-15}
|x'-(x^\star)'| &=\big|(x',0)-((x^\star)',0)\big|\leq\big|(x',0)-(y',t)\big|
+\big|(y',t)-((x^\star)',0)\big|
\nonumber\\[4pt]
&\leq\big[1+(1+\kappa^2)^{1/2}\big]\big|(y',t)-((x^\star)',0)\big|
\nonumber\\[4pt]
&\leq\big[1+(1+\kappa^2)^{1/2}\big]\big\{|y-x^\star|+4r\big\}
\nonumber\\[4pt]
&\leq\big[1+(1+\kappa^2)^{1/2}\big]
\Big\{|y-x^\star|+\frac{|x'-(x^\star)'|}{2\big(1+(1+\kappa^2)^{1/2}\big)}\Big\}.
\end{align}
Consequently, 
\begin{equation}\label{GHCewd-16}
|x'-(x^\star)'|\leq 2\big[1+(1+\kappa^2)^{1/2}\big]|y-x^\star|.
\end{equation}
In turn, \eqref{GHCewd-16} allows us to estimate 
\begin{align}\label{GHCewd-17}
\big|(x',0)-x^\star\big| &\leq\big|x'-(x^\star)'\big|+(x^\star)_n=\big|x'-(x^\star)'\big|+4r
\nonumber\\[4pt]
&\leq\big|x'-(x^\star)'\big|+\frac{|x'-(x^\star)'|}{2\big(1+(1+\kappa^2)^{1/2}\big)}
\nonumber\\[4pt]
&\leq\big(3+2(1+\kappa^2)^{1/2}\big)|y-x^\star|.
\end{align}
With this in hand we may now conclude that 
\begin{equation}\label{GHCewd-18}
\sup_{y\in\Gamma_\kappa(x')}\Big[|y-x^\star|^{-N}\Big]\leq C_{N,\kappa}
\big|(x',0)-x^\star\big|^{-N}\,\,\text{ for all }\,\,x'\in{\mathbb{R}}^{n-1}
\,\,\text{ as in }\,\,\eqref{GHCewd-UUU}.
\end{equation}

There remains to establish a similar estimate in the case when 
\begin{equation}\label{GHCewd-19}
\big|x'-(x^\star)'\big|\leq 8\big(1+(1+\kappa^2)^{1/2}\big)r.
\end{equation}
To this end, pick an arbitrary $y\in{\mathbb{R}}^n_{+}\setminus K_\star$ and
note that this forces $|y-x^\star|>r$. On the other hand, 
\begin{equation}\label{GHCewd-20}
\big|(x',0)-x^\star\big|\leq\big|x'-(x^\star)'\big|+(x^\star)_n
\leq 4\big(3+2(1+\kappa^2)^{1/2}\big)r,
\end{equation}
hence $\big|(x',0)-x^\star\big|\leq 4\big(3+2(1+\kappa^2)^{1/2}\big)|y-x^\star|$
which goes to show that
\begin{equation}\label{GHCewd-21}
\sup_{y\in{\mathbb{R}}^n_{+}\setminus K_\star}\Big[|y-x^\star|^{-N}\Big]\leq C_{N,\kappa}
\big|(x',0)-x^\star\big|^{-N}\,\,\text{ for all }\,\,x'\in{\mathbb{R}}^{n-1}
\,\,\text{ as in }\,\,\eqref{GHCewd-19}.
\end{equation}
In concert, \eqref{GHCewd-18} and \eqref{GHCewd-21} prove the left-pointing 
inequality in \eqref{GHCewd-12}. This finishes the proof of \eqref{GHCewd-12}.

In turn, \eqref{GHCewd-12} readily self-improves to an estimate of the following sort. 
\begin{equation}\label{GHCewd-12-best}
\begin{array}{c}
\text{for each point $x^\ast\in{\mathbb{R}}^n_{+}$, each compact $K\subset{\mathbb{R}}^n_{+}$ with $x^\ast\in\mathring{K}$}, 
\\[6pt]
\text{each $N>0$, and each $\kappa>0$, there exists a constant $C=C(x^\ast,K,N,\kappa)\in(0,\infty)$ such that}
\\[10pt]
\displaystyle
\big(\mathcal{N}^{^{\,{\mathbb{R}}^n_{+}\setminus K}}_\kappa\big(|\cdot-x^\ast|^{-N})\big)(x')\leq C(1+|x'|)^{-N}
\,\,\text{ for each }\,\,x'\in{\mathbb{R}}^{n-1}.
\end{array}
\end{equation}

Going further, having fixed any $y\in{\mathbb{R}}^n_{+}$ along with some compact 
set $K\subset{\mathbb{R}}^n_{+}$ such that $y\in{\mathring{K}}$, for any multi-index 
$\gamma\in{\mathbb{N}}^n_0$ we may estimate 
\begin{equation}\label{GHCewd-4}
\big|(\partial^\gamma E^{L})(x-y)-(\partial^\gamma E^L)(x-\overline{y}\,)\big|
\leq\frac{C(n,\gamma,K,y)}{|x-y|^{n-1+|\gamma|}},\qquad\forall\,x\in\overline{{\mathbb{R}}^n_{+}}\setminus K,
\end{equation}
for some finite constant $C(n,\gamma,K)>0$, by the Mean Value Theorem and Theorem~\ref{FS-prop}. 
In particular, if we consider 
\begin{equation}\label{GHCewd-5}
f(x'):=E^{L}\big((x',0)-y\big)-E^{L}\big((x',0)-\overline{y}\,\big),\qquad\forall\,x'\in{\mathbb{R}}^{n-1},
\end{equation}
then $f\in\big[{\mathscr{C}}^\infty({\mathbb{R}}^{n-1})\big]^{M}$ and \eqref{GHCewd-4} allows us 
to conclude that for each $\gamma\,'\in{\mathbb{N}}_0^{n-1}$ we have
\begin{equation}\label{GHCewd-5BBIISS}
|(\partial^{\gamma\,'}f)(x')|\leq\frac{C(y,\gamma\,')}{1+|x'|^{n-1+|\gamma\,'|}}
\,\,\text{ for each }\,\,x'\in{\mathbb{R}}^{n-1}.
\end{equation}
Together with Lemma~\ref{lemma:M-ball}, the version of \eqref{GHCewd-5BBIISS} corresponding 
to $\gamma\,'=(0,\dots,0)$ ensures that 
\begin{equation}\label{GHCewd-5BB-trr}
\big({\mathcal{M}}f\big)(x')\leq C\Big(\frac{1+\log_{+}|x'|}{1+|x'|^{n-1}}\Big)
\,\,\text{ for each }\,\,x'\in{\mathbb{R}}^{n-1}.
\end{equation}

At this stage, return to \eqref{Ihab-Ygab673} and, for each $y\in{\mathbb{R}}^n_{+}$ and 
$x=(x',t)\in{\mathbb{R}}^n_{+}\setminus\{y\}$, express
\begin{equation}\label{Ihab-Ygab673.a}
G^L(x,y)=G^L_1(x,y)+G^L_2(x,y),
\end{equation}
where 
\begin{equation}\label{Ihab-Ygab673.b}
\begin{array}{c}
G^L_1(x,y):=E^L(x-y)-E^L(x-\overline{y}\,)\,\,\text{ and}
\\[6pt]
G^L_2(x,y):=-P^L_t\ast\Big(\big[E^L(\cdot-y)-E^L(\cdot-\overline{y}\,)\big]\big|_{\partial{\mathbb{R}}^n_{+}}\Bigr)(x').
\end{array}
\end{equation}
Then from \eqref{GHCewd-4} and \eqref{GHCewd-12-best} we see that
\begin{equation}\label{hBb-TR.D2.c}
\begin{array}{c}
\text{for each point $y\in{\mathbb{R}}^n_{+}$, each compact $K\subset{\mathbb{R}}^n_{+}$ 
with $y\in\mathring{K}$}, 
\\[6pt]
\text{and each $\kappa>0$, there exists a constant $C_{y,K}\in(0,\infty)$ such that}
\\[10pt]
\displaystyle
\big(\mathcal{N}^{^{\,{\mathbb{R}}^n_{+}\setminus K}}_\kappa G^L_1(\cdot,y)\big)(x')\leq
\frac{C_{y,K}}{1+|x'|^{n-1}}\,\,\text{ for each }\,\,x'\in{\mathbb{R}}^{n-1},
\end{array}
\end{equation}
while from \eqref{exist:Nu-Mf}, \eqref{GHCewd-5}, and \eqref{GHCewd-5BB-trr} we see that
\begin{equation}\label{hBb-TR.D2.d}
\begin{array}{c}
\text{for each point $y\in{\mathbb{R}}^n_{+}$, each compact $K\subset{\mathbb{R}}^n_{+}$ 
with $y\in\mathring{K}$}, 
\\[6pt]
\text{and each $\kappa>0$, there exists a constant $C_{y,K}\in(0,\infty)$ such that}
\\[10pt]
\displaystyle
\big(\mathcal{N}^{^{\,{\mathbb{R}}^n_{+}\setminus K}}_\kappa G^L_2(\cdot,y)\big)(x')\leq
C_{y,K}\Big(\frac{1+\log_{+}|x'|}{1+|x'|^{n-1}}\Big)\,\,\text{ for each }\,\,x'\in{\mathbb{R}}^{n-1}.
\end{array}
\end{equation}
Collectively, \eqref{hBb-TR.D2.c} and \eqref{hBb-TR.D2.d} prove \eqref{hBb-TR.D2}.

Pressing on, the membership in \eqref{GHCewd-22.RRe.2} is a consequence of \eqref{Ihab-Ygab673.a}, the fact 
that the function $G^L_2(\cdot,y)$ belongs to $\big[{\mathscr{C}}^\infty({\mathbb{R}}^n_{+})\big]^{M\times M}$ 
(thanks to \eqref{GHCewd-5}-\eqref{GHCewd-5BBIISS} and the first property in \eqref{exist:u2}), 
and Theorem~\ref{FS-prop} (cf. item {\it (4)} in particular). 

To show that $G^L(\cdot,\cdot)$ from \eqref{Ihab-Ygab673} is a Green function for $L$ 
in $\mathbb{R}^n_{+}$ in the sense of Definition~\ref{ta.av-GGG}, observe that \eqref{GHCewd-22.RRe}
is contained in \eqref{GHCewd-22.RRe.2}, while \eqref{GHCewd-24.RRe} follows from 
\eqref{Ihab-Ygab673} and the last property recorded in \eqref{exist:u2}. Also, \eqref{GHCewd-25.RRe}
is implied by \eqref{hBb-TR.D2}, whereas \eqref{GHCewd-23.RRe} is seen from Theorem~\ref{FS-prop}
and \eqref{exist:u2}.

\vskip 0.08in
\noindent{\tt Step~2:} \textit{For each $y\in{\mathbb{R}}^n_{+}$ we have 
$G^L(\cdot\,,y)\in\big[{\mathscr{C}}^\infty\big(\overline{{\mathbb{R}}^n_{+}}
\setminus B(y,\varepsilon)\big)\big]^{M\times M}$ for each $\varepsilon>0$.}

\vskip 0.08in
This is a direct consequence of Lemma~\ref{UYfa-uVdssk} 
(keeping in mind \eqref{tarad-tre445}) and the fact that, as seen 
from \eqref{GHCewd-4}, for each fixed $y\in{\mathbb{R}}^n_{+}$ the function 
$\psi:{\mathbb{R}}^{n-1}\to{\mathbb{C}}^{M\times M}$ given by 
\begin{equation}\label{GHCewd-4EWs}
\psi(z'):=E^L\big((z',0)-y\big)-E^L\big((z',0)-\overline{y}\,\big),
\qquad\forall\,z'\in{\mathbb{R}}^{n-1},
\end{equation}
is smooth and satisfies \eqref{TFFV-b54EE.1} with $N:=n-1$ (see \eqref{GHCewd-5}-\eqref{GHCewd-5BBIISS}).

\vskip 0.08in
\noindent{\tt Step~3:} \textit{If $R_L(x,y):=E^L(x-y)-G^L(x,y)$ for each 
$x,y\in{\mathbb{R}}^n_{+}$ with $x\not=y$, then this extends to a function 
\begin{equation}\label{Rj-CINF}
R_L(\cdot,\cdot)\in\big[{\mathscr{C}}^\infty\big({\mathbb{R}}^n_{+}\times{\mathbb{R}}^n_{+}\big)\big]^{M\times M},
\end{equation}
which satisfies 
\begin{equation}\label{Rj}
|R_L(x,y)|\leq C\,|x-{\overline{y}}|^{2-n}
+c_n\big|\!\ln|x-\overline{y}|\big|,
\qquad\forall\,(x,y)\in{\mathbb{R}}^n_{+}\times{\mathbb{R}}^n_{+},
\end{equation}
where $c_n=0$ if $n\geq 3$.
}
\vskip 0.08in

To justify this claim, use \eqref{Ihab-Ygab673} in order to decompose 
\begin{equation}\label{GHCkn.Re.1A}
R_L(x,y)= R^{(1)}_L(x,y)+R^{(2)}_L(x,y),
\end{equation}
where, for each $x=(x',x_n)\in{\mathbb{R}}^n_{+}$ and each 
$y=(y',y_n)\in{\mathbb{R}}^n_{+}\setminus\{x\}$ we have set
\begin{equation}\label{GHCkn.Re.1B}
R^{(1)}_L(x,y):=E^L(x-\overline{y}\,),
\end{equation}
and 
\begin{align}\label{GHCkn.Re.1C}
R^{(2)}_L(x,y):=&\,P^L_{x_n}\ast\Big(\big[E^L(\cdot-y)-E^L(\cdot-\overline{y}\,)\big]
\big|_{\partial{\mathbb{R}}^n_{+}}\Bigr)(x')
\nonumber\\[4pt]
=& \int_{{\mathbb{R}}^{n-1}}P^L_{x_n}(x'-z')\big[E^L\big((z',0)-y\big)
-E^L\big((z',0)-\overline{y}\big)\big]\,dz'.
\end{align}
It is then clear from \eqref{GHCkn.Re.1B} and the first condition in \eqref{smmth-odd} that 
the function $R^{(1)}_L(\cdot,\cdot)$ belongs to 
$\big[{\mathscr{C}}^\infty\big({\mathbb{R}}^n_{+}\times{\mathbb{R}}^n_{+}\big)\big]^{M\times M}$, while 
\eqref{GHCewd-5}-\eqref{GHCewd-5BBIISS} and the first property in \eqref{exist:u2} imply that $R^{(2)}_L(\cdot,\cdot)$ 
also belongs to $\big[{\mathscr{C}}^\infty\big({\mathbb{R}}^n_{+}\times{\mathbb{R}}^n_{+}\big)\big]^{M\times M}$. 
As such, \eqref{Rj-CINF} follows, in light of \eqref{GHCkn.Re.1A}. 

Moving on, from part {\it (3)} in Theorem~\ref{FS-prop} we know that 
\begin{equation}\label{fs-strUUbV}
E^L(z)=\Phi(z)+{\mathfrak{C}}_n\ln|z|,\qquad\forall\,z\in\mathbb{R}^n\setminus\{0\},
\end{equation}
where $\Phi=\big(\Phi_{\alpha\beta}\big)_{1\leq\alpha,\beta\leq M}
\in\big[{\mathscr{C}}^\infty(\mathbb{R}^n\setminus\{0\})\big]^{M\times M}$ is a function 
which is positive homogeneous of degree $2-n$, and ${\mathfrak{C}}_n\in{\mathbb{C}}^{M\times M}$
is a matrix which is identically zero when $n\geq 3$, and when $n=2$ is given by 
\begin{equation}\label{Fm-PjkX}
{\mathfrak{C}}_2=\frac{1}{4\pi^2}\int_{S^{1}}\big[L(\xi)\big]^{-1}\,d{\mathcal{H}}^1(\xi).
\end{equation}
In relation to $R^{(2)}_L(\cdot,\cdot)$ we claim that there exists a finite constant $C>0$ with the property that
\begin{equation}\label{Rj-222}
\big|R^{(2)}_L(x,y)\big|\leq C\,|x-{\overline{y}}|^{2-n},
\qquad\forall\,(x,y)\in{\mathbb{R}}^n_{+}\times{\mathbb{R}}^n_{+}.
\end{equation}
To see that this is the case, note that $\big|\big((z',0)-y\big)\big|=\big|\big((z',0)-\overline{y}\big)\big|$ 
for each $z'\in{\mathbb{R}}^{n-1}$. Based on this and \eqref{fs-strUUbV} it follows that, on the one hand, 
\begin{align}\label{fs-sJba.1}
\big|E^L\big((z',0)-y\big)-& E^L\big((z',0)-\overline{y}\big)\big|
=\big|\Phi\big((z',0)-y\big)-\Phi\big((z',0)-\overline{y}\big)\big|
\nonumber\\[4pt]
& \leq\frac{2\|\Phi\|_{[L^\infty(S^{n-1})]^{M\times M}}}{\big|\big((z',0)-y\big)\big|^{n-2}},
\qquad\forall\,z'\in{\mathbb{R}}^{n-1}.
\end{align}
On the other hand, part $(a)$ in Definition~\ref{defi:Poisson} guarantees the 
existence of a constant $C\in(0,\infty)$ such that
\begin{equation}\label{GHCkn.Re.2}
|P^L_t(z')|\leq\frac{Ct}{(|z'|^2+t^2)^{\frac{n}{2}}},\qquad\forall\,z'\in\mathbb{R}^{n-1},
\qquad\forall\,t>0.
\end{equation}
Granted \eqref{fs-sJba.1}-\eqref{GHCkn.Re.2}, we may then estimate 
\begin{equation}\label{GBbohg-865gbp}
\big|R^{(2)}_L(x,y)\big|
\leq\int_{{\mathbb{R}}^{n-1}}\big|P^L_{x_n}(x'-z')\big|
\big|E^L\big((z',0)-y\big)-E^L\big((z',0)-\overline{y}\big)\big|\,dz'
\leq\,CU_y(x),\quad
\end{equation}
where for each fixed $y=(y',y_n)\in{\mathbb{R}}^n_{+}$ we have set
\begin{equation}\label{E1-abc}
U_y(x',t):=\frac{2}{\omega_{n-1}}
\int_{{\mathbb{R}}^{n-1}}\frac{t}{(|x'-z'|^2+t^2)^{\frac{n}{2}}}
\frac{1}{(|z'-y'|^2+y_n^2)^{\frac{n-2}{2}}}\,dz'
\end{equation}
for each $(x',t)\in{\mathbb{R}}^n_{+}$. Hence, if $n=2$ then $U_y(x',t)\leq 1$ for all $(x',t)\in{\mathbb{R}}^n_{+}$ 
which, in view of \eqref{GBbohg-865gbp}, yields \eqref{Rj-222} in this case. 
To prove \eqref{Rj-222} when $n\geq 3$, fix $y\in{\mathbb{R}}^n_{+}$ and 
denote by $P^\Delta$ the Poisson kernel for the Laplacian in ${\mathbb{R}}^n_{+}$, i.e., 
\begin{equation}\label{Uah-TTT}
P^{\Delta}(x'):=\frac{2}{\omega_{n-1}}\frac{1}{\big(1+|x'|^2\big)^{\frac{n}{2}}},
\qquad\forall\,x'\in{\mathbb{R}}^{n-1}.
\end{equation}
Then, as seen from \eqref{E1-abc} and \eqref{Uah-TTT}, 
\begin{equation}\label{E1-abc.2}
\begin{array}{c}
U_y(x',t)=\big(P^\Delta_t\ast f_y\big)(x'),
\qquad\forall\,(x',t)\in{\mathbb{R}}^n_{+},
\\[12pt]
\text{where }\,\,f_y(z'):=\big(|z'-y'|^2+y_n^2\big)^{\frac{2-n}{2}},
\quad\forall\,z'\in{\mathbb{R}}^{n-1}.
\end{array}
\end{equation}
Since $f_y\in L^\infty({\mathbb{R}}^{n-1})$ it follows from Proposition~\ref{thm:existence} 
(used for $L:=\Delta$) and \eqref{E1-abc.2} that (for each fixed $\kappa>0$) we have 
\begin{equation}\label{E1-abc.2BB}
\Delta U_y=0\,\,\text{ in }\,\,{\mathbb{R}}^n_{+},\quad U_y\in L^\infty({\mathbb{R}}^n_{+}),\quad
U_y\big|^{{}^{\kappa-{\rm n.t.}}}_{\partial{\mathbb{R}}^n_{+}}=f_y\,\,\text{ a.e. in }\,\,{\mathbb{R}}^{n-1}.
\end{equation}
Then the function given by 
\begin{equation}\label{E1-abc.3}
u(x',t):=U_y(x',t)-|x-\overline{y}|^{2-n},\qquad\forall\,x=(x',t)\in{\mathbb{R}}^n_{+},
\end{equation}
is harmonic in ${\mathbb{R}}^n_{+}$, and for a.e. $x'\in{\mathbb{R}}^{n-1}$ satisfies
\begin{align}\label{E1-abc.4}
\Big(u\Big|^{{}^{\kappa-{\rm n.t.}}}_{\partial{\mathbb{R}}^n_{+}}\Big)(x')
=& f_y(x')-\big|(x',0)-\overline{y}\big|^{2-n}
\nonumber\\[4pt]
=& \big(|x'-y'|^2+y_n^2\big)^{\frac{2-n}{2}}-\big(|x'-y'|^2+y_n^2\big)^{\frac{2-n}{2}}=0,
\end{align}
thanks to \eqref{E1-abc.3}, \eqref{E1-abc.2BB}, and \eqref{E1-abc.2}. 
Since from \eqref{E1-abc.3} and \eqref{E1-abc.2BB} we also have that 
\begin{equation}\label{E1-abc.5}
u\in L^\infty({\mathbb{R}}^n_{+}),
\end{equation}
we may conclude from the above analysis and \cite[Proposition~1, p.\,199]{St70} 
(also bearing in mind the last property in \eqref{exist:u2}) that $u\equiv 0$ in ${\mathbb{R}}^n_{+}$. 
Upon recalling \eqref{E1-abc.3}, this further implies that
\begin{equation}\label{E1-abc.6}
U_y(x',t)=|x-\overline{y}|^{2-n},\qquad\forall\,x=(x',t)\in{\mathbb{R}}^n_{+},
\end{equation}
and \eqref{Rj-222} now follows from \eqref{GBbohg-865gbp} and \eqref{E1-abc.6}.

To finish the proof of \eqref{Rj}, observe that by \eqref{GHCkn.Re.1B} 
and \eqref{fs-strUUbV} we also have 
\begin{equation}\label{nJBB-iy5}
\big|R^{(1)}_L(x,y)\big|=
\big|E^L(x-\overline{y}\,)\big|\leq\|\Phi\|_{L^\infty(S^{n-1})}
|x-{\overline{y}}|^{2-n}+\big|{\mathfrak{C}}_n\big|\,\big|\!\ln|x-\overline{y}|\big|.
\end{equation}
Now \eqref{Rj} follows from \eqref{GHCkn.Re.1A}, \eqref{Rj-222}, and \eqref{nJBB-iy5}.

\vskip 0.08in
\noindent{\tt Step~4:} \textit{If $R_L(\cdot,\cdot)$ is as in Step~3, then 
\begin{equation}\label{Pavr-6T4}
\begin{array}{c}
\text{for each }\,\alpha\in\mathbb{N}_0^n\,\text{ with }\,
|\alpha|>0\,\text{ there exists }\,C_{\alpha}\in(0,\infty)\,\text{ such that}
\\[4pt]
\big|\big(\partial^\alpha_X R_L\big)(x,y)\big|
\leq C_\alpha|x-\overline{y}|^{2-n-|\alpha|}\,\,\text{ for each }\,\,
(x,y)\in{\mathbb{R}}^n_{+}\times{\mathbb{R}}^n_{+}.
\end{array}
\end{equation}
}
\vskip -0.05in

To obtain pointwise estimates for derivatives in the first set of variables
for the function $R_L(\cdot,\cdot)$, we shall rely on the local estimates near 
the boundary from Proposition~\ref{c1.2}. To set the stage, recall 
$R^{(1)}_L(\cdot,\cdot)$, $R^{(2)}_L(\cdot,\cdot)$ from \eqref{GHCkn.Re.1B}-\eqref{GHCkn.Re.1C}. 
Also, fix an arbitrary point $y\in{\mathbb{R}}^n_{+}$ and observe that, thanks to 
\eqref{GHCewd-5BBIISS} and Lemma~\ref{UYfa-uVdssk}, 
\begin{equation}\label{Pavr-6T4.B}
R^{(2)}_L(\cdot,y)\in\big[{\mathscr{C}}^\infty\big(\overline{{\mathbb{R}}^n_{+}}\,\big)\big]^{M\times M}.
\end{equation}
In particular, 
\begin{equation}\label{GHCkn.Re.7}
R^{(2)}_L(\cdot,y)\in\big[W^{1,2}({\mathbb{R}}^n_{+}\cap B(0,R)\big)\big]^{M\times M}\,\,\text{ for each }\,\,R>0.
\end{equation}
In addition, by \eqref{GHCkn.Re.1C} and Proposition~\ref{thm:existence}, we have
\begin{equation}\label{GHCkn.Re.8}
\left\{
\begin{array}{l}
L\big[R^{(2)}_L(\cdot,y)\big]=0\,\,\text{ in }\,\,{\mathbb{R}}^n_{+}\,\,\text{ and}
\\[8pt]
\big[R^{(2)}_L(\cdot,y)\big]\Big|^{{}^{\kappa-{\rm n.t.}}}_{\partial{\mathbb{R}}^n_{+}}
=\big[E^L(\cdot-y)-E^L(\cdot-\overline{y})\big]\Big|_{\partial{\mathbb{R}}^n_{+}}
\,\,\text{ a.e. in }\,\,{\mathbb{R}}^{n-1},
\end{array}
\right.
\end{equation}
for each $\kappa>0$. In fact, thanks to \eqref{Pavr-6T4.B}, the 
last condition above may be reformulated as
\begin{equation}\label{GHCkn.Re.8BIS}
{\rm Tr}\,\big[R^{(2)}_L(\cdot,y)\big]
=\big[E^L(\cdot-y)-E^L(\cdot-\overline{y})\big]\big|_{\partial{\mathbb{R}}^n_{+}}
\in\big[{\mathscr{C}}^\infty({\mathbb{R}}^{n-1})\big]^{M\times M}.
\end{equation}

We claim that  
\begin{equation}\label{eq1.20}
\begin{array}{c}
\text{for each multi-index }\,\alpha\in\mathbb{N}_0^n\,\text{ there exists }
\,C_\alpha\in(0,\infty)\,\text{ such that}
\\[4pt]
\big|\big(\partial^\alpha_X R^{(2)}_L\big)(x,y)\big|
\leq C_\alpha|x-\overline{y}|^{2-n-|\alpha|}\,\,\text{ for each }\,\,
(x,y)\in{\mathbb{R}}^n_{+}\times{\mathbb{R}}^n_{+}.
\end{array}
\end{equation}
To prove this claim assume that two arbitrary points $x,y\in{\mathbb{R}}^n_+$ have been given. 
Introduce the quantity 
\begin{equation}\label{GHCkn.arV}
\rho:=|x-\overline{y}|/5>0 
\end{equation}
and pick some $z\in\overline{{\mathbb{R}}^n_{+}}$ such that $|x-z|=\rho/2$. 
It follows that for any $w\in\overline{\mathbb{R}^n_{+}}\cap B(z,2\rho)$ we have
\begin{align}\label{GHCkn.Re.9}
|x-\overline{y}| &\leq |x-z|+|z-w|+|w-\overline{y}|
\nonumber\\[4pt]
&\leq\rho/2+2\rho+|w-\overline{y}|=|x-\overline{y}|/2+|w-\overline{y}|.
\end{align}
In particular, 
\begin{equation}\label{GHCkn.Re.10}
|x-\overline{y}|/2\leq |w-\overline{y}|=|w-y|\,\,\text{ for every }\,\,
w\in\partial{\mathbb{R}}^n_+\cap B(z,2\rho). 
\end{equation}
Let us also observe that by \eqref{fs-strUUbV}, given any $\ell\in\mathbb{N}_0$, we may write 
\begin{align}\label{fs-sJba.2}
\big|\nabla^\ell_{\xi'}\big[E^L\big((\xi',0)-y\big)
-& E^L\big((\xi',0)-\overline{y}\big)\big]\big|
\nonumber\\[4pt]
& \leq\big|(\nabla^\ell\Phi)\big((\xi',0)-y\big)\big|
+\big|(\nabla^\ell\Phi)\big((\xi',0)-\overline{y}\big)\big|
\nonumber\\[4pt]
& \leq\frac{2\|\nabla^\ell\Phi\|_{L^\infty(S^{n-1})}}
{\big|\big((\xi',0)-y\big)\big|^{n-2+\ell}},
\qquad\forall\,\xi'\in{\mathbb{R}}^{n-1}.
\end{align}
Based on \eqref{GHCkn.arV} and \eqref{GHCkn.Re.10}-\eqref{fs-sJba.2}, for each $\ell\in\mathbb{N}_0$
we may therefore estimate 
\begin{align}\label{nablaPhi}
\rho^{\ell-1}\cdot\sup_{(\xi',0)\in\partial\mathbb{R}^n_{+}\cap B(z,2\rho)} &
\big|\nabla^\ell_{\xi'}\big[E^L\big((\xi',0)-y\big)
-E^L\big((\xi',0)-\overline{y}\big)\big]\big|
\nonumber\\[4pt]
& \leq\sup_{w\in\partial\mathbb{R}^n_{+}\cap B(z,2\rho)}
\frac{C_\ell|x-\overline{y}|^{\ell-1}}{|w-y|^{n-2+\ell}}
\leq\frac{C_\ell}{|x-\overline{y}|^{n-1}}.
\end{align}
With this in hand, for each $k\in{\mathbb{N}}$ we may write 
(keeping in mind that $x\in\mathbb{R}^n_{+}\cap B(z,\rho)$)
\begin{align}\label{eKna-6YH}
\big|\big(\nabla_X^k & R^{(2)}_L\big)(x,y)\big|\leq 
\,\,\sup_{\mathbb{R}^n_{+}\cap B(z,\rho)}
\big|\big(\nabla_X^k R^{(2)}_L\big)(\cdot,y)\big|
\nonumber\\[4pt]
\leq & \,C\,\rho^{-k}\Big(\sup_{\mathbb{R}^n_{+}\cap B(z,2\rho)}
\big|R^{(2)}_L(\cdot,y)\big|\Big)
\nonumber\\[4pt]
& +\rho^{1-k}\sum_{\ell=0}^{k+1}\rho^{\ell-1}\Big(
\sup_{(\xi',0)\in\partial\mathbb{R}^n_{+}\cap B(z,2\rho)}
\big|\nabla^\ell_{\xi'}\big[E^L\big((\xi',0)-y\big)
-E^L\big((\xi',0)-\overline{y}\big)\big]\big|\Big)
\nonumber\\[4pt]
\leq & \,C\rho^{-k}|x-\overline{y}|^{2-n}+C\rho^{1-k}|x-\overline{y}|^{1-n}
=C|x-\overline{y}|^{2-n-k},
\end{align}
by virtue of \eqref{GHCkn.Re.7}-\eqref{GHCkn.Re.9}, Proposition~\ref{c1.2}, 
\eqref{Rj-222}, and \eqref{nablaPhi}. This establishes \eqref{eq1.20} in the case when 
$\alpha\in{\mathbb{N}}_0^n$ has $|\alpha|>0$. Since the case when $|\alpha|=0$ is already 
contained in \eqref{Rj-222}, the proof of \eqref{eq1.20} is complete. 
As regards $R^{(1)}_L(\cdot,\cdot)$, by combining \eqref{GHCkn.Re.1B} 
and \eqref{fs-est} we see that
\begin{equation}\label{Yjn1.20}
\begin{array}{c}
\text{given }\,\alpha,\beta\in\mathbb{N}_0^n\,\text{ with }\,
|\alpha|+|\beta|>0\,\text{ there exists }\,C_{\alpha\beta}\in(0,\infty)\,\text{ so that}
\\[4pt]
\big|\big(\partial^\alpha_X\partial^\beta_Y R^{(1)}_L\big)(x,y)\big|
\leq C_{\alpha\beta}|x-\overline{y}|^{2-n-|\alpha|-|\beta|}\,\,\text{ for each }\,\,
(x,y)\in{\mathbb{R}}^n_{+}\times{\mathbb{R}}^n_{+}.
\end{array}
\end{equation}
Then \eqref{Pavr-6T4} follows from \eqref{GHCkn.Re.1A}, \eqref{eq1.20}, and
\eqref{Yjn1.20}.

\vskip 0.08in
\noindent{\tt Step~5:} \textit{We have that
\begin{equation}\label{Pavr-6Y4.a1}
\begin{array}{c}
\text{for each point $y\in{\mathbb{R}}^n_{+}$, each compact $K\subset{\mathbb{R}}^n_{+}$ 
with $y\in\mathring{K}$,}
\\[4pt]
\text{and each $\kappa>0$, there exists a constant $C_{y,K}\in(0,\infty)$ such that}
\\[4pt]
\big({\mathcal{N}}_\kappa^{\,\mathbb{R}^{n}_{+}\setminus K}(\nabla_X G^L)(\cdot\,,y)\big)(x')\leq C_{y,K}(1+|x'|)^{1-n}
\,\,\text{ at each }\,\,x'\in{\mathbb{R}}^{n-1}.
\end{array}
\end{equation}
}
\vskip -0.05in

Having fixed two arbitrary points $x,y\in{\mathbb{R}}^n_{+}$ with $x\not=y$, we may estimate 
\begin{align}\label{RBvav-y5p}
\big|(\nabla_X G^L\big)(x,y)\big|\leq & \big|(\nabla_X E^L\big)(x-y)\big|
+\big|(\nabla_X R_L\big)(x,y)\big|
\nonumber\\[4pt]
\leq & C|x-y|^{1-n}+C|x-\overline{y}|^{1-n}\leq C|x-y|^{1-n},
\end{align}
making use of Theorem~\ref{FS-prop}, \eqref{Pavr-6T4}, and the fact that 
\begin{equation}\label{Ojhb}
|x-y|=\big[|x'-y'|^2+|x_n-y_n|^2\big]^{1/2}
\leq\big[|x'-y'|^2+(x_n+y_n)^2\big]^{1/2}=|x-\overline{y}|,
\end{equation}
for each $x=(x',x_n)\in{\mathbb{R}}^n_{+}$ and 
$y=(y',y_n)\in{\mathbb{R}}^n_{+}$. Then \eqref{Pavr-6Y4.a1} is a consequence 
of \eqref{RBvav-y5p} and \eqref{GHCewd-12-best}.

\vskip 0.08in
\noindent{\tt Step~6:} \textit{There exists precisely one Green 
function for $L$ in ${\mathbb{R}}^n_{+}$, in the sense of Definition~\ref{ta.av-GGG}.
}
\vskip 0.08in
From Step~1 we know that $L$ has at least one Green 
function in ${\mathbb{R}}^n_{+}$, in the sense of Definition~\ref{ta.av-GGG}. 
To show that this is unique, suppose $G_1(\cdot,\cdot)$, $G_2(\cdot,\cdot)$, are two such
Green functions and fix an arbitrary point $y\in{\mathbb{R}}^n_{+}$.
If $u:=G_1(\cdot,y)-G_2(\cdot,y)$ in ${\mathbb{R}}^n_{+}\setminus\{y\}$, then 
properties \eqref{GHCewd-22.RRe} and \eqref{GHCewd-23.RRe} imply that 
$u\in\big[L^1_{\rm loc}({\mathbb{R}}^n_{+})\big]^{M\times M}$ and $Lu=0$ in 
$\big[{\mathcal{D}}'({\mathbb{R}}^n_{+})\big]^{M\times M}$. In particular, 
\begin{equation}\label{naIaPa.4}
u\in\big[{\mathscr{C}}^\infty({\mathbb{R}}^n_{+})\big]^{M\times M}
\,\,\text{ and }\,\,Lu=0\,\,\text{ pointwise in }\,\,{\mathbb{R}}^n_{+},
\end{equation}
by elliptic regularity. Furthermore, \eqref{GHCewd-24.RRe} implies that there exists $\kappa>0$ such that 
\begin{equation}\label{naIaPa.5}
u\big|^{{}^{\kappa-{\rm n.t.}}}_{\partial{\mathbb{R}}^n_{+}}=0\,\,\text{ a.e. in }\,\,{\mathbb{R}}^{n-1}.
\end{equation}
In addition, from \eqref{GHCewd-25.RRe} and the first condition in \eqref{naIaPa.4} we deduce that  
\begin{equation}\label{GHReRFv.BNv}
\int_{{\mathbb{R}}^{n-1}}\big({\mathcal{N}}_{\kappa}u\big)(x')\frac{dx'}{1+|x'|^{n-1}}<+\infty.
\end{equation}

To proceed, we make the claim that 
\begin{equation}\label{GHUNNDs}
\text{any function $u$ satisfying \eqref{naIaPa.4}-\eqref{GHReRFv.BNv} vanishes
identically in }\,\,{\mathbb{R}}^n_{+}.
\end{equation}
To justify this claim, let 
$G^{L^\top}(\cdot,\cdot)=\big(G^{L^\top}_{\alpha\gamma}(\cdot,\cdot)\big)_{1\leq\alpha,\gamma\leq M}$ 
be the Green function constructed as in Step~1 in relation to the operator $L^\top$, 
and fix an arbitrary point $y\in{\mathbb{R}}^n_{+}$. In particular, by Step~1,
Step~2, and Step~5 (all invoked with $L^\top$ in place of $L$) we know that 
for each compact subset $K$ of ${\mathbb{R}}^n_{+}$ with $y\in\mathring{K}$
there exists a constant $C_{y,K}\in(0,\infty)$ with the property that 
\begin{align}\label{GLa.1.i}
&\big(\mathcal{N}^{\,\mathbb{R}^{n}_{+}\setminus K}_\kappa G^{L^\top}(\cdot,y)\big)(x')\leq
C_{y,K}\Big(\frac{1+\log_{+}|x'|}{1+|x'|^{n-1}}\Big)\,\,\text{ at each }\,\,x'\in{\mathbb{R}}^{n-1},
\\[4pt]
&\big({\mathcal{N}}_\kappa^{\,\mathbb{R}^{n}_{+}\setminus K}(\nabla_X G^{L^\top})(\cdot\,,y)\big)(x')
\leq\frac{C_{y,K}}{1+|x'|^{n-1}}\,\,\text{ at each }\,\,x'\in{\mathbb{R}}^{n-1},
\label{GLa.1.ii}
\\[4pt]
& G^{L^\top}(\cdot\,,y)\in\big[W^{1,1}_{\rm loc}({\mathbb{R}}^n_{+})\big]^{M\times M},\qquad
G^{L^\top}(\cdot\,,y)\Big|_{\partial\mathbb{R}^{n}_{+}}^{{}^{\kappa-{\rm n.t.}}}=0
\,\,\text{ a.e. on }\,\,{\mathbb{R}}^{n-1},
\label{GLa.2}
\\[4pt]
&\text{and }\,\,\,\big(\nabla_X G^{L^\top}\big)(\cdot\,,y)
\Big|_{\partial\mathbb{R}^{n}_{+}}^{{}^{\kappa-{\rm n.t.}}}
\,\,\text{ exists at a.e. point on }\,\,{\mathbb{R}}^{n-1}.
\label{GLa.3}
\end{align}

Having fixed an index $\gamma\in\{1,\dots,M\}$, we now invoke Lemma~\ref{Lgav-TeD} 
with the function $u$ as above and with $v:=G^{L^\top}_{\cdot\,\gamma}(\cdot,y)$. 
Properties \eqref{naIaPa.4}-\eqref{GHReRFv.BNv} and \eqref{GLa.1.i}-\eqref{GLa.3} 
guarantee that all hypotheses of Lemma~\ref{Lgav-TeD} are presently satisfied. Granted this, we obtain
from \eqref{Ua-eDPa4.L} and \eqref{naIaPa.4}-\eqref{naIaPa.5} that for each aperture parameter $\kappa>0$ we have
\begin{equation}\label{Ua-eDPaVVV}
u_\gamma(y)=-\int_{\mathbb{R}^{n-1}}
\Big(u_\alpha\big|_{\partial\mathbb{R}^{n}_{+}}^{{}^{\kappa-{\rm n.t.}}}\Big)a^{\beta\alpha}_{nn}\,
\Big[\big(\partial_{X_n} G^{L^\top}_{\beta\gamma}\big)(\cdot,y)
\Big]\Big|_{\partial\mathbb{R}^{n}_{+}}^{{}^{\kappa-{\rm n.t.}}}\,d{\mathscr{L}}^{n-1}=0.
\end{equation}
Since $y$ and $\gamma$ are arbitrary, this proves that $u\equiv 0$ 
in ${\mathbb{R}}^n_{+}$, finishing the proof of \eqref{GHUNNDs}. In the 
present case, this further implies that $G_1(\cdot,y)=G_2(\cdot,y)$ 
in ${\mathbb{R}}^n_{+}\setminus\{y\}$ which shows 
that $G_1(\cdot,\cdot)=G_2(\cdot,\cdot)$ in 
${\mathbb{R}}^n_{+}\times{\mathbb{R}}^n_{+}\setminus{\rm diag}$, as wanted.

\vskip 0.08in
\noindent{\tt Step~7:} \textit{Denote by $G^{L^\top}(\cdot,\cdot)$ the Green 
function for $L^{\top}$ in ${\mathbb{R}}^n_{+}$ constructed in a similar manner 
to \eqref{Ihab-Ygab673}, with $L$ replaced by $L^{\top}$. 
Then $G^L(x,y)=\big[G^{L^\top}\!\!(y,x)\big]^\top$ for all 
$x,y\in{\mathbb{R}}^n_{+}$ with $x\not=y$.}

\vskip 0.08in
The strategy for proving this symmetry condition is to apply Lemma~\ref{Lgav-TeD} 
in the following setting. Choose two arbitrary distinct points $x,y\in{\mathbb{R}}^n_{+}$, 
and pick some aperture parameter $\kappa>0$. Take $x_1^\star:=x$, $x_2^\star:=y$, 
and fix two positive numbers $r_1$, $r_2$ which are sufficiently small (such that the conditions 
in \eqref{CPTs} are satisfied). Let us also select $\phi,\psi$ as in \eqref{Ua-eNBVVatGG} for these specifications. 
Next, fix an arbitrary $\eta,\gamma\in\{1,\dots,M\}$ and consider the vector-valued functions 
$u:=\big(G^L_{\alpha\eta}(\cdot,x)\big)_{1\leq\alpha\leq M}$ and 
$v:=\big(G^{L^\top}_{\beta\gamma}(\cdot,y)\big)_{1\leq\beta\leq M}$.

Since $L^\top$ enjoys similar properties to those of $L$, from Step~1, Step~2, and Step~5, we deduce that 
\eqref{GLa.1.i}-\eqref{GLa.3} hold with $K:=\overline{B(y,r_2)}$. Similarly, \eqref{GLa.1.i}-\eqref{GLa.3} written with 
$G^L(\cdot,x)$ in place of $G^{L^\top}(\cdot,y)$ and for $K:=\overline{B(y,r_1)}$ are also valid. 
As such, conditions \eqref{hBbac.1}-\eqref{DiBBa.LLap.BBBB} are satisfied by the functions $u,v$. 
Consequently, \eqref{Ua-eDPa4.L} yields in this case 
\begin{align}\label{Ua-eDPa4.LiR2}
&\hskip -0.40in
{}_{{\mathcal{E}}'(\mathbb{R}^{n}_{+})}\Big\langle\delta_{\beta\eta}\delta_x,
\phi\,G^{L^\top}_{\beta\gamma}(\cdot,y)\Big\rangle_{{\mathcal{E}}(\mathbb{R}^{n}_{+})}
={}_{{\mathcal{E}}'(\mathbb{R}^{n}_{+})}\Big\langle\delta_{\alpha\gamma}\delta_y,\,
\psi\,G^L_{\alpha\eta}(\cdot,x)\Big\rangle_{{\mathcal{E}}(\mathbb{R}^{n}_{+})}
\nonumber\\[6pt]
&\hskip 0.40in
+\int_{\mathbb{R}^{n-1}}
\Big(G^L_{\alpha\eta}(\cdot,x)\big|_{\partial\mathbb{R}^{n}_{+}}^{{}^{\kappa-{\rm n.t.}}}\Big) 
a^{\beta\alpha}_{nn}\,\Big[\big(\partial_{X_n} G^{L^\top}_{\beta\gamma}\big)(\cdot,y)\Big]
\Big|_{\partial\mathbb{R}^{n}_{+}}^{{}^{\kappa-{\rm n.t.}}}\,d{\mathscr{L}}^{n-1}.
\end{align}
Upon recalling that $G^L_{\alpha\eta}(\cdot,x)\big|_{\partial\mathbb{R}^{n}_{+}}^{{}^{\kappa-{\rm n.t.}}}=0$ 
a.e. in ${\mathbb{R}}^{n-1}$ and keeping in mind that $\phi(x)=\psi(y)=1$, formula \eqref{Ua-eDPa4.LiR2} 
simply becomes $G^{L^\top}_{\eta\gamma}(x,y)=G^{L}_{\gamma\eta}(y,x)$, proving the desired symmetry condition.  

\vskip 0.08in
\noindent{\tt Step~8:} \textit{If $R_L(\cdot,\cdot)$ is as in Step~3, then for any 
$\alpha,\beta\in\mathbb{N}_0^n$ there exists $C_{\alpha\beta}\in(0,\infty)$ such 
that for every $(x,y)\in{\mathbb{R}}^n_{+}\times{\mathbb{R}}^n_{+}$ one has}
\begin{equation}\label{maiTafbE}
\big|\big(\partial^\alpha_X\partial^\beta_Y R_L\big)(x,y)\big|
\leq C_{\alpha\beta}\,|x-\overline{y}|^{2-n-|\alpha|-|\beta|}\,\,
\text{ if $|\alpha|+|\beta|>0$, or $n\geq 3$}.
\end{equation}
In light of \eqref{GHCkn.Re.1A} and \eqref{Yjn1.20} it suffices to show that
\begin{equation}\label{Yjn1.GaeEE}
\begin{array}{c}
\text{for each multi-indices }\,\alpha,\beta\in\mathbb{N}_0^n
\,\text{ there exists }\,C_{\alpha\beta}\in(0,\infty)\,\text{ such that}
\\[4pt]
\big|\big(\partial^\alpha_X\partial^\beta_Y R^{(2)}_L\big)(x,y)\big|
\leq C_{\alpha\beta}|x-\overline{y}|^{2-n-|\alpha|-|\beta|}\,\,\text{ for each }\,\,
(x,y)\in{\mathbb{R}}^n_{+}\times{\mathbb{R}}^n_{+}.
\end{array}
\end{equation}
To this end, we first observe that in concert with \eqref{E-Trans} and
\eqref{Rj-CINF}, the transposition law proved in Step~7 implies 
\begin{equation}\label{eq1.11}
R_L(x,y)=\big[R_{L^\top}(y,x)\bigr]^\top,\qquad\forall\,(x,y)\in
{\mathbb{R}}^n_{+}\times{\mathbb{R}}^n_{+}.
\end{equation}
Fix an arbitrary multi-index
$\beta\in\mathbb{N}_0^n$ and note that \eqref{eq1.11} entails
\begin{equation}\label{eq1.21}
\big(\partial^\beta_Y R_L\big)(x,y)
=\Big[\big(\partial^\beta_X R_{L^\top}\big)(y,x)\Bigr]^{\top},
\qquad\forall\,(x,y)\in{\mathbb{R}}^n_{+}\times{\mathbb{R}}^n_{+}.
\end{equation}
Hence, if $R^{(1)}_{L^\top}(\cdot,\cdot)$ and $R^{(2)}_{L^\top}(\cdot,\cdot)$
are defined analogously to $R^{(1)}_L(\cdot,\cdot)$ and $R^{(2)}_L(\cdot,\cdot)$,
with $L^\top$ playing now the role of $L$, then
for each $(x,y)\in{\mathbb{R}}^n_{+}\times{\mathbb{R}}^n_{+}$ we obtain
\begin{equation}\label{eq1.21YhanPPP}
\big(\partial^\beta_Y R^{(2)}_L\big)(x,y)
=\Big[\big(\partial^\beta_X R^{(1)}_{L^\top}\big)(y,x)\Big]^{\top}
+\Big[\big(\partial^\beta_X R^{(2)}_{L^\top}\big)(y,x)\Big]^{\top}
-\big(\partial^\beta_Y R^{(1)}_L\big)(x,y).
\end{equation}
Observe that $R^{(1)}_{L^\top}(\cdot,\cdot)$ and $R^{(2)}_{L^\top}(\cdot,\cdot)$
have properties similar to those of $R^{(1)}_{L}(\cdot,\cdot)$ and $R^{(2)}_{L}(\cdot,\cdot)$, 
respectively. Granted this, from \eqref{eq1.21YhanPPP}, \eqref{eq1.20}, and \eqref{Yjn1.20} written 
for $L^\top$ in place of $L$, as well as \eqref{Yjn1.20} as originally stated, we obtain 
(upon noticing that $|y-\overline{x}|=|x-\overline{y}|$ for all $x,y\in\mathbb{R}^n_+$)
\begin{equation}\label{eq1.22}
\begin{array}{c}
\big|\big(\partial^\beta_Y R^{(2)}_L\big)(x,y)\big|
\leq C_\beta|x-\overline{y}|^{2-n-|\beta|},
\quad\forall\,(x,y)\in{\mathbb{R}}^n_{+}\times{\mathbb{R}}^n_{+},
\\[6pt]
\text{for every multi-index }\,\,\beta\in\mathbb{N}_0^n\,\,
\text{ with length }\,\,|\beta|>0.
\end{array}
\end{equation}
Let us also point out that by differentiating
\eqref{GHCkn.Re.1C} with respect to $y$ we obtain
\begin{equation}\label{GHCkn.Re.1BcD}
\begin{array}{c}
\big(\partial^\beta_Y R^{(2)}_L\big)(x,y)=(-1)^{|\beta|}
P^L_{x_n}\ast\Big(\big[\big(\partial^\beta E^L\big)(\cdot-y)
-\big(\partial^\beta E^L\big)(\cdot-\overline{y})\big]
\big|_{\partial{\mathbb{R}}^n_{+}}\Big)(x'),
\\[10pt]
\text{for each $x=(x',x_n)\in{\mathbb{R}}^n_{+}$, each   
$y\in{\mathbb{R}}^n_{+}$, and each $\beta\in{\mathbb{N}}_0^n$}.
\end{array}
\end{equation}
Much as before, this implies that for each $\beta\in{\mathbb{N}}_0^n$,
\begin{equation}\label{eq1.23}
\left\{
\begin{array}{l}
\big(\partial^\beta_Y R^{(2)}_L\big)(\cdot,y)\in\big[W^{1,2}\big({\mathbb{R}}^n_{+}\cap B(0,R)\big)\big]^{M\times M}
\,\,\text{ for each }\,\,R>0,
\\[8pt]
L\big[\big(\partial^\beta_Y R^{(2)}_L\big)(\cdot,y)\big]=0
\,\,\text{ in }\mathbb{R}^n_{+},
\\[10pt]
{\rm Tr}\,\big[\big(\partial^\beta_Y R^{(2)}_L\big)(\cdot,y)\big]
=(-1)^{|\beta|}\big[(\partial^\beta E^L)(\cdot-y)
-(\partial^\beta E^L)(\cdot-\overline{y})\big]\Big|_{\partial{\mathbb{R}}^n_{+}}
\,\,\text{ on }\,\,\mathbb{R}^{n-1}.
\end{array}
\right.
\end{equation}
With \eqref{eq1.22} and \eqref{eq1.23} in place of \eqref{Rj-222}
and \eqref{GHCkn.Re.7}-\eqref{GHCkn.Re.8BIS}, respectively, 
we can now run the same program that has led to \eqref{eq1.20}
and obtain \eqref{Yjn1.GaeEE} in the case in which $|\beta|>0$.
There remains to observe that, in the case when $|\beta|=0$, 
the claim in \eqref{Yjn1.GaeEE} is already contained in \eqref{eq1.20}.
Hence, \eqref{Yjn1.GaeEE} holds as stated.

\vskip 0.08in
\noindent{\tt Step~9:} \textit{Proof of claims in parts {\it (6)}-{\it (9)} in the statement
of Theorem~\ref{ta.av-GGG.2A}.}
\vskip 0.08in

In the case when $|\alpha|=|\beta|=0$ and $n=2$, estimate \eqref{mainest} 
is contained in \eqref{Rj}, while the case when either $|\alpha|+|\beta|>0$, or 
$n\geq 3$, follows from \eqref{GHCkn.Re.1A}, \eqref{Yjn1.20}, and \eqref{Yjn1.GaeEE}.
Moreover, that $G^L(\cdot,\cdot)$ satisfies \eqref{mainest2G}-\eqref{maKnaTTGB}
is a consequence of \eqref{defRRR}-\eqref{mainest}, \eqref{fs-est}, 
and \eqref{Ojhb}.

Next, the fact that the Green function $G^L(\cdot,\cdot)$ satisfies the 
property displayed in \eqref{mainest3G} follows from \eqref{mainest2G}
and the observation that given $N\in(0,\infty)$ and $x_\star\in{\mathbb{R}}^n$ the 
function $f$ defined by $f(x):=|x-x_\star|^{-N}$ for each $x\in{\mathbb{R}}^n_{+}$ 
has the property that
\begin{equation}\label{Uagv-9ybv5RT}
\|f\|_{L^{n/N,\infty}({\mathbb{R}}^n_{+})}\leq\big(n\,\omega_{n-1}\big)^{N/n}.
\end{equation}
Indeed, for each $\lambda>0$ we have 
\begin{align}\label{Uagv-9ybv5RT.2}
{\mathscr{L}}^n\Big(\big\{x\in{\mathbb{R}}^n_{+}:\,|f(x)|>\lambda\big\}\Big)
=&\,{\mathscr{L}}^n\Big(\big\{x\in{\mathbb{R}}^n_{+}:\,|x-x_\star|<\lambda^{-1/N}
\big\}\Big)
\nonumber\\[4pt]
\leq &\,{\mathscr{L}}^n\Big(B\big(x_\star,\lambda^{-1/N}\big)\Big)
=\frac{n\,\omega_{n-1}}{\lambda^{n/N}},
\end{align}
hence \eqref{Uagv-9ybv5RT} follows upon writing  
\begin{align}\label{Uagv-9ybv5RT.3}
\|f\|_{L^{n/N,\infty}({\mathbb{R}}^n_{+})}
=\sup\,\Big\{\lambda\,{\mathscr{L}}^n\Big(\big\{x\in{\mathbb{R}}^n_{+}:\,
|f(x)|>\lambda\big\}\Big)^{N/n}:\,\lambda>0\Big\}\leq\big(n\,\omega_{n-1}\big)^{N/n}.
\end{align}

Moving on, the goal is to show that if 
$p\in\big[1,\frac{n}{n-1}\big)$ then the properties listed in \eqref{mainest4G} 
hold for each $\zeta\in{\mathscr{C}}^\infty_c({\mathbb{R}}^n)$. To see that this 
is the case pick $\zeta\in{\mathscr{C}}^\infty_c({\mathbb{R}}^n)$ and fix an arbitrary 
$y\in{\mathbb{R}}^n_{+}$. Then by \eqref{mainest2G}-\eqref{maKnaTTGB} we have
\begin{align}\label{Uagv-9ybv5RT.5}
\big|\nabla(\zeta G^L(\cdot,y))\big|
\leq &\|\zeta\|_{L^\infty({\mathbb{R}}^n)}\big|(\nabla_X G^L)(\cdot,y)\big|+|\nabla\zeta|\big|G^L(\cdot,y)\big|
\nonumber\\[4pt]
\leq & C\|\zeta\|_{L^\infty({\mathbb{R}}^n)}{\mathbf 1}_{{\rm supp}\,\zeta}|\cdot-y|^{1-n}
\nonumber\\[4pt]
& +C\|\nabla\zeta\|_{L^\infty({\mathbb{R}}^n)}{\mathbf 1}_{{\rm supp}\,\zeta}
\left\{
\begin{array}{ll}
|\cdot-y|^{2-n} &\text{ if }\,\,n\geq 3,
\\[4pt]
1+\big|\ln|\cdot-\overline{y}|\big| &\text{ if }\,\,n=2,
\end{array}
\right.
\end{align} 
and, hence, $\nabla(\zeta G^L(\cdot,y))$ has components in $L^p({\mathbb{R}}^n_{+})$ for each 
$p\in\big[1,\frac{n}{n-1}\big)$, uniformly in $y$. In a similar fashion, 
$\zeta G^L(\cdot,y)\in\big[L^p({\mathbb{R}}^n_{+})\big]^{M\times M}$ for each 
$p\in\big[1,\frac{n}{n-1}\big)$, uniformly in $y$, which concludes the 
proof of the inequality in \eqref{mainest4G}. Finally, the fact that 
the membership in \eqref{mainest4G} also holds is a consequence of what 
we have just proved, \eqref{GHCewd-24.RRe}, and Step~2.

\vskip 0.08in
\noindent{\tt Step~10:} \textit{Proof of the claims in parts {\it (1)}-{\it (5)} and {\it (10)} 
in the statement of Theorem~\ref{ta.av-GGG.2A}.}
\vskip 0.08in

The membership in \eqref{GHCewd-22.RRe.4} in the case when $|\alpha|=|\beta|=0$ 
is contained in \eqref{hBb-TR.D2}, while the case when $|\alpha|+|\beta|>0$
follows from \eqref{mainest2G} and \eqref{GHCewd-12-best}.
The regularity property \eqref{Fvabbb-7tF} has been established in Step~2,
while the claims in \eqref{Aivb-gVV} are consequences of \eqref{GHCewd-22.RRe.4},
Proposition~\ref{p3.2.6}, and \eqref{GHCewd-24.RRe}.
As regards the translation invariance property \eqref{JKBvc-ut4}, for each 
$(x,y)\in{\mathbb{R}}^n_{+}\times{\mathbb{R}}^n_{+}\setminus{\rm diag}$ 
and $z'\in{\mathbb{R}}^{n-1}$ we may write, based on \eqref{Ihab-Ygab673}, 
\begin{align}\label{BVvtr.R3}
G^L\big(x-(z',0), & y-(z',0)\big)
\\[4pt]
= & \,E^L\big(x-(z',0)-y+(z',0)\big)
-E^L\big(x-(z',0)-\overline{y-(z',0)}\,\big)
\nonumber\\[4pt]
& -P^L_t\ast\Big(\big[E^L\big(\cdot-y+(z',0)\big)
-E^L\big(\cdot-\,\overline{y-(z',0)}\,\big)\big]
\big|_{\partial{\mathbb{R}}^n_{+}}\Big)\big((x-(z',0))'\big).
\nonumber
\end{align}
Note that the translation invariance of the reflection $y\mapsto\overline{y}$ 
readily gives
\begin{align}\label{Hbav-YRDSZ}
&\hskip -0.80in
E^L\big(x-(z',0)-y+(z',0)\big)-E^L\big(x-(z',0)-\overline{y-(z',0)}\,\big)
\nonumber\\[4pt]
&\hskip 2.00in
=E^L(x-y)-E^L(x-\overline{y}\,)
\end{align}
while the translation invariance of the convolution permits us to write 
\begin{align}\label{Hbav-YRDSZ.2}
P^L_t\ast\Big(\big[E^L\big(\cdot-y+(z',0)\big)
& -E^L\big(\cdot-\,\overline{y}+(z',0)\,\big)\big]
\big|_{\partial{\mathbb{R}}^n_{+}}\Big)(x'-z')
\nonumber\\[4pt]
= & P^L_t\ast\Big(\big[E^L(\cdot-y)-E^L(\cdot-\,\overline{y}\,)\big]
\big|_{\partial{\mathbb{R}}^n_{+}}(\cdot+z')\Big)(x'-z')
\nonumber\\[4pt]
= & P^L_t\ast\Big(\big[E^L(\cdot-y)-E^L(\cdot-\,\overline{y}\,)\big]
\big|_{\partial{\mathbb{R}}^n_{+}}\Big)(x').
\end{align}
Hence, \eqref{JKBvc-ut4} follows from \eqref{BVvtr.R3}-\eqref{Hbav-YRDSZ.2}.

Finally, the homogeneity property recorded in \eqref{mainest3G.LLL} follows from \eqref{Ihab-Ygab673}, 
item {\it (1)} in Theorem~\ref{FS-prop}, and item {\it (iv)} from Remark~\ref{Ryf-uyf} 
(bearing in mind the manner in which $P^L$ and $K^L$ are related; 
cf. part {\it (c)} of Definition~\ref{defi:Poisson}).  

\vskip 0.08in
\noindent{\tt Step~11:} \textit{The Green function $G^L(\cdot,\cdot)$ for $L$ in ${\mathbb{R}}^n_{+}$ 
may be written as in \eqref{GHCewd-2PiK}.
}
\vskip 0.08in
To see that this is the case fix $y\in{\mathbb{R}}^n_{+}$ and observe that \eqref{fs-est} gives
\begin{equation}\label{IhJG-YD-54}
\big|E^L((x',0)-y)\big|+\big|E^L((x',0)-\overline{y}\,)\big|\leq C\left\{
\begin{array}{ll}
\displaystyle
\frac{1}{1+|x'|^{n-2}} & \text{ if }\,n\geq 3,
\\[12pt]
1+|\ln|x'|| & \text{ if }\,n=2,
\end{array}
\right.
\end{equation}
for each $x'\in{\mathbb{R}}^{n-1}$, for some $C\in(0,\infty)$ depending only on $y$. 
In concert with the estimate in item {\it (a)} of Definition~\ref{defi:Poisson},
this shows that the convolution in the last term in \eqref{Ihab-Ygab673} may be decoupled. This allows us to write, 
for each $y\in{\mathbb{R}}^n_{+}$ and each $x=(x',t)\in{\mathbb{R}}^n_{+}\setminus\{y\}$, 
\begin{align}\label{Ihioh-he2}
G^L(x,y)= & E^L(x-y)
-P^L_t\ast\Big(\big[E^L(\cdot-y)\big]\big|_{\partial{\mathbb{R}}^n_{+}}\Big)(x')
\nonumber\\[4pt]
& -E^L(x-\overline{y}\,)+P^L_t\ast\Big(\big[E^L(\cdot-\overline{y}\,)
\big]\big|_{\partial{\mathbb{R}}^n_{+}}\Big)(x').
\end{align}
As such, matters have been reduced to proving that the two terms in the last line above cancel. To this end, define 
$u,v:{\mathbb{R}}^n_{+}\to{\mathbb{C}}^M$ by setting for each $x=(x',t)\in{\mathbb{R}}^n_{+}$
\begin{equation}\label{TCVV-yr6y}
u(x):=E^L(x-\overline{y}\,)\,\,\text{ and }\,\,
v(x):=P^L_t\ast\Big(\big[E^L(\cdot-\overline{y}\,)\big]\big|_{\partial{\mathbb{R}}^n_{+}}\Big)(x').
\end{equation}
Henceforth, let us assume that $n\geq 3$. Then, with 
\begin{equation}\label{TCVV-yr6y.123.FFF}
f:=E^L(\cdot-\overline{y}\,)\big|_{\partial{\mathbb{R}}^n_{+}}, 
\end{equation}
from \eqref{fs-est} we have 
\begin{equation}\label{TCVV-yr6y.123}
|f(x')|\leq C(1+|x'|)^{2-n}\,\,\text{ for each }\,\,x'\in{\mathbb{R}}^{n-1}. 
\end{equation}
In particular, 
\begin{equation}\label{TCVV-yr6y.123-bis}
f\in\bigcap_{p>\frac{n-1}{n-2}}\big[L^p({\mathbb{R}}^{n-1},{\mathcal{L}}^{n-1})\big]^{M\times M}
\subseteq\Big[L^1\Big({\mathbb{R}}^{n-1}\,,\,\frac{dx'}{1+|x'|^{n}}\Big)\Big]^{M\times M}
\end{equation}
so Proposition~\ref{thm:existence} and the properties of $E^L$ from Theorem~\ref{FS-prop} give that for each $\kappa>0$ we have
\begin{equation}\label{TCVV-yr6y.2}
u,v\in\big[{\mathscr{C}}^\infty({\mathbb{R}}^n_{+})\big]^{M\times M},\quad
Lu=Lv=0\,\,\text{ in }\,\,{\mathbb{R}}^n_{+},\,\,\text{ and }\,\,
u\Big|^{{}^{\kappa-{\rm n.t.}}}_{\partial{\mathbb{R}}^n_{+}}=f
=v\Big|^{{}^{\kappa-{\rm n.t.}}}_{\partial{\mathbb{R}}^n_{+}}.
\end{equation}
Also, from \eqref{TCVV-yr6y}, \eqref{exist:Nu-Mf}, \eqref{TCVV-yr6y.123-bis}, and the boundedness of the Hardy-Littlewood maximal operator 
we see that 
\begin{equation}\label{pauha.2}
{\mathcal{N}}_\kappa v\in\bigcap_{p>\frac{n-1}{n-2}}L^p({\mathbb{R}}^{n-1},{\mathcal{L}}^{n-1})
\subseteq L^1\Big({\mathbb{R}}^{n-1},\frac{dx'}{1+|x'|^{n-1}}\Big).
\end{equation}
Consequently, we may conclude that $u=v$ in ${\mathbb{R}}^n_{+}$, as wanted, by \eqref{GHUNNDs} 
as soon as we show that the current function $u$ satisfies \eqref{GHReRFv.BNv}. In turn, this is going to be a direct consequence of \eqref{TCVV-yr6y},
\eqref{fs-est} and a general fact to the effect that for any dimension $n\in{\mathbb{N}}$ with $n\geq 2$, any exponent $N\geq 0$, 
any compact set $K\subset{\mathbb{R}}^n$, any point $z\in\mathring{K}$, and any aperture parameter $\kappa>0$, 
there exists a finite constant $C(n,N,K,\kappa,z)>0$ with the property that 
\begin{equation}\label{UUNNpkjgr}
\sup_{y\in\Gamma_\kappa(x')\setminus K}\Big[|y-z|^{-N}\Big]\leq C(n,N,K,\kappa,z)
\big[|x'|+1\big]^{-N}\,\,\,\text{ for each }\,\,x'\in{\mathbb{R}}^{n-1}.
\end{equation}

To justify the claim made in \eqref{UUNNpkjgr}, note that it suffices to show that there exists a finite constant 
$C(n,K,\kappa,z)>0$ for which 
\begin{equation}\label{UUNNpkjgr.2}
\inf_{y\in\Gamma_\kappa(x')\setminus K}|y-z|\geq C(n,K,\kappa,z)
\max\big\{|x'|,1\big\}\,\,\,\text{ for each }\,\,x'\in{\mathbb{R}}^{n-1}.
\end{equation}
With this goal in mind, we first estimate 
\begin{equation}\label{UUNNpkjgr.3}
\inf_{y\in\Gamma_\kappa(x')\setminus K}|y-z|\geq{\rm dist}\,(z,\partial K)=:C(n,K,\kappa,z)\in(0,\infty).
\end{equation}
Thus, there remains to show that there exists a constant $C(n,K,\kappa,z)\in(0,\infty)$ such that 
\begin{equation}\label{UUNNpkjgr.4}
\inf_{y\in\Gamma_\kappa(x')\setminus K}|y-z|\geq C(n,K,\kappa,z)|x'|\,\,\,\text{ for each }\,\,x'\in{\mathbb{R}}^{n-1},
\end{equation}
To prove \eqref{UUNNpkjgr.4}, fix $x'\in{\mathbb{R}}^{n-1}$ and 
define the angle $\alpha\in(0,\pi/2)$ to be half the aperture of the cone 
$\Gamma_\kappa(x')$ (note that $\alpha$ depends exclusively on $\kappa$). 
Also, pick $R_K\in(0,\infty)$ sufficiently large so that $K\subseteq B(0,R_K)$ and define 
\begin{equation}\label{eq:yugwdsytg2qd}
C_\ast:=\frac{2R_K}{\cos\alpha\cdot{\rm dist}\,(z,\partial K)}\in(0,\infty).
\end{equation}
We next consider two cases. First, if $|x'|\leq C_\ast\cdot{\rm dist}\,(z,\partial K)$ then 
\begin{equation}\label{UUNNpkjgr.3hbs}
\inf_{y\in\Gamma_\kappa(x')\setminus K}|y-z|\geq{\rm dist}\,(z,\partial K)\geq|x'|/C_\ast
\end{equation}
so \eqref{UUNNpkjgr.4} holds in this situation. Second, if $|x'|>C_\ast\cdot{\rm dist}\,(z,\partial K)$ then 
\eqref{eq:yugwdsytg2qd} implies $|x'|>\frac{2R_K}{\cos\alpha}$. Hence ${\rm dist}\,\big(0,\Gamma_\kappa(x')\big)=|x'|\cos\alpha>R_K$, 
which forces $B(0,R_K)$ to be disjoint from $\Gamma_\kappa(x')$. In particular, $K$ is disjoint from $\Gamma_\kappa(x')$. 
Let us also observe that since $z\in K\subseteq B(0,R_K)$ we have $|z|<R_K$. Bearing these two properties in mind, we may then estimate 
\begin{align}\label{eq:YYFUaytagNb7}
|x'|\cos\alpha &={\rm dist}\,\big(0,\Gamma_\kappa(x')\big)\leq
{\rm dist}\,\big(z,\Gamma_\kappa(x')\big)+|z|
\nonumber\\[4pt]
&<\inf_{y\in\Gamma_\kappa(x')}|y-z|+\Big(\frac{R_K}{{\rm dist}\,(z,\partial K)}\Big){\rm dist}\,(z,\partial K)
\nonumber\\[4pt]
&<\inf_{y\in\Gamma_\kappa(x')\setminus K}|y-z|+\Big(\frac{R_K}{{\rm dist}\,(z,\partial K)}\Big)\frac{|x'|}{C_\ast}
\nonumber\\[4pt]
&=\inf_{y\in\Gamma_\kappa(x')\setminus K}|y-z|+\frac{|x'|\cos\alpha}{2},
\end{align}
which shows that \eqref{UUNNpkjgr.4} also holds in this case. At this stage, the proof of \eqref{UUNNpkjgr.4} is complete, 
so \eqref{UUNNpkjgr} has been justified. As noted earlier, having justified \eqref{UUNNpkjgr} concludes the proof of 
\eqref{GHCewd-2PiK} when $n\geq 3$.

A proof of \eqref{GHCewd-2PiK} which works for $n\geq 2$ goes as follows. Fix an exponent $\eta\in(0,1)$ along with 
a point $y=(y',y_n)\in{\mathbb{R}}^n_{+}$ and note that for each $x=(x',x_n)\in{\mathbb{R}}^n_{+}$ we may use \eqref{fs-est}
to estimate 
\begin{align}\label{UUNNpkjgr.13.D+M}
x_n^{1-\eta}|(\nabla E^L)(x-\overline{y})| & \leq\frac{Cx_n^{1-\eta}}{\big[|x'-y'|^2+(x_n+y_n)^2\big]^{(n-1)/2}}
\nonumber\\[4pt]
&\leq\frac{Cx_n^{1-\eta}}{(x_n+y_n)^{n-1}}\leq\frac{C}{y_n^{n-2+\eta}}.
\end{align}
As such, 
\begin{equation}\label{UUNNpkjgr.13.D+M.222}
\sup_{x=(x',x_n)\in{\mathbb{R}}^n_{+}}\Big[x_n^{1-\eta}|(\nabla E^L)(x-\overline{y})|\Big]\leq\frac{C}{y_n^{n-2+\eta}}<+\infty.
\end{equation}
In addition, 
\begin{equation}\label{UUNNpkjgr.13.D+M.333}
E^L(\cdot-\overline{y})\in\big[{\mathscr{C}}^\infty({\mathbb{R}}^n_{+})\big]^{M\times M}\,\,\text{ and }\,\,
L\big[E^L(\cdot-\overline{y})\big]=0\,\,\text{ in }\,\,{\mathbb{R}}^n_{+}.
\end{equation}
From \eqref{UUNNpkjgr.13.D+M.333} and \eqref{UUNNpkjgr.13.D+M.222} we then conclude (see, e.g., \cite[(5.11.78), p.\,490]{GHA.I}) that 
\begin{equation}\label{UUNNpkjgr.13.D+M.444}
E^L(\cdot-\overline{y})\,\,\text{ belongs to }\,\,\big[\dot{\mathscr{C}}^\eta\big(\overline{{\mathbb{R}}^n_{+}}\big)\big]^{M\times M},
\end{equation}
where $\dot{\mathscr{C}}^\eta\big(\overline{{\mathbb{R}}^n_{+}}\big)$ denotes 
the homogeneous H\"older space of order $\eta$ in $\overline{{\mathbb{R}}^n_{+}}$.
In particular, if $f$ is as in \eqref{TCVV-yr6y.123.FFF}, then
\begin{equation}\label{UUNNpkjgr.13.D+M.444.bis}
f\in\big[\dot{\mathscr{C}}^\eta({\mathbb{R}}^{n-1})\big]^{M\times M},
\end{equation}
which readily implies that, with the supremum taken over all cubes $Q\subseteq{\mathbb{R}}^{n-1}$, 
\begin{equation}\label{defi-BMO.2b.D+M}
\sup_{Q\subset\mathbb{R}^{n-1}}\ell(Q)^{-\eta}\Big(\fint_{Q}\big|f(x')-f_Q\big|^2\,dx'\Big)^{\frac{1}{2}}<+\infty,
\end{equation}
where $\ell(Q)$ stands for the side-length of $Q$ and $f_Q$ denotes the integral average of $f$ over $Q$. 
From \eqref{UUNNpkjgr.13.D+M} we also deduce that if $Q$ is an arbitrary cube in $\mathbb{R}^{n-1}$ then 
\begin{align}\label{UUNNpkjgr.13.D+M.555}
\ell(Q)^{-\eta}\Big(\fint_{Q}\Big(\int_{0}^{\ell(Q)} &\big|(\nabla E^L)\big((x',t)-\overline{y}\big)\big|^2\,t\,dt\Big)\,dx'\Big)^\frac{1}{2}
\nonumber\\[4pt]
&\leq\frac{C\ell(Q)^{-\eta}}{y_n^{n-2+\eta}}\Big(\int_{0}^{\ell(Q)}\frac{1}{t^{2(1-\eta)}}t\,dt\Big)^\frac{1}{2}
=\frac{C}{y_n^{n-2+\eta}},
\end{align}
hence 
\begin{align}\label{UUNNpkjgr.13.D+M.666}
\sup_{Q\subset\mathbb{R}^{n-1}}
\ell(Q)^{-\eta}\Big(\fint_{Q}\Big(\int_{0}^{\ell(Q)}\big|(\nabla E^L)\big((x',t)-\overline{y}\big)\big|^2\,t\,dt\Big)\,dx'\Big)^\frac{1}{2}
<+\infty.
\end{align}
Granted \eqref{UUNNpkjgr.13.D+M.333}, \eqref{TCVV-yr6y.123.FFF}, \eqref{defi-BMO.2b.D+M}, and \eqref{UUNNpkjgr.13.D+M.666}, 
we conclude from the well-posedness result established in \cite[Theorem~1.21]{MMMM16} that the functions 
$u$ and $v$ defined in \eqref{TCVV-yr6y} coincide in ${\mathbb{R}}^n_{+}$. In view of this and \eqref{Ihioh-he2}, 
formula \eqref{GHCewd-2PiK} is now established whenever $n\geq 2$. 

\vskip 0.08in
The proof of Theorem~\ref{ta.av-GGG.2A} is therefore complete. 
\end{proof}

Second, we present the proof of Theorem~\ref{FS-prop.INTR}.

\vskip 0.08in
\begin{proof}[Proof of Theorem~\ref{FS-prop.INTR}]
Work under the assumption that the system $L$ is weakly elliptic and reflection invariant.
Then, as noted in \eqref{eq:edxuye.FDa.4}, the fundamental solution $E^L$ canonically associated with $L$ as in 
Theorem~\ref{FS-prop} is reflection invariant. In such a scenario, having fixed a point $y\in{\mathbb{R}}^n_{+}$, we have
\begin{equation}\label{JKBk-at5G.Faf}
\big[E^L(\cdot-y)-E^L(\cdot-\,\overline{y}\,)\big]\big|_{\partial{\mathbb{R}}^n_{+}}=0
\end{equation}
since $\overline{(x',0)-y}=(x',0)-\overline{y}$ for each $x'\in{\mathbb{R}}^{n-1}$. Consequently, the function 
$G^L(\cdot,\cdot)$ defined in \eqref{JKBvc-ut4.YFav} satisfies \eqref{GHCewd-22.RRe}, \eqref{GHCewd-24.RRe}, and \eqref{GHCewd-23.RRe} 
in Definition~\ref{ta.av-GGG}.

The next step is to show that the function $G^L(\cdot,\cdot)$ defined in \eqref{JKBvc-ut4.YFav} also satisfies 
\eqref{bouMNN.XXX}. To this end, fix a compact set $K$ contained in ${\mathbb{R}}^n_{+}$ and pick a multi-index 
$\alpha\in\mathbb{N}_0^n$. We claim that for each $y\in\mathring{K}$ there exists a constant $C(L,K,\alpha,y)\in(0,\infty)$ 
such that 
\begin{equation}\label{jdhdfasg}
\big|(\partial^\alpha E^L)(x-y)-(\partial^\alpha E^L)(x-\overline{y})\big|
\leq\frac{C(L,K,\alpha,y)}{|x-y|^{n-1+|\alpha|}}\,\,\text{ for all }\,\,x\in{\mathbb{R}}^n_{+}\setminus K.
\end{equation}
To see why this is true, pick $y=(y_1,\dots,y_n)\in\mathring{K}$ along with $x\in{\mathbb{R}}^n_{+}\setminus K$, and consider two cases. 

The first case corresponds to $|x-y|\geq 4y_n$. Since for any $\xi$ between $y$ and $\overline{y}$ we have 
$|x-y|\leq |x-\xi|+|\xi-y|\leq |x-\xi|+2y_n\leq |x-\xi|+\frac{1}{2}|x-y|$, it follows that for all points $\xi$ between 
$y$ and $\overline{y}$ we have $|x-\xi|\geq\frac{1}{2}|x-y|>0$. This, the Mean Value Theorem, and \eqref{fs-est} then imply
\begin{align}\label{ABB-lm-jL.1}
\big|(\partial^\alpha E^L)(x-y)-(\partial^\alpha E^L)(x-\overline{y})\big|
&\leq 2y_n\sup_{\xi\in[y,\overline{y}]}\big|(\nabla\partial^\alpha E^L)(x-\xi)\big|
\nonumber\\[4pt]
&\leq 2y_n\sup_{\xi\in[y,\overline{y}]}\frac{C(L,\alpha)}{|x-\xi|^{n-1+|\alpha|}}
\nonumber\\[4pt]
&\leq\frac{C(L,\alpha,y)}{|x-y|^{n-1+|\alpha|}}.
\end{align}

In the second case, suppose $|x-y|<4y_n$. Since also $|x-y|\geq{\rm dist}(y,\partial K)>0$, it follows that 
\begin{equation}\label{mjxbsh}
\text{$|x-y|\approx 1$ with proportionality constants depending only on $K$ and $y$.}
\end{equation}
Moreover, $|x-\overline{y}|\geq{\rm dist}(\overline{y},{\mathbb{R}}^n_{+})=y_n$ and
$|x-\overline{y}|\leq |x-y|+|y-\overline{y}|<4y_n+2y_n=6y_n$, thus 
\begin{equation}\label{mjxbsh.b}
\text{$|x-\overline{y}|\approx 1$ with proportionality constants depending only on $y$.}
\end{equation}
From \eqref{fs-est} and \eqref{mjxbsh}-\eqref{mjxbsh.b}, we see that there exists 
$C(L,K,\alpha,y)\in(0,\infty)$ such that 
\begin{align}\label{ABB-lm-jL.2}
\big|(\partial^\alpha E^L)(x-y)-(\partial^\alpha E^L)(x-\overline{y})\big|
&\leq\big|(\partial^\alpha E^L)(x-y)\big|+\big|(\partial^\alpha E^L)(x-\overline{y})\big|
\nonumber\\[4pt]
&\leq\frac{C(L,K,\alpha,y)}{|x-y|^{n-1+|\alpha|}}.
\end{align}
Together, \eqref{ABB-lm-jL.1} and \eqref{ABB-lm-jL.2} take care of the estimate in \eqref{jdhdfasg}. 
In turn, \eqref{jdhdfasg} and \eqref{JKBvc-ut4.YFav} imply
\begin{align}\label{ABB-lm-jL.3}
\big|(\partial^\alpha_X G^L)(x,y)\big|\leq\frac{C(L,K,\alpha,y)}{|x-y|^{n-1+|\alpha|}}
\,\,\text{ for all }\,\,x\in{\mathbb{R}}^n_{+}\setminus K.
\end{align}
With this in hand, for each $x'\in{\mathbb{R}}^{n-1}$ we may invoke \eqref{UUNNpkjgr} to estimate
\begin{align}\label{ABB-lm-jL.4}
\Big(\mathcal{N}^{\,{\mathbb{R}}^n_{+}\setminus K}_\kappa(\partial^\alpha_X G^L)(\cdot,y)\Big)(x')
&\leq \Big(\mathcal{N}^{\,{\mathbb{R}}^n_{+}\setminus K}_\kappa\Big(\frac{C_{L,K,\alpha,y}}{|\cdot-y|^{n-1+|\alpha|}}\Big)\Big)(x')
\nonumber\\[4pt]
&\leq\frac{C_{L,K,\alpha,\kappa,y}}{1+|x'|^{n-1+|\alpha|}},
\end{align}
as claimed in \eqref{bouMNN.XXX}. In turn, from \eqref{bouMNN.XXX} (used with $|\alpha|=0$) it clear that 
\eqref{GHCewd-25.RRe} is also satisfied. As such, we may now conclude that $G^L(\cdot,\cdot)$ defined in \eqref{JKBvc-ut4.YFav}
is a Green function for $L$ in $\mathbb{R}^n_{+}$, in the sense of Definition~\ref{ta.av-GGG}.

Moving on, let us prove that there exists precisely one Green function for $L$ in ${\mathbb{R}}^n_{+}$, 
in the sense of Definition~\ref{ta.av-GGG}. The argument so far shows that if 
$L=\Big(a^{\alpha\beta}_{rs}\partial_r\partial_s\Big)_{1\leq\alpha,\beta\leq M}$ is assumed to be weakly elliptic and reflection invariant then it has at least one Green function in ${\mathbb{R}}^n_{+}$, in the sense of Definition~\ref{ta.av-GGG}. 
To establish that this is unique, assume $G_1(\cdot,\cdot)$, $G_2(\cdot,\cdot)$ are two such Green functions. Fix an arbitrary 
point $y\in{\mathbb{R}}^n_{+}$ and define 
\begin{equation}\label{eq:UUU}
u:=G_1(\cdot,y)-G_2(\cdot,y)\,\,\text{ in }\,\,{\mathbb{R}}^n_{+}\setminus\{y\}. 
\end{equation}
Then properties \eqref{GHCewd-22.RRe} and \eqref{GHCewd-23.RRe} guarantee that 
$u\in\big[L^1_{\rm loc}({\mathbb{R}}^n_{+})\big]^{M\times M}$ and $Lu=0$ in 
$\big[{\mathcal{D}}'({\mathbb{R}}^n_{+})\big]^{M\times M}$. Elliptic regularity therefore ensures that 
\begin{equation}\label{naIaPa.4.D+M}
u\in\big[{\mathscr{C}}^\infty({\mathbb{R}}^n_{+})\big]^{M\times M}
\,\,\text{ and }\,\,Lu=0\,\,\text{ pointwise in }\,\,{\mathbb{R}}^n_{+}.
\end{equation}
In addition, \eqref{GHCewd-24.RRe} guarantees that there exists $\kappa>0$ such that 
\begin{equation}\label{naIaPa.5.D+M}
u\big|^{{}^{\kappa-{\rm n.t.}}}_{\partial{\mathbb{R}}^n_{+}}=0\,\,\text{ a.e. in }\,\,{\mathbb{R}}^{n-1}.
\end{equation}
Furthermore, from the finiteness condition in \eqref{GHCewd-25.RRe} and the first property in \eqref{naIaPa.4.D+M} we see that 
$\int_{{\mathbb{R}}^{n-1}}\big({\mathcal{N}}_{\kappa}u\big)(x')\frac{dx'}{1+|x'|^{n-1}}<+\infty$. In particular,
\begin{equation}\label{GHReRFv.BNv.D+M}
\int_{{\mathbb{R}}^{n-1}}\big({\mathcal{N}}_{\kappa}u\big)(x')\frac{dx'}{1+|x'|^{n}}<+\infty.
\end{equation}

Next, we claim that 
\begin{equation}\label{GHUNNDs.D+M}
\parbox{10.00cm}{if $L$ is weakly elliptic and reflection invariant then any function $u$ 
satisfying \eqref{naIaPa.4.D+M}-\eqref{GHReRFv.BNv.D+M} vanishes identically in ${\mathbb{R}}^n_{+}$.}
\end{equation}
To justify this claim, observe from \eqref{JKBvc-ut4.YFav.a.222} that $L$ is reflection invariant if and only if 
$L^\top$ is reflection invariant. Also, it is clear from definitions that $L$ is weakly elliptic if and only if 
$L^\top$ is weakly elliptic. Thus, it makes sense to consider 
$G^{L^\top}(\cdot,\cdot)=\big(G^{L^\top}_{\alpha\gamma}(\cdot,\cdot)\big)_{1\leq\alpha,\gamma\leq M}$, 
the Green function defined as in \eqref{JKBvc-ut4.YFav} in relation to the operator $L^\top$, i.e., 
\begin{equation}\label{JKBvc-ut4.YFav.bxy}
G^{L^\top}(x,y):=E^{L^\top}(x-y)-E^{L^\top}(x-\overline{y})
\,\,\text{ for all }\,\,(x,y)\in{\mathbb{R}}^n_{+}\times{\mathbb{R}}^n_{+}\setminus{\rm diag}.
\end{equation}
To proceed, fix an arbitrary point $y\in{\mathbb{R}}^n_{+}$. Results established in the first part in the current proof
(with $L^\top$ in place of $L$) show that for each compact subset $K$ of ${\mathbb{R}}^n_{+}$ with $y\in\mathring{K}$
there exists a constant $C_{y,K}\in(0,\infty)$ with the property that 
\begin{align}\label{GLa.1.i.D+M}
&\big(\mathcal{N}^{\,\mathbb{R}^{n}_{+}\setminus K}_\kappa G^{L^\top}(\cdot,y)\big)(x')\leq
\frac{C_{y,K}}{1+|x'|^{n-1}}\,\,\text{ for each }\,\,x'\in{\mathbb{R}}^{n-1},
\\[4pt]
&\big({\mathcal{N}}_\kappa^{\,\mathbb{R}^{n}_{+}\setminus K}(\nabla_X G^{L^\top})(\cdot\,,y)\big)(x')
\leq\frac{C_{y,K}}{1+|x'|^{n}}\,\,\text{ for each }\,\,x'\in{\mathbb{R}}^{n-1},
\label{GLa.1.ii.D+M}
\\[4pt]
& G^{L^\top}(\cdot\,,y)\in\big[W^{1,1}_{\rm loc}({\mathbb{R}}^n_{+})\big]^{M\times M},\qquad
G^{L^\top}(\cdot\,,y)\Big|_{\partial\mathbb{R}^{n}_{+}}^{{}^{\kappa-{\rm n.t.}}}=0
\,\,\text{ a.e. on }\,\,{\mathbb{R}}^{n-1},
\label{GLa.2.D+M}
\\[4pt]
&\text{and }\,\,\,\big(\nabla_X G^{L^\top}\big)(\cdot\,,y)\Big|_{\partial\mathbb{R}^{n}_{+}}^{{}^{\kappa-{\rm n.t.}}}
\,\,\text{ exists on }\,\,{\mathbb{R}}^{n-1},
\label{GLa.3.D+M}
\\[4pt]
& L^{\top}\big[G^{L^{\top}}(\cdot\,,y)\big]=\delta_{y}\,I_{M\times M}\,\,\text{ in }\,\,\big[{\mathcal{D}}'({\mathbb{R}}^n_{+})\big]^{M\times M},
\label{GHCewd-23.RRe.TTT}
\end{align}
with the last property visible from \eqref{JKBvc-ut4.YFav.bxy}. 

At this stage, fix an index $\eta\in\{1,\dots,M\}$ and, with $u$ as in \eqref{eq:UUU}, define $w:=(u_{\gamma\eta})_{1\leq\gamma,\eta\leq M}$. 
Next, fix $\gamma\in\{1,\dots,M\}$ and introduce $v:=G^{L^\top}_{\cdot\,\gamma}(\cdot,y)$. 
We apply Lemma~\ref{Lgav-TeD} with $u$ replaced by $w$ and $v$ as above. 
Properties \eqref{naIaPa.4.D+M}-\eqref{GHReRFv.BNv.D+M} and \eqref{GLa.1.i.D+M}-\eqref{GLa.3.D+M} 
ensure that the hypotheses of Lemma~\ref{Lgav-TeD} are verified. As such, we obtain
from \eqref{Ua-eDPa4.L} and \eqref{naIaPa.4.D+M}-\eqref{naIaPa.5.D+M} that for each aperture parameter $\kappa>0$ we have
\begin{equation}\label{Ua-eDPaVVV.D+M}
u_{\gamma\eta}(y)=-\int_{\mathbb{R}^{n-1}}
\Big(u_{\alpha\eta}\big|_{\partial\mathbb{R}^{n}_{+}}^{{}^{\kappa-{\rm n.t.}}}\Big)a^{\beta\alpha}_{nn}\,
\Big[\big(\partial_{X_n} G^{L^\top}_{\beta\gamma}\big)(\cdot,y)
\Big]\Big|_{\partial\mathbb{R}^{n}_{+}}^{{}^{\kappa-{\rm n.t.}}}\,d{\mathscr{L}}^{n-1}=0.
\end{equation}
In view of the fact that $y$ and $\gamma$, $\eta$ are arbitrary, this forces $u\equiv 0$ in ${\mathbb{R}}^n_{+}$, 
finishing the proof of \eqref{GHUNNDs.D+M}. Presently, this further implies $G_1(\cdot,\cdot)=G_2(\cdot,\cdot)$ in 
${\mathbb{R}}^n_{+}\times{\mathbb{R}}^n_{+}\setminus{\rm diag}$, as desired.

Going further, the properties claimed in \eqref{GHCewd-22.RRe.5.D+M} are clear from 
the definition of $G^L(\cdot,\cdot)$ given in \eqref{JKBvc-ut4.YFav}, the fact that the fundamental solution 
$E^L$ canonically associated with $L$ as in Theorem~\ref{FS-prop} is even (cf. \eqref{smmth-odd}), 
reflection invariant, and satisfies the transposition property recorded in item {\it (5)} of Theorem~\ref{FS-prop}.
Finally, that in the case when $L$ is as in \eqref{L-def}-\eqref{L-ell.X} and also reflection invariant the Green function 
constructed in \eqref{JKBvc-ut4.YFav} coincides with the Green function from Theorem~\ref{ta.av-GGG.2A} is clear from the 
uniqueness properties enjoyed by said Green functions. The proof of Theorem~\ref{FS-prop.INTR} is therefore complete. 
\end{proof}

Next we give the proof of Corollary~\ref{y6tFFF.cc}.

\vskip 0.08in
\begin{proof}[Proof of Corollary~\ref{y6tFFF.cc}]
Fix some function $f=(f_\alpha)_{1\leq\alpha\leq M}\in\big[{\mathscr{C}}^\infty_c({\mathbb{R}}^{n-1})\big]^M$, 
along with some point $x=(x',t)\in{\mathbb{R}}^n_{+}$, and some index $\gamma\in\{1,\dots,M\}$. 
Specializing \eqref{Ua-eDPa4.L} to the case when $u(z',s):=(P^L_s\ast f)(z')$ at each 
$(z',s)\in{\mathbb{R}}^n_{+}$, and $v:=\big(G^{L^\top}_{\beta\gamma}\big(\cdot,(x',t)\big)\big)_{1\leq\beta\leq M}$ 
yields the conclusion that for each $\gamma\in\{1,\dots,M\}$ and each 
$(x',t)\in{\mathbb{R}}^n_{+}$ we have (for any chosen aperture parameter $\kappa>0$) 
\begin{align}\label{Ua-eDPa4.LBBw}
\int_{\mathbb{R}^{n-1}} & \big(P^L_{\gamma\alpha}\big)_t(x'-z')f_\alpha(z')\,dz'
=(P^L_t\ast f)_\gamma(x')=u_\gamma(x',t)
\nonumber\\[6pt]
&=-\int_{\mathbb{R}^{n-1}}f_\alpha(z')a^{\beta\alpha}_{nn}
\Big(\Big[\big(\partial_{X_n} G^{L^\top}_{\beta\gamma}\big)\big(\cdot,(x',t)\big)\Big]
\Big|_{\partial\mathbb{R}^{n}_{+}}^{{}^{\kappa-{\rm n.t.}}}\Big)(z',0)\,dz'
\nonumber\\[6pt]
&=-\int_{\mathbb{R}^{n-1}}f_\alpha(z')a^{\beta\alpha}_{nn}
\big(\partial_{X_n} G^{L^\top}_{\beta\gamma}\big)\big((z',0),(x',t)\big)\,dz'
\nonumber\\[6pt]
&=-\int_{\mathbb{R}^{n-1}}f_\alpha(z')a^{\beta\alpha}_{nn}
\big(\partial_{Y_n} G^{L}_{\gamma\beta}\big)\big((x',t),(z',0)\big)\,dz'
\nonumber\\[6pt]
&=-\int_{\mathbb{R}^{n-1}}f_\alpha(z')a^{\beta\alpha}_{nn}
\big(\partial_{Y_n} G^{L}_{\gamma\beta}\big)\big((x'-z',t),0\big)\,dz',
\end{align}
on account of \eqref{GHCewd-22.RRe.5}, \eqref{Fvabbb-7tF}, and \eqref{JKBvc-ut4}.
The arbitrariness of $f$ then forces that for every $\alpha,\gamma\in\{1,\dots,M\}$ we have
\begin{equation}\label{Ua-eD.LanTF}
\big(P^L_{\gamma\alpha}\big)_t(x'-z')=-a^{\beta\alpha}_{nn}
\big(\partial_{Y_n} G^{L}_{\gamma\beta}\big)\big((x'-z',t),0\big)
\end{equation}
for each $x'\in{\mathbb{R}}^{n-1}$, $t>0$, and a.e. $z'\in{\mathbb{R}}^{n-1}$.
Given the continuity properties of the functions involved, it follows that \eqref{Ua-eD.LanTF} actually 
holds for every $z'\in{\mathbb{R}}^{n-1}$. Further specializing this identity to the case when $z':=0'$ 
and $t:=1$ then yields \eqref{Ua-eD.LBBw.3B}. 

Since the Green function for $L$ is unique (cf. Step~6 in the proof of Theorem~\ref{ta.av-GGG.2A}) 
we conclude from \eqref{Ua-eD.LBBw.3B} that the Poisson kernel for $L$ in ${\mathbb{R}}^n_{+}$ 
(whose existence is guaranteed by Theorem~\ref{ya-T4-fav}) is unique as well. 

Finally, the last claim in the statement of Corollary~\ref{y6tFFF.cc} (pertaining to \eqref{Ua-eD.kab}) 
is a consequence of \eqref{Ua-eD.LBBw.3B} and \eqref{JKBvc-ut4.YFav}.
\end{proof}

We continue by proving Theorem~\ref{thm:FP}.

\vskip 0.08in
\begin{proof}[Proof of Theorem~\ref{thm:FP}]
Throughout, assume $L$ is an $M\times M$ system with constant complex coefficients as in \eqref{L-def}-\eqref{L-ell.X}.
In a first stage, suppose $u$ is as in \eqref{jk-lm-jhR-LLL-HM-RN.w} and satisfies the additional assumption  
\begin{eqnarray}\label{Tafva.2222.XXX}
u\big|^{{}^{\kappa-{\rm n.t.}}}_{\partial{\mathbb{R}}^n_{+}}
\,\,\text{ exists at ${\mathscr{L}}^{n-1}$-a.e. point in }\,\,{\mathbb{R}}^{n-1}.
\end{eqnarray}
Then for each $x=(x',t)\in{\mathbb{R}}^n_{+}$ and any $\gamma\in\{1.\dots,M\}$ we have 
\begin{align}\label{Ua-eDPa4.LBBw.XXX}
u_\gamma(x) &=-\int_{\mathbb{R}^{n-1}}\Big(u_\alpha\big|^{{}^{\kappa-{\rm n.t.}}}_{\partial{\mathbb{R}}^n_{+}}\Big)(z')a^{\beta\alpha}_{nn}
\Big(\Big[\big(\partial_{X_n}G^{L^\top}_{\beta\gamma}\big)\big(\cdot,(x',t)\big)\Big]
\Big|_{\partial\mathbb{R}^{n}_{+}}^{{}^{\kappa-{\rm n.t.}}}\Big)(z',0)\,dz'
\nonumber\\[6pt]
&=-\int_{\mathbb{R}^{n-1}}\Big(u_\alpha\big|^{{}^{\kappa-{\rm n.t.}}}_{\partial{\mathbb{R}}^n_{+}}\Big)(z')a^{\beta\alpha}_{nn}
\big(\partial_{Y_n} G^{L}_{\gamma\beta}\big)\big((x'-z',t),0\big)\,dz'
\nonumber\\[6pt]
&=\int_{\mathbb{R}^{n-1}}\big(P^L_{\gamma\alpha}\big)_t(x'-z')
\Big(u_\alpha\big|^{{}^{\kappa-{\rm n.t.}}}_{\partial{\mathbb{R}}^n_{+}}\Big)(z')\,dz'.
\end{align}
Above, the first equality is a consequence of \eqref{Ua-eDPa4.L} (used with $v:=G^{L^\top}_{\cdot\gamma}(\cdot,x)$ and 
$K_2$ a small compact neighborhood of $x$ in ${\mathbb{R}}^n_{+}$), whose applicability in the present context is ensured 
by \eqref{jk-lm-jhR-LLL-HM-RN.w}, \eqref{Tafva.2222.XXX}, as well as \eqref{Fvabbb-7tF} and \eqref{bound-NK-G}-\eqref{bouMNN} 
(written for $L^\top$ in place of $L$). The second equality in \eqref{Ua-eDPa4.LBBw.XXX} may be justified much as in the case 
of (the last part of) \eqref{Ua-eDPa4.LBBw}. Lastly, the final equality in \eqref{Ua-eDPa4.LBBw.XXX} comes from \eqref{Ua-eD.LanTF}.
Having established \eqref{Ua-eDPa4.LBBw.XXX}, all desired conclusions (in \eqref{Tafva.2222}) readily follow. 

To dispense with the additional assumption made in \eqref{Tafva.2222.XXX}, suppose $u$ is precisely as in \eqref{jk-lm-jhR-LLL-HM-RN.w}. 
Running the previous argument with $u(\cdot,\cdot+\varepsilon)$ in place of $u$, where $\varepsilon>0$ is arbitrary, 
we obtain that 
\begin{align}\label{AA-lm-jLL-HM-RN.ppp.wsa}
u(x',t+\varepsilon)=\int_{{\mathbb{R}}^{n-1}}P^L_t(x'-y')f_\varepsilon(y')\,dy'
\,\,\,\text{ for each }\,\,x=(x',t)\in{\mathbb{R}}^n_{+},
\end{align}
where we have abbreviated
\begin{align}\label{AA-lm-jLL-HM-RN.ppp.wsa.f}
f_\varepsilon:=u(\cdot,\varepsilon):{\mathbb{R}}^{n-1}\longrightarrow{\mathbb{C}}^M
\,\,\text{ for each }\,\,\varepsilon>0.
\end{align}
If we also consider the weight 
\begin{equation}\label{eq:16t44.iii.D+M}
\begin{array}{c}
\text{$\omega:{\mathbb{R}}^{n-1}\to(0,\infty)$ defined as} 
\\[4pt]
\text{$\omega(x'):=(1+|x'|^{n-1})^{-1}$ for each $x'\in{\mathbb{R}}^{n-1}$}, 
\end{array}
\end{equation}
then the last condition in \eqref{jk-lm-jhR-LLL-HM-RN.w} entails  
\begin{equation}\label{eq:16t44.iii}
\sup_{\varepsilon>0}|f_\varepsilon|\leq{\mathcal{N}}_{\kappa} u\in 
L^1\big({\mathbb{R}}^{n-1}\,,\,\omega{\mathscr{L}}^{n-1}\big). 
\end{equation}
Granted this, the weak-$\ast$ convergence result from Lemma~\ref{ydadHBB} may be used for the 
sequence $\big\{f_\varepsilon\big\}_{\varepsilon>0}\subset L^1\big({\mathbb{R}}^{n-1}\,,\,\omega{\mathscr{L}}^{n-1}\big)$
to conclude that there exists some $f\in L^1\big({\mathbb{R}}^{n-1}\,,\,\omega{\mathscr{L}}^{n-1}\big)$ 
and some sequence $\{\varepsilon_j\}_{j\in{\mathbb{N}}}\subset(0,\infty)$ which converges to zero 
with the property that 
\begin{equation}\label{eq:16t44.MAD}
\lim_{j\to\infty}\int_{{\mathbb{R}}^{n-1}}\varphi(y')f_{\varepsilon_j}(y')\frac{dy'}{1+|y'|^{n-1}}
=\int_{{\mathbb{R}}^{n-1}}\varphi(y')f(y')\frac{dy'}{1+|y'|^{n-1}}
\end{equation}
for every bounded continuous function $\varphi$ in ${\mathbb{R}}^{n-1}$.
The fact that there exists a constant $C\in(0,\infty)$ for which  
\begin{equation}\label{eq:16t44.iagG}
|P^L(z')|\leq\frac{C}{(1+|z'|^2)^{n/2}}\,\,\text{ for each }\,\,z'\in{\mathbb{R}}^{n-1}
\end{equation}
(see item {\it (a)} of Definition~\ref{defi:Poisson}) 
ensures for each fixed point $(x',t)\in{\mathbb{R}}^n_{+}$ the assignment 
\begin{equation}\label{Fvabbb-7tF.Tda}
\begin{array}{c}
\displaystyle
{\mathbb{R}}^{n-1}\ni y'\mapsto\varphi(y'):=(1+|y'|^{n-1})P^L_t(x'-y')\in{\mathbb{C}}^{M\times M}
\\[6pt]
\text{is a bounded continuous function in ${\mathbb{R}}^{n-1}$}.
\end{array}
\end{equation}
At this stage, from \eqref{AA-lm-jLL-HM-RN.ppp.wsa} and \eqref{eq:16t44.MAD} 
used for the function $\varphi$ defined in \eqref{Fvabbb-7tF.Tda} we obtain 
(bearing in mind that $u$ is continuous in ${\mathbb{R}}^n_{+}$) that 
\begin{align}\label{AA-lm-jLL-HM-RN.ppp.wsa.2}
u(x',t)=\int_{{\mathbb{R}}^{n-1}}P^L_t(x'-y')f(y')\,dy'\,\,\text{ for each }\,\,x=(x',t)\in{\mathbb{R}}^n_{+}.
\end{align}
With this in hand, and since $L^1\big({\mathbb{R}}^{n-1}\,,\,\omega{\mathscr{L}}^{n-1}\big)
\subseteq L^1\Big({\mathbb{R}}^{n-1}\,,\,\displaystyle\frac{dx'}{1+|x'|^{n}}\Big)$, we may invoke 
Proposition~\ref{thm:existence} to conclude that
\begin{equation}\label{exist:u2.b}
\text{$u\big|^{{}^{\kappa-{\rm n.t.}}}_{\partial\mathbb{R}^{n}_{+}}$ exists and equals $f$ 
at ${\mathscr{L}}^{n-1}$-a.e. point in $\mathbb{R}^{n-1}$}.
\end{equation}
Once this has been established, all conclusions in \eqref{Tafva.2222} are implied by 
\eqref{AA-lm-jLL-HM-RN.ppp.wsa.2}-\eqref{exist:u2.b}. 

In the second part of the proof, make the additional assumption that $L$ is reflection invariant  
and assume now that $u$ is as in \eqref{jk-lm-jhR-LLL-HM-RN.w.BIS}. As before, we first work under 
the additional hypothesis made in \eqref{Tafva.2222.XXX}. Granted this, all equalities in \eqref{Ua-eDPa4.LBBw.XXX}
continue to be valid. Most notably, while now $u$ satisfies the weaker integrability property in 
\eqref{jk-lm-jhR-LLL-HM-RN.w.BIS}, the fact that $L$ is reflection invariant ensures that the Green function 
enjoys the stronger property listed in \eqref{bouMNN.XXX}. In concert, these ultimately allow us to invoke 
the integration by parts formula \eqref{Ua-eDPa4.L}. Everything else in \eqref{Ua-eDPa4.LBBw.XXX} is justified as before, 
given the compatibility between the Green functions constructed in Theorems~\ref{ta.av-GGG.2A}-\ref{FS-prop.INTR}.

To dispense with the additional hypothesis made in \eqref{Tafva.2222.XXX}, we employ the same weak-$\ast$ argument as before, 
with the caveat that in place of \eqref{eq:16t44.iii.D+M} we now choose the weight 
\begin{equation}\label{eq:16t44.iii.D+M.222}
\begin{array}{c}
\text{$\omega:{\mathbb{R}}^{n-1}\to(0,\infty)$ defined as} 
\\[4pt]
\text{$\omega(x'):=(1+|x'|^{n})^{-1}$ for each $x'\in{\mathbb{R}}^{n-1}$}.
\end{array}
\end{equation}
Such a choice works well with the last condition in \eqref{jk-lm-jhR-LLL-HM-RN.w.BIS}, since it guarantees that   
\eqref{eq:16t44.iii} continues to be presently valid. This opens the door for applying the weak-$\ast$ convergence result 
from Lemma~\ref{ydadHBB} to the sequence $\big\{f_\varepsilon\big\}_{\varepsilon>0}\subset 
L^1\big({\mathbb{R}}^{n-1}\,,\,\omega{\mathscr{L}}^{n-1}\big)$ and conclude that there exists 
$f\in L^1\big({\mathbb{R}}^{n-1}\,,\,\omega{\mathscr{L}}^{n-1}\big)$ 
along with a sequence $\{\varepsilon_j\}_{j\in{\mathbb{N}}}\subset(0,\infty)$ which converges to zero such that 
\begin{equation}\label{eq:16t44.MAD.D+M}
\lim_{j\to\infty}\int_{{\mathbb{R}}^{n-1}}\varphi(y')f_{\varepsilon_j}(y')\frac{dy'}{1+|y'|^{n}}
=\int_{{\mathbb{R}}^{n-1}}\varphi(y')f(y')\frac{dy'}{1+|y'|^{n}}
\end{equation}
for every bounded continuous function $\varphi$ in ${\mathbb{R}}^{n-1}$. Recall the estimate for the 
Poisson kernel recorded in \eqref{eq:16t44.iagG}. This presently implies that  
for each fixed point $(x',t)\in{\mathbb{R}}^n_{+}$ the assignment 
\begin{equation}\label{Fvabbb-7tF.Tda.D+M}
\begin{array}{c}
\displaystyle
{\mathbb{R}}^{n-1}\ni y'\mapsto\varphi(y'):=(1+|y'|^{n})P^L_t(x'-y')\in{\mathbb{C}}^{M\times M}
\\[6pt]
\text{is a bounded continuous function in ${\mathbb{R}}^{n-1}$}.
\end{array}
\end{equation}
At this stage, from \eqref{AA-lm-jLL-HM-RN.ppp.wsa} and \eqref{eq:16t44.MAD.D+M} 
used for the function $\varphi$ defined in \eqref{Fvabbb-7tF.Tda.D+M} we obtain 
(in view of the continuity of $u$ in ${\mathbb{R}}^n_{+}$) that the Poisson integral representation 
formula from \eqref{AA-lm-jLL-HM-RN.ppp.wsa.2} is presently true. 
Granted this, and since we now actually have $L^1\big({\mathbb{R}}^{n-1}\,,\,\omega{\mathscr{L}}^{n-1}\big)
=L^1\Big({\mathbb{R}}^{n-1}\,,\,\displaystyle\frac{dx'}{1+|x'|^{n}}\Big)$, we may rely on  
Proposition~\ref{thm:existence} to once again conclude \eqref{exist:u2.b}.
This establishes the first property in \eqref{Tafva.2222.BIS}. 

The second property in \eqref{Tafva.2222.BIS} follows from the first and the second line in \eqref{jk-lm-jhR-LLL-HM-RN.w.BIS}
(bearing in mind that $u\big|^{{}^{\kappa-{\rm n.t.}}}_{\partial\mathbb{R}^{n}_{+}}$ is a measurable function, as seen from 
\eqref{exist:u2.b}). The final property in \eqref{Tafva.2222.BIS} is clear from \eqref{exist:u2.b} and \eqref{AA-lm-jLL-HM-RN.ppp.wsa.2}.
\end{proof}

We conclude this section with the proof of Corollary~\ref{Them-Gen}.

\vskip 0.08in
\begin{proof}[Proof of Corollary~\ref{Them-Gen}]
Generally speaking, for each $m\geq 0$ the membership 
$f\in{\mathscr{Z}}_{m}$ entails $f\in L^1\big({\mathbb{R}}^{n-1}\,,\,\tfrac{dx'}{1+|x'|^{m}}\big)$,
since $f$ is measurable and $|f|\leq{\mathcal{M}}f$ a.e. in ${\mathbb{R}}^{n-1}$. Thus, for each $m\geq 0$ we have a 
continuous embedding
\begin{equation}\label{eq:iuwtrGajq.D+M}
{\mathscr{Z}}_{m}\hookrightarrow L^1\big({\mathbb{R}}^{n-1}\,,\,\tfrac{dx'}{1+|x'|^{m}}\big).
\end{equation}

As far as the boundary value problem \eqref{jk-lm-jhR-LLL-HM-RN.w.BVP} is concerned, a solution is obtained 
by defining $u(x):=(P^L_t\ast f)(x')$ for each $x=(x',t)\in{\mathbb{R}}^n_{+}$. 
Since $f\in[{\mathscr{Z}}_{n-1}]^M$, the embedding in \eqref{eq:iuwtrGajq.D+M} (with $m:=n-1$) entails 
$f\in\Big[L^1\big({\mathbb{R}}^{n-1}\,,\,\tfrac{dx'}{1+|x'|^{n-1}}\big)\big]^M$. Granted this, Proposition~\ref{thm:existence}
guarantees that $u$ is a meaningfully defined smooth function in ${\mathbb{R}}^n_{+}$, satisfying $Lu=0$ in ${\mathbb{R}}^n_{+}$ and
\begin{align}\label{hytftrdtrfahafa767fd}
\int_{\mathbb{R}^{n-1}}\big({\mathcal{N}}_{\kappa}u\big)(x')\,\frac{dx'}{1+|x'|^{n-1}}
&\leq C\int_{\mathbb{R}^{n-1}}\big({\mathcal{M}}f\big)(x')\,\frac{dx'}{1+|x'|^{n-1}}
\nonumber\\[4pt]
&=C\|f\|_{[{\mathscr{Z}}_{n-1}]^M}<\infty.
\end{align}
These properties show that $u$ is indeed a solution of \eqref{jk-lm-jhR-LLL-HM-RN.w.BVP}.
The fact that this solution is unique then follows from \eqref{GHUNNDs}.

Under the additional assumption that $L$ is also reflection invariant, the well-posedness of the boundary value problem 
\eqref{jk-lm-jhR-LLL-HM-RN.w.BVP.D+M} is justified in a similar fashion, with uniqueness now guaranteed by \eqref{GHUNNDs.D+M}.
\end{proof}


\begin{thebibliography}{99999}
\parskip=0.1cm


\bibitem{ADNI} S.\,Agmon, A.\,Douglis, and L.\,Nirenberg, {\it Estimates near the
boundary for solutions of elliptic partial differential equations satisfying general
boundary conditions, I}, Comm. Pure Appl. Math., 12 (1959), 623--727.
%
\bibitem{ADNII} S.\,Agmon, A.\,Douglis, and L.\,Nirenberg, {\it Estimates near the
boundary for solutions of elliptic partial differential equations satisfying general
boundary conditions, II}, Comm. Pure Appl. Math., 17 (1964), 35--92.
%
\bibitem{ABR} S.\,Axler, P.\,Bourdon, and W.\,Ramey, {\it Harmonic Function Theory}, 2nd edition,
Graduate Texts in Mathematics, Vol.\,137, Springer-Verlag, New York, 2001.
%
\bibitem{BMMM} K.\,Brewster, D.\,Mitrea, I.\,Mitrea, and M.\,Mitrea, 
{\it Extending Sobolev functions with partially vanishing traces from locally 
$(\varepsilon,\delta)$-domains and applications to mixed boundary problems}, J. Funct. Anal., 266 (2014), 4314--4421.
%
\bibitem{DK21} H.\,Dong and S.\,Kim, {\it Green's function for nondivergence elliptic operators in two dimensions}, 
Siam J. Math. Anal., 53 (2021), no.~4, 4637--4656.
%
\bibitem{GCRF85} J.\,Garc{\'\i}a-Cuerva and J.\,Rubio de Francia, {\it Weighted Norm
Inequalities and Related Topics}, North Holland, Amsterdam, 1985.
%
\bibitem{HK07} S.\,Hofmann and S.\,Kim, {\it The Green function estimates for strongly elliptic systems of second order},
Manuscripta Math., 124 (2007), 139--172.
%
\bibitem{HoMiTa07} S.\,Hofmann, M.\,Mitrea, and M.\,Taylor, {\it Geometric and transformational 
properties of Lipschitz domains, Semmes-Kenig-Toro domains, and other classes of finite perimeter 
domains}, J. Geom. Anal., 17 (2007), no.~4, 593--647.
%
\bibitem{HoMiTa10} S.\,Hofmann, M.\,Mitrea and M.\,Taylor, {\it Singular integrals and elliptic 
boundary problems on regular Semmes-Kenig-Toro domains}, International Mathematics Research Notices, 
Oxford University Press, 2010 (14), 2567--2865.
%
\bibitem{HK20} S.\,Hwang and S.\,Kim, {\it Green's function for second order elliptic equations in non-divergence form}
Potential Analysis, 52 (2020), 27--39.
%
\bibitem{KM2012} G.\,Kresin and V.\,Maz'ya, {\it Maximum Principles and Sharp Constants for Solutions
of Elliptic and Parabolic Systems}, Mathematical Surveys and Monographs, Vol.~183, American Mathematical 
Society, Providence, 2012.
%
\bibitem{KMR2} V.A.\,Kozlov, V.G.\,Maz'ya and J.\,Rossmann, {\it Spectral Problems 
Associated with Corner Singularities of Solutions to Elliptic Equations}, AMS, 2001.
%
\bibitem{LaMi24} M.\,Laurel and M.\,Mitrea, {\it Weighted Morrey Spaces: Calder\'on-Zygmund Theory and 
Boundary Problems}, De Gruyter Studies in Mathematics Vol.\,99, 2024.
%
\bibitem{Lop} Ya.B.\,Lopatinski\u{\i}, {\it On a method of reducing boundary value problems for systems 
of differential equations of elliptic type to regular integral equations}, (Russian) Ukrain. Mat. 
\v{Z}urnal, 5 (1953), no.~2, 123--151. 
%
\bibitem{MMMMM-Fatou} J.J.\,Mar{\'\i}n, J.M.\,Martell, D.\,Mitrea, I.\,Mitrea, and M.\,Mitrea,
{\it A Fatou theorem and Poisson's integral representation formula for elliptic systems in the upper-half space}, 
pp.\,105--124 in ``Topics in Clifford Analysis'', Special Volume in Honor of Wolfgang Spr\"oßig, 
S.\,Bernstein editor, Birkh\"auser, 2019,
%
\bibitem{MMMMM} J.J.\,Mar\'{i}n, J.\,Mar{\'\i}a Martell, D.\,Mitrea, I.\,Mitrea, and M.\,Mitrea, 
{\it Singular Integrals, Quantitative Flatness, and Boundary Problems}, Progress in Mathematics, Vol.~344, Birkh\"auser, 2022.
%
\bibitem{MMMM16} J.M.\,Martell, D.\,Mitrea, I.\,Mitrea, and M.\,Mitrea, {\it The Dirichlet 
problem for elliptic systems with data in K\"othe function spaces}, Revista Mat. Iberoamericana, 
Vol.~32 (2016), 913--970.
%
\bibitem{MaMiSh} V.\,Maz'ya, M.\,Mitrea, and T.\,Shaposhnikova, {\it The Dirichlet
problem in Lipschitz domains with boundary data in Besov spaces for higher order 
elliptic systems with rough coefficients}, Journal d'Analyse Math\'ematique, 
110 (2010), no.\,1, 167--239.
%
\bibitem{DMit18} D.\,Mitrea, {\it Distributions, Partial Differential Equations,
and Harmonic Analysis}, Second edition, Springer Nature Switzerland, 2018.
%
\bibitem{GHA.I} D.\,Mitrea, I.\,Mitrea, M.\,Mitrea, 
{\it Geometric Harmonic Analysis\,-\,Volume~I: A Sharp Divergence Theorem with Nontangential Pointwise Traces}, 
Developments in Mathematics, Vol.\,72, Springer, 2022.
%
\bibitem{GHA.II} D.\,Mitrea, I.\,Mitrea, M.\,Mitrea, 
{\it Geometric Harmonic Analysis\,-\,Volume~II: Function Spaces Measuring Size and Smoothness on Rough Sets}, 
Developments in Mathematics, Vol.\,73, Springer, 2022.
%
\bibitem{GHA.III} D.\,Mitrea, I.\,Mitrea, M.\,Mitrea, 
{\it Geometric Harmonic Analysis\,-\,Volume~III: Integral Representations, Calder\'{o}n-Zygmund Theory, Fatou Theorems, 
and Applications to Scattering}, Developments in Mathematics, Vol.\,74, Springer, 2023.
%
\bibitem{GHA.IV} D.\,Mitrea, I.\,Mitrea, M.\,Mitrea, 
{\it Geometric Harmonic Analysis\,-\,Volume~IV: Boundary Layer Potentials in Uniformly Rectifiable Domains, 
and Applications to Complex Analysis}, Developments in Mathematics, Vol.\,75, Springer, 2023.
%
\bibitem{GHA.V} D.\,Mitrea, I.\,Mitrea, M.\,Mitrea, 
{\it Geometric Harmonic Analysis\,-\,Volume~V: Fredholm Theory and Finer Estimates for Integral Operators, 
with Applications to Boundary Problems}, Developments in Mathematics, Vol.\,76, Springer, 2023.
%
\bibitem{MiTak24} M.\,Mitrea and P.\,Takemura, \textit{Singular Integrals, Herz-type Function Spaces, and Boundary Problems}, book manuscript 2024.
%
\bibitem{OKB15} K.A.\,Ott, S.\,Kim, and R.M.\,Brown, {\it The Green function for the mixed problem for the linear Stokes system in domains in the plane}, 
Mathematische Nachrichten, 288 (2015), no.~4, 452--464.
%
\bibitem{Shap} Z.Ya.\,Shapiro, {\it The first boundary value problem for an elliptic system 
of differential equations}, (Russian) Mat. Sb., 28 (1951), no.~1, 55--78. 
%
\bibitem{Sol} V.A.\,Solonnikov, {\it Estimates for solutions of general boundary
value problems for elliptic systems}, Doklady Akad. Nauk. SSSR, 151 (1963), 783--785
(Russian). English translation in Soviet Math., 4 (1963), 1089--1091.
%
\bibitem{Sol1} V.A.\,Solonnikov, {\it General boundary value problems
for systems elliptic in the sense of A. Douglis and L. Nirenberg. I},
(Russian) Izv. Akad. Nauk SSSR, Ser. Mat., 28 (1964), 665--706.
%
\bibitem{Sol2} V.A.\,Solonnikov, {\it General boundary value problems for
systems elliptic in the sense of A. Douglis and L. Nirenberg. II}, (Russian)
Trudy Mat. Inst. Steklov, 92 (1966), 233--297.
%
\bibitem{St70} E.M.\,Stein, {\it Singular Integrals and Differentiability
Properties of Functions}, Princeton Mathematical Series, No.\,30,
Princeton University Press, Princeton, NJ, 1970.
%
\bibitem{Stein93} E.M.\,Stein, {\it Harmonic Analysis: Real-Variable Methods,
Orthogonality, and Oscillatory Integrals}, Princeton Mathematical Series,
Vol.\,43, Monographs in Harmonic Analysis, III, Princeton University Press,
Princeton, NJ, 1993.
%
\bibitem{StWe71} E.M.\,Stein and G.\,Weiss, {\it Introduction to Fourier Analysis
on Euclidean Spaces}, Princeton Mathematical Series, Princeton University Press, 
Princeton, NJ, 1971.
%
\bibitem{TKB12} J.L.\,Taylor, S.\,Kim, and R.M.\,Brown, {\it The Green function for elliptic systems in two dimensions}, 
Communications in Partial Differential Equations, 38(9) (2012), 1574--1600.
%
\end{thebibliography}
\end{document}